\documentclass[oneside,11pt]{amsart}

\usepackage{amsmath, amssymb, amsthm, hyperref, graphicx,   verbatim, enumerate, xcolor, tikzsymbols, stmaryrd, epsfig, xypic}
\usepackage{mathrsfs}  
\usepackage{mathabx}
 \usepackage{times}
 
 \DeclareFontFamily{U}{mathc}{}
\DeclareFontShape{U}{mathc}{m}{it}%
{<->s*[1.03] mathc10}{}

\DeclareMathAlphabet{\mathscr}{U}{mathc}{m}{it}

\usepackage[paper=a4paper]{geometry}
\usepackage[all]{xy}
\usepackage[utf8]{inputenc}

\usepackage[english]{babel}
 
\newcommand{\e}{\varepsilon}

 \newcommand{\disk}{\mathbb{D}}

\newcommand{\fr}{\partial}

\newcommand{\set}[1]{\left\{#1\right\}}
\newcommand{\norm}[1]{{\left\Vert#1\right\Vert}}
\newcommand{\abs}[1]{\left\vert#1\right\vert}

\newcommand{\rest}[1]{ \arrowvert_{#1}}

\newcommand{\unsur}[1]{\frac{1}{#1}}

\newcommand{\lrpar}[1]{\left(#1\right)}

\newcommand{\inv}{^{-1}}

\DeclareMathOperator{\supp}{Supp}

\DeclareMathOperator{\Div}{div}

\DeclareMathOperator{\Crit}{Crit}
\DeclareMathOperator{\id}{id}
\DeclareMathOperator{\jac}{Jac}

\DeclareMathOperator{\dist}{dist}

\DeclareMathOperator{\Tor}{Tor} 
\DeclareMathOperator{\Rk}{R} 
\DeclareMathOperator{\Hal}{Hal}

\def\C{\mathbf{C}}
\def\R{\mathbf{R}}
\def\Q{\mathbf{Q}}
\def\Z{\mathbf{Z}}
\def\N{\mathbf{N}}

\def\ii{{\mathsf{i}}}
\def\jj{{\mathsf{j}}}

\def\bfe{{\mathbf{e}}}

\newcommand{\Hyp}{\mathbb{H}}

\newcommand{\NS}{{\mathrm{NS}}}

\newcommand{\Sing}{{\mathrm{Sing}}}
\newcommand{\Tang}{{\mathrm{Tang}}}
\newcommand{\STang}{{\mathrm{STang}}}

\newcommand{\NT}{{\mathrm{NT}}}
\newcommand{\Lat}{{\mathrm{Lat}}}

\def\hol{{\mathrm{hol}}}
\def\Ima{{\mathsf{Im}}}
\def\Rea{{\mathsf{Re}}}

\def\P{\mathbb{P}}

\def\Aut{\mathsf{Aut}}

 \def\PGL{{\sf{PGL}}}
\def\PSL{{\sf{PSL}}}

\def\PO{{\sf{PO}}}
\def\End{{\sf{End}}}
\def\GL{{\sf{GL}}}
\def\SL{{\sf{SL}}}

\def\vol{{\sf{vol}}}

\def\Vect{{\mathrm{Vect}}}

\def\dist{{\sf{dist}}}

  \newcommand{\romain}[1]{{\color{blue}*}\marginpar{\tiny  \color{blue} RD: #1}}
  \newcommand{\serge}[1]{{\color{red}*}\marginpar{\tiny  \color{red} SC: #1}}

\theoremstyle{plain}

\newtheorem{thm}{Theorem}[section]
\newtheorem{cor}[thm]{Corollary}
\newtheorem{pro}[thm]{Proposition}
\newtheorem{lem}[thm]{Lemma}

\newtheorem{mthm}{Theorem}

\newtheorem{mcor}[mthm]{Corollary}
\newtheorem{mthmprime}{Theorem}

\theoremstyle{definition}
\newtheorem{defi}[thm]{Definition}

\newtheorem{eg}[thm]{Example}
\newtheorem{rem}[thm]{Remark}
\newtheorem{nota}[thm]{Notation}

\numberwithin{equation}{section}       % Number formulas within sections

\addtocounter{section}{0}             % Start with section 1
\numberwithin{equation}{section}       % Number formulas within sections
\begin{document}

\setlength{\parskip}{.2em}
\setlength{\baselineskip}{1.27em}

\begin{abstract}
We classify invariant probability measures for non-elementary groups of automorphisms, on any compact 
Kähler 
surface $X$, under the assumption that the group contains a so-called ``parabolic automorphism''. 
We also prove that except in certain rigid situations known as Kummer examples, 
there are only finitely many    invariant, ergodic, probability measures with a Zariski 
dense support. If $X$ is a K3 or Enriques surface, and the group does not preserve any algebraic subset,
 this leads to a complete description of orbit closures.
\end{abstract}

%
%%%%%%%%%%%%%%%%%%%%%%%%%%%%%%%%%%%%%%%%%%%%%%%%%%%%%%%%%%%%%%%%%%
%
\title[Invariant measures for automorphism groups of surfaces]
{Invariant measures for large automorphism groups of projective surfaces}
\date{\today}

\author{Serge Cantat}
\address{Serge Cantat, IRMAR, Campus de Beaulieu,
b\^atiments 22-23
263 avenue du G\'en\'eral Leclerc, CS 74205
35042  RENNES C\'edex}
\email{serge.cantat@univ-rennes1.fr}

\author{Romain Dujardin}
\address{Romain Dujardin,  Sorbonne Universit\'e, CNRS, Laboratoire de Probabilit\'es, Statistique  et Mod\'elisation  (LPSM), F-75005 Paris, France}
\email{romain.dujardin@sorbonne-universite.fr}

\thanks{ }

%\author{}
%\address{D\'epartement de math\'ematiques\\
%         Universit\'e de Rennes\\
%         Rennes\\
%         France}
%\email{serge.cantat@univ-rennes1.fr, junyi.xie@univ-rennes1.fr}
%
%%%%%%%%%%%%%%%%%%%%%%%%%%%%%%%%%%%%%%%%%%%%%%%%%%%%%%%%%%%%%%%%%%
%

\maketitle
 
\setcounter{tocdepth}{1}
\tableofcontents

%
%%%%%%%%%%%%%%%%%%%%%%%%%%%%%%%%%%%%%%%%%%%%%%%%%%%%%%%%%%%%%%%%%%
% 
%
%%%%%%%%%%%%%%%%%%%%%%%%%%%%%%%%%%%%%%%%%%%%%%%%%%%%%%%%%%%%%%%%%%
%

%
%\tableofcontents

%\renewcommand{\thefootnote}{}
%\footnotetext{\textit{2010 Mathematics Subject Classification}: 37P55, 	32H50}
%\footnotetext{\textit{Keywords}: polynomial endomorphism, orbit, valuative tree}

%%%%%%%%%%%%%%%%%%%%%%%%%%%%%%%%%%%%%%%%%%%%%%
%%%%%%%%%%%%%%%%%%%%%%%%%%%%%%%%%%%%%%%%%%%%%%
\section{Introduction}\label{par:intro}
%%%%%%%%%%%%%%%%%%%%%%%%%%%%%%%%%%%%%%%%%%%%%%
%%%%%%%%%%%%%%%%%%%%%%%%%%%%%%%%%%%%%%%%%%%%%%
%%%%%
\subsection{Non-elementary groups and parabolic automorphisms}
%%%%%
Let $\Gamma$ be a group of automorphisms of a compact K\"ahler surface $X$. We say 
that $\Gamma$ is {\bf{non-elementary}} if its image $\Gamma^*$ in  $\GL(H^2(X;\Z))$, 
given by its  action on the cohomology, contains a non-abelian free group.  
We refer to~\cite{stiffness} for a description of  such groups; in particular, 
it is shown in~\cite{stiffness} that the existence of a non-elementary subgroup in 
$\Aut(X)$ implies that $X$ is a projective surface. 
By definition, an automorphism $g\colon X\to X$ is {\bf{parabolic}}  if $\norm{(g^n)^*}$ grows quadratically with the number $n$ of iterates,
where $\norm{\cdot }$ is any operator norm on $\End(H^2(X;\C))$. Any parabolic automorphism  
preserves some genus $1$ fibration, acting by translation along the fibers.
A group $\Gamma\subset \Aut(X)$  containing a parabolic automorphism is non-elementary if and only if it contains two parabolic automorphisms 
preserving distinct fibrations~\cite{stiffness, finite_orbits}. Examples of surfaces such that $\Aut(X)$ is non-elementary and contains 
parabolic automorphisms include some rational surfaces, some K3 surfaces (notably Wehler surfaces), and  general Enriques surfaces
(see \S~\ref{sec:examples} below).
 
%%%%%
\subsection{Classification}
%%%%%
%In this article, we classify  probability measures on $X$ which are invariant by a non-elementary group \emph{containing a parabolic automorphism}. 
%In the following theorem, we say that an analytic subset $\Sigma$ in an open $U\subset X$ is 
%a {\bf{totally real}} surface if  for every smooth point $x$ of $\Sigma$, 
%the (real) tangent space $T_x\Sigma$ has (real) dimension $2$ and contains a basis of the complex tangent space $T_xX$; equivalently $T_x\Sigma$ and its image $\jj_X( T_x\Sigma )$ by the complex structure satisfy $T_x\Sigma\oplus_\R \jj_X( T_x\Sigma )=T_xX$. 
Our first theorem was already announced in \cite{stiffness, finite_orbits}:

\begin{mthm}\label{thm:main}  
Let $X$ be a compact K\"ahler surface. Let $\Gamma$ be a non-elementary subgroup of $\Aut(X)$ containing a parabolic element. Let $\mu$ be a $\Gamma$-invariant ergodic probability measure on $X$. 
Then, $\mu$ satisfies exactly one  of the following properties. 

\begin{enumerate}[\em (a)]
\item   $\mu$ is the average on a finite orbit of $\Gamma$.
\item  $\mu$ is nonatomic and  supported on a $\Gamma$-invariant curve $D\subset X$.
\item There is a $\Gamma$-invariant proper algebraic subset $Z$ of $X$, and  a $\Gamma$-invariant, 
totally real  analytic surface $\Sigma$ of 
$X\setminus Z$ such that {\em {(1)}} $\mu(\overline{\Sigma}) = 1$ and $\mu(Z)=0$; {\em {(2)}}  $\Sigma$ has finitely many irreducible components;   
{\em {(3)}} the singular locus of $\Sigma$ is locally finite;  
{\em {(4)}} $\mu$ 
is absolutely continuous with respect to the Lebesgue measure  on $\Sigma$,  and {\em {(5)}}  its density (with respect to any real analytic area form on the regular part of  $\Sigma$)
is real analytic.
 \item  There is a $\Gamma$-invariant proper algebraic subset $Z$ of $X$ such that {\em {(1)}} $\mu(Z)=0$; {\em {(2)}} the support of $\mu$ is equal to $X$; {\em {(3)}} $\mu$ is absolutely continuous with respect to the Lebesgue measure on $X$; and {\em {(4)}} the density of $\mu$ with respect to any real analytic
volume form on $X$ is real analytic on $X\setminus Z$.
\end{enumerate}
\end{mthm}

\begin{rem} %~
\begin{enumerate}[\em $(1)$]
\item 
Recall that an analytic surface  $\Sigma$ in an open $U\subset X$ is 
  {\bf{totally real}}   if  for every smooth point $x$ of $\Sigma$, 
the (real) tangent space $T_x\Sigma$  contains a basis of the complex tangent space $T_xX$; equivalently  
%$T_x\Sigma$ and its image $\jj_X( T_x\Sigma )$ by the complex structure satisfy 
$T_x\Sigma\oplus_\R \jj_X( T_x\Sigma )=T_xX$.
%In Assertion~(c),  we will simply say that the support of  $\mu$ is totally real.
%In fact\romain{de nouveau je suis convaincu que toute SV analytique totalement réelle est localement de cette forme (voir Baouendi Ebenfelt Rothschild Prop. 1.3.8.) C'était quoi le point de cette remarque déjà? }, we will prove that, locally, $\Sigma$ is defined by finitely many equations of type ${\mathsf{Re}}(u_i(x,y))=0$ where the   $u_i$ are holomorphic functions.   

\item  Each of the four cases (a), (b), (c), and (d) is  characterized by a property of the support $\supp(\mu)$: being finite, Zariski dense in a 
curve, totally real, or equal to~$X$.

\item  Given any non-elementary group $\Gamma\subset \Aut(X)$, there is a unique {\bf{maximal $\Gamma$-invariant curve}} $D_\Gamma\subset X$ (see \S~\ref{par:invariant_curves} below). 
The invariant algebraic  set  $Z$ is independent of $\mu$ and admits an explicit description (see Propositions~\ref{pro:smooth_case} and~\ref{pro:analytic_sigma}). It  
is made of components of $D_\Gamma$ together with  a residual finite invariant set.
\end{enumerate}
\end{rem}

\begin{mcor}\label{cor:no_invariant_subset}
Let $X$ be a compact K\"ahler surface. 
Let $\Gamma$ be a non-elementary subgroup of $\Aut(X)$ that contains a parabolic element
and does not preserve any proper algebraic subset of $X$.  
If $\mu$ is a $\Gamma$-invariant and ergodic probability measure on $X$, then 
$\mu$ is:
\begin{enumerate}[\em (a)]
\item  either a   measure with real-analytic density 
 on a compact, smooth, totally real,  and real analytic surface $\Sigma$ of $X$;
\item  or a   measure with real-analytic density on $X$. 
\end{enumerate}
\end{mcor}

In the totally real case~(c) of Theorem~\ref{thm:main}, it is natural to inquire about the structure of  $\overline{\Sigma}$ on the whole surface $X$, including $Z$. Under a mild geometric condition 
(AC) we are indeed able to show 
that $\overline\Sigma$ admits a semi-analytic extension across $Z$ 
(see Theorem~\ref{thm:semi-analytic} in Section~\ref{sec:semi-analytic}); this means that $\overline\Sigma$ is defined locally by finitely many analytic inequalities (see \S\ref{par:semi-analytic-vocabulary}). 
%\begin{mthmprime}\label{thm:semi-analytic}
%If in  Theorem~\ref{thm:main} we further assume the   
%non-degeneracy condition (AC), then in the totally real case~$(c)$ we can add the conclusion: (5) $\overline{\Sigma}$ is a semi-analytic subset of $X$. 
%\end{mthmprime}
Since it requires some additional concepts, the 
 condition (AC) will be described 
  only in \S~\ref{par:global_Rabpq}:   
it concerns  the action of $\Gamma$ on the singular fibers of 
  elliptic fibrations invariant by   parabolic elements of $\Gamma$. 
This condition is   satisfied in many interesting cases: it holds for instance when $D_\Gamma$ is empty,
 and can be checked on concrete examples.   
 Note that if $\overline\Sigma$ is semi-analytic, the singular locus of $\Sigma$ is finite: indeed isolated singularities cannot accumulate in this case. %\romain{je rétablis cette phrase qui me semble intéressante en fait}

In Sections~\ref{sec:not_real_parts} and~\ref{sec:boundary}, we 
 provide examples  showing  that the geometric conclusions of 
  Theorems~\ref{thm:main} and~\ref{thm:semi-analytic} are, in a sense, optimal. 
More precisely it is shown that  in case~(c), 
\begin{enumerate}[--]
\item it may happen that $\Sigma$ is \textbf{not} contained in the real part of $X$, for any real structure 
on $X$; in other words,  there is no anti-holomorphic involution $\sigma\colon X\to X$ for which $\Sigma$ would be contained in the  the fixed point set ${\mathrm{Fix}}(\sigma)$ (see Corollary~\ref{cor:not_real_parts});
\item  $\Sigma$ can have a non-empty boundary (see \S\ref{par:examples_deformation}). 
\end{enumerate}

%%%%%%
 \subsection{Finitely many  invariant measures} 
 %%%%%%

Theorem \ref{thm:main} is a key   ingredient  of the finiteness of the number of  
periodic orbits obtained in~\cite{finite_orbits}. %\romain{dans l'esprit de \cite[Thm A]{finite_orbits} il serait intéressant de montrer que génériquement pour Wehler il n'y a pas de mesure totalement réelle}
It also leads to the following alternative which is reminiscent of, but independent from, 
\cite[Thm B]{finite_orbits}.

\begin{mthm}\label{thm:finiteness}
Let $X$ be a compact K\"ahler surface. If $\Gamma$ is a non-elementary subgroup of $\Aut(X)$  
containing a parabolic element, then there are  only finitely many ergodic $\Gamma$-invariant 
probability measures  giving no mass to proper Zariski closed subsets, unless $(X, \Gamma)$ is a Kummer group.
\end{mthm}

 We refer to~\cite{finite_orbits} for the definition of Kummer groups; roughly speaking  it means that the action of $\Gamma$ on $X$ comes from the action of a group of automorphisms on some torus $\C^2/\Lambda$. 
This result will be established  in Section~\ref{sec:finiteness}. We also show in \S~\ref{subs:infinitely_many_measures} that a Kummer group can indeed
 admit infinitely many ergodic    invariant measures of totally real type. 
Together with \cite[Thm. C]{finite_orbits} (finitely many finite orbits)  we thus obtain:

\begin{mcor}
 Let $X$ be a compact Kähler surface which is not a torus. Let $\Gamma$
  be a subgroup of $\Aut(X)$ which contains a parabolic element and does not preserve any algebraic curve.  Then there are at most finitely many  $\Gamma$-invariant,  ergodic probability measures on $X$. 
\end{mcor}

% The following corollary provides an  interesting connection with  some recent  (in)finite\-ness results for 
%real structures on  complex surfaces (see e.g.~\cite[Appendix D]{Degtyarev-Itenberg-Kharlamov:LNM} and~\cite{Dinh-Oguiso:duke}). 
%
%\begin{mcor}\label{cor:finiteness_geometric}
%Let $X$ be a K3 or Enriques surface, and     $\Gamma$
%  be a subgroup of $\Aut(X)$, such that $(X, \Gamma)$ is not a Kummer group. 
%  Then there are only finitely many $\Gamma$-invariant closed totally real surfaces. 
%\end{mcor}
%
%The proof is given in  Section~\ref{sec:finiteness}.
%
%%%%%%
 \subsection{Orbit closures} 
 %%%%%%
 
Although the main focus is on invariant measures, our methods can also
be used to describe the topological structure of orbits. In conjunction with the results of~\cite{stiffness}, 
this   leads  to the following  neat statement, which relies crucially on the finiteness Theorem~\ref{thm:finiteness}.

\begin{mthm}\label{mthm:topological}
Let $X$ be a K3 or Enriques surface. Let $\Gamma\leq \Aut(X)$ be a non-elementary subgroup
 which   contains    parabolic elements 
 and does not preserve any non-empty proper algebraic subset.
  Then there exists a real analytic and  totally real  surface $\Sigma \subset X$ such that:
  \begin{enumerate}[{\em (a)}]
 \item  if $x\notin \Sigma$ then 
 $\overline{\Gamma(x)} = X$, and 
 \item  if $x\in \Sigma$ then 
 $\overline{\Gamma(x)}$ is a union of connected components of $\Sigma$. 
 \end{enumerate}
\end{mthm}

 This applies for instance to very general Wehler (2,2,2)-surfaces (see Example~\ref{eg:wehler}). 
The assumption that $\Gamma$ does not admit any invariant  algebraic (in particular finite) 
subset is essential to our strategy;  we will  come back 
to this issue  in a forthcoming work~\cite{hyperbolic}.

%%%%%%%
% \subsection{Notes} 
% %%%%%%
%A  weak form of Theorem~\ref{thm:main} is proven in \cite{cantat_groupes} 
%in the special case of K3 surfaces\footnote{Also, one of the statements in~\cite{cantat_groupes} is slightly erroneous. In case (c) of Theorem~\ref{thm:main}, one first shows that $\mu$ gives mass to some {\emph{germs of real analytic surfaces}}, and one has to glue these germs together to construct the surface $\Sigma$; to do so there is a monodromy problem which is overlooked in \cite{cantat_groupes}. To say it 
%differently,    the results of \cite{cantat_groupes} only imply that  
% $\Sigma$ is the analytic continuation of such a germ, which could a priori  be dense in $X$. Overcoming this problem occupies a significant part of the present paper. }. The results of \cite{cantat_groupes} 
%do not describe the support of the measure or the smoothness of its density, 
%and are not sufficient to derive the global structure of $\Sigma$ given by Corollary~\ref{cor:no_invariant_subset}, nor the finiteness result of Theorem~\ref{thm:finiteness}.

%%%%%%
 \subsection{Comments on the proofs} 
 %%%%%%

A  weak form of Theorem~\ref{thm:main} was proven in~\cite{cantat_groupes} 
in the special case of K3 surfaces. The results of \cite{cantat_groupes} 
do not describe the support of the measure nor the smoothness of its density, 
and are not sufficient to derive the global structure of $\Sigma$ given by Corollary~\ref{cor:no_invariant_subset}, nor the finiteness result of Theorem~\ref{thm:finiteness}.

The main idea to prove Theorem~\ref{thm:main} is quite natural: if $\mu$ is an ergodic 
 invariant measure which gives no mass to  algebraic subsets, we can find a $\mu$-generic point $x$ where 
 two parabolic elements act as translations in two different directions and it follows that locally $\mu$ is either absolutely continuous (case (d)) or supported  by a germ of totally real analytic surface 
 $\Sigma$ at $x$ (case (c)). 
We can then propagate this picture using the dynamics of the group, or the analytic continuation of the local germ $\Sigma$. 
Then it could be that in case (c), the propagation of  $\Sigma$ 
is an  immersed surface which may be dense in some 
exceptional minimal set, or even dense in $X$; this difficulty was overlooked in~\cite{cantat_groupes}. 
Overcoming this issue occupies a significant part of the paper (see Section~\ref{sec:Main_Proof} for details, 
and Section~\ref{sec:closed_invariant} for a related argument in the topological category).

  Theorems~\ref{thm:semi-analytic} on the semi-analyticity of $\overline\Sigma$ and Theorem~\ref{thm:finiteness} are largely intertwined. Both rely on a careful analysis
 of the dynamics of parabolic automorphisms  near the singular fibers of their invariant  fibrations; this requires to dive into Kodaira's description of genus $1$ fibrations. 
  
%%%%%%
\subsection{Structure of the paper} 
%%%%%%

  In Section~\ref{sec:examples} we start by briefly describing 
 a few basic examples which are useful to be kept in mind;   more advanced examples are
  given in  Sections~\ref{sec:not_real_parts} and~\ref{sec:boundary}. In Section~\ref{sec:halphen} we collect 
some  preliminary results on genus~$1$ fibrations, their singular fibers, and the dynamics of ``Halphen twists'' 
preserving such fibrations. Since such automorphisms 
appear in many different instances, we hope that will   prove useful beyond this paper (see also Duistermaat's monograph~\cite{duistermaat} for a thorough treatment with a   different focus). 
The core of the paper extends from Sections~\ref{par:Main Proof1} to~\ref{sec:finiteness}.  A basic dichotomy  
  is whether $X$ is birationally equivalent to a torus, or not. 
 The torus  case relies on  elementary tools from homogeneous dynamics, and the details are given in  Appendix \ref{app:abelian}.   The proof of Theorem~\ref{thm:main} occupies Sections~\ref{par:Main Proof1} and~\ref{sec:Main_Proof}.  
% Theorems~\ref{thm:semi-analytic} on the semi-analyticity of $\overline\Sigma$ and Theorem~\ref{thm:finiteness} are largely intertwined, and  rest on a careful analysis
% of the action of the parabolic elements of $\Gamma$
%  near the singular fibers of the associated elliptic fibrations. 
%  The details are   given
  Theorems~\ref{thm:semi-analytic}   and Theorem~\ref{thm:finiteness} are established 
  in Sections~\ref{sec:semi-analytic} and~\ref{sec:finiteness}. 
  In Section~\ref{sec:closed_invariant} we study the topological 
  classification  of orbits and prove in particular Theorem~\ref{thm:orbit_closures}, which contains  Theorem~\ref{mthm:topological}. 

%%%%%%
\subsection{Acknowledgment} We are grateful to Yves de Cornulier for useful discussions which led to Theorem~\ref{thm:chabauty}, after we had obtained Corollary~\ref{cor:closed_epsilon_sparse}.

%%%%%%%%%%%%%%%%%%%%%%%%%%%%%%%%%%%%%%%%%%%%%%
%%%%%%%%%%%%%%%%%%%%%%%%%%%%%%%%%%%%%%%%%%%%%%
\section{Two examples}\label{sec:examples}
%%%%%%%%%%%%%%%%%%%%%%%%%%%%%%%%%%%%%%%%%%%%%%
%%%%%%%%%%%%%%%%%%%%%%%%%%%%%%%%%%%%%%%%%%%%%%

%Before starting the proof of our results, let us describe a few examples which are useful to keep in mind. 
%\romain{répétition}\romain{on pourrait déplacer cette section après la section 3, et développer un peu l'analyse des fibres singulières dans le cas Coble  (et raccourcir Wehler) pour avoir une discussion  un peu nouvelle par rapport à  \cite{stiffness}} 

\subsection{K3 surfaces}\label{par:examples_K3surfaces} Let $X$ be any K3 surface. 
There is a holomorphic $2$-form $\Omega_X$ on $X$ that does not vanish and 
satisfies $\int_X\Omega_X\wedge {\overline{\Omega_X}}=1$; this form is unique up to multiplication by 
a complex number of modulus $1$. Thus, the volume  form $\vol_X:=\Omega_X\wedge  {\overline{\Omega_X}}$ is 
$\Aut(X)$-invariant. If $X$ comes with a real structure for which $X(\R)$ is non-empty, then $X(\R)$ is orientable 
and some multiple of $\Omega$ restricts to a positive area form on $X(\R)$ (see~\cite[\S VIII.4]{Silhol:LNM} and \S 1 of  \cite{Degtyarev-Kharlamov:1997}).
This area form is multiplied by 
$\pm 1$ by elements of $\Aut(X_\R)$; in particular, the measure induced by this form is invariant (see also Remark~\ref{rmk:real_area_form} below). We refer to~\cite{Degtyarev-Itenberg-Kharlamov:LNM, Degtyarev-Kharlamov:2000} for the topology of $X(\R)$: it can be a sphere,  the union of a sphere and a surface
of genus $2$,   a torus, etc. Here are two explicit examples. 
%\serge{J'ai raccourci l'exemple suivant, enlevé des formules explicites, et ajouté un blabla ci-dessus qui parle des surfaces K3 de manière générale}

\begin{eg}\label{eg:wehler}(See \cite{stiffness, Cantat-Oguiso}).-- Take three copies of 
$\P^1$, with respective coordinates $z_i=[x_i:y_i]$, $i=1,2,3$. Let $X\subset \P^1\times \P^1\times \P^1$ be a Wehler surface, {i.e.} a smooth surface of degree $(2,2,2)$; 
assume that {\em{$X$ is very general in 
the family of such surfaces}}. In particular, $X$ is smooth and it is a K3 surface.
Fix an index $k\in \{1,2,3\}$, let $i<j$ be the two indices such that $\{i,j,k\}=\{1,2,3\}$, and let  
$\pi_{ij}\colon X\to \P^1\times \P^1$ be the projection that forgets the $k$-th coordinate; 
this projection is a $2$-to-$1$ cover, and we denote by $\sigma_k$ the involution 
that permutes the points in the fibers of $\pi_{ij}$. 
Then,  $\Aut(X)$ is generated by the three involutions $\sigma_k$ and 
is non-elementary (see~\cite{Cantat:Acta, stiffness, Wang:1995}); the composition 
$\sigma_i\circ \sigma_j$ is a parabolic automorphism  preserving the genus $1$ fibration $\pi_k(z_1,z_2,z_3)=z_k$; and the composition $\sigma_1\circ \sigma_2\circ \sigma_3$ is a loxodromic automorphism  --with topological entropy $\log(9+4\sqrt{5})>0$. It is shown in 
\cite[Thm A]{finite_orbits} that for a very general $X$, there is no $\Aut(X)$-invariant 
proper algebraic subset.  If
$X$ is defined by a polynomial equation with real coefficients, then $\Aut(X)$ preserves the real structure 
$X_\R$ because the three involutions do. In particular, the real part $X(\R)$ is $\Aut(X)$-invariant. 

For future reference let us note that the canonical invariant 2-form admits  a simple explicit expression: 
consider affine coordinates $x_i\in \C$ corresponding to  each of the three $\P^1$ factors (with $z_i=[x_i:1]$); then, $X$ is defined by a polynomial equation $P(x_1,x_2,x_3)=0$, and at every point in $X$, one of the partial derivatives of $P$ does not vanish because $X$ is smooth; then,
  up to some constant factor,  
 \begin{equation}
 \Omega_X=\frac{dx_1\wedge dx_2}{\partial_{x_3} P}=\frac{dx_2\wedge dx_3}{\partial_{x_1} P}=\frac{dx_3\wedge dx_1}{\partial_{x_2} P}.
 \end{equation}
\end{eg}

%To get explicit formulas for $\Omega_X$\romain{c'est quoi le point de ce passage?}, consider the affine coordinates $x_i\in \C$ on each of the three $\P^1$s (with $z_i=[x_i:1]$); then, $X$ is defined by a polynomial equation $P(x_1,x_2,x_3)=0$ and, up to some constant factor,  
%\begin{equation}
%\Omega_X=\frac{dx_1\wedge dx_2}{\partial_{x_3} P}=\frac{dx_2\wedge dx_3}{\partial_{x_1} P}=\frac{dx_3\wedge dx_1}{\partial_{x_2} P};
%\end{equation}
%each point of $X$ has a neighbourhood on which one of these derivatives does not vanish, 
%because $X$ is smooth, and the formula defines a global holomorphic form on 
%$X\subset (\P^1)^3$ that does not vanish because $\deg(P)=(2,2,2)$.
%When $P$ has real coefficients, then $\Omega_X$ is defined over $\R$ and its real part induces an area form on $X(\R)$;
%in particular, $X(\R)$ is orientable (as any K3 surface, see~\cite[\S VIII.4]{Silhol:LNM} and \S 1 of  \cite{Degtyarev-Kharlamov:1997}). This area form is multiplied by $\pm 1$ by elements of $\Aut(X_\R)$.  We refer to~\cite{Degtyarev-Itenberg-Kharlamov:LNM, Degtyarev-Kharlamov:2000} for the topology of $X(\R)$: it can be a sphere,  the union of a sphere and a surfaceof genus $2$,   a torus, etc. 

\begin{eg}(See \cite[\S 3.2]{stiffness}).--  
Fix five lengths $(\ell_0,\ell_2,\ldots, \ell_4)\in (\R_+^*)^5$ such that there is at least one pentagon ${\mathsf{P}}=(a_0, \ldots, a_4)$
in $\R^2$, the sides of which satisfy $(a_i,a_{i+1})=\ell_i$ (for $i$ taken modulo $5$); here, by a {\bf{pentagon}}, we just mean an ordered 
set of five points $a_i$ in $\R^2$.  Assume that the family 
of such pentagons does not contain any flat pentagon (for instance this imposes $\ell_0+\ell_1\neq \ell_2+\ell_3+\ell_4$). Consider the set of all such pentagons modulo affine positive isometries of $\R^2$; thus, each pentagon can now be put in a 
normal position, with $a_0=(0,0)$ and $a_1=(\ell_0,0)$. This set can be identified with a real algebraic surface $X(\R)$
that depends on $(\ell_0, \ldots, \ell_4)$. 
There are five natural involutions acting algebraically on this surface: 
given one of the vertices $a_i$ of a pentagon ${\mathsf{P}}\in X(\R)$, consider the two circles with centers $a_{i-1}$ and $a_{i+1}$
and respective radii $\ell_{i-1}$ and $\ell_i$, where indices are taken modulo $5$; these circles intersect in two points $a_i$ and $a_i'$; thus we get an involution 
$\sigma_i$ of $X(\R)$, mapping ${\mathsf{P}}$ to the pentagon $\sigma_i({\mathsf{P}})$ with the same vertices except for $a_i$ that is 
replaced by~$a_i'$. 
Our hypotheses imply that $X(\R)$ is the real part of some real K3 surface $X_\R$ and $\sigma_i\in \Aut(X_\R)$.
Again, the composition $\sigma_i \circ \sigma_{i+1}$ is a parabolic automorphism of $X$ when the lengths are chosen generically. 
\end{eg}

Similar examples of large groups of automorphisms preserving a volume form on 
$X$ (resp. a smooth measure on $X(\R)$)  
can be constructed on some abelian surfaces and on most Enriques surfaces (see~\cite{stiffness} for instance).  

\begin{rem}\label{rmk:real_area_form} 
If $\Sigma$ is a totally real surface (of class $C^1$, say) in an abelian or 
 K3 surface $X$, which is invariant under a group 
$\Gamma\subset \Aut(X)$, then the canonical 2-form $\Omega_X$ induces a $\Gamma$-invariant 
measure on $\Sigma$. Indeed for every $x\in X$, the tangent space $T_xX$ contains two $\C$ linearly independent vectors, thus $\Omega_X\rest{\Sigma}$ induces a complex valued 
2-form on $\Sigma$ that does not vanish. Thus locally we can define an area form 
$\Omega_\Sigma$ by $\Omega_{\Sigma, x}  = \xi(x) \Omega_{X, x}$ 
where $\xi$ is a function with values in the unit circle, and whenever $\Sigma$ is orientable or not, this 
 induces (by   taking the associated density $\abs{\Omega_\Sigma}$ in the non-orientable case) 
the desired measure on $\Sigma$. 
If $X$ is an Enriques surface, the universal cover $q\colon X'\to X$ is an \'etale  $2$-to-$1$ cover by a K3 surface. 
If $\Sigma\subset X$ is a totally real surface, then its pre-image $\Sigma'=q^{-1}(\Sigma)$ is also totally real, and 
the automorphism of the covering $q$ is an element of $\Aut(X';\Sigma')$. Thus, 
applying the above construction in $X'$ and pushing forward to $X$, we get an invariant measure 
on $\Sigma$ as well. 
Finally, if $X$ is a blow-up of an abelian, K3, or Enriques surface, 
the same construction applies, except that  the density of the associated volume form 
may vanish along the exceptional divisor of the blow-up.
\end{rem}

\subsection{Rational surfaces}\label{par:examples_rational} The
 family of Coble surfaces (see~\cite{Cantat-Dolgachev, stiffness}) and the examples described by Blanc in~\cite{Blanc:Michigan}  
 give rational surfaces $X$ such that $\Aut(X)$ is  non-elementary and contains parabolic elements (see~\cite{stiffness}, {\S}3.4, and Example~\ref{eg:condition_(AC)} below). They are constructing
 by blowing up a finite number of points in $\P^2$: the $10$ double points of some rational sextic $S$ for Coble surfaces; a finite number of points
 on a cubic curve $C$ for Blanc surfaces. The strict transform $S'$ and $C'$ of these curves are preserved by $\Aut(X)$. Denote by $K_X$ the canonical bundle of the surface. 
 
\begin{enumerate}
\item In the Coble case, there is a meromorphic section $\Omega$ of $K_X^{\otimes 2}$ that does not vanish and has a simple pole along $S'$.

\item In Blanc's example, there is
a meromorphic section $\Omega$ of $K_X$  that does not vanish and has a simple pole along $C'$.
\end{enumerate} 
In both cases, $\Omega$ induces a natural measure on $X$: for Blanc surfaces, it is given by the form $\Omega\wedge \overline{\Omega}$; for Coble surfaces, it is given by $\Omega^{1/2}\wedge \overline{\Omega}^{1/2}$.
In the Coble case, the total mass of this measure is finite, while in Blanc's example it is infinite. Moreover, 
if $\Gamma\subset \Aut(X)$ is any subgroup generated by parabolic elements, then $\Gamma$ preserves this measure.

\subsection{Subgroups} In each of the previous examples, one can replace $\Aut(X)$ by a {\bf{thin}} subgroup, {i.e.} 
a subgroup $\Gamma\subset \Aut(X)$ of infinite index  but with 
\begin{equation}
{\mathrm{Zar}}(\Aut(X)^*)/{\mathrm{Zar}}(\Gamma^*)<+\infty;
\end{equation}
here, $\mathrm{Zar}(\Gamma^*)$ is the Zariski closure in $\GL(H^2(X;\R))$. For instance, pick finitely many parabolic automorphisms $g_i$  in $\Aut(X)$, 
and consider the group $\Gamma$ generated by high powers $g_i^m$  of the $g_i$. 
If the $g_i$ do not preserve the same fibrations (see below), $\Gamma$ is non-elementary; and if one chooses the $g_i$ correctly, $\Gamma$ is thin. In case of Wehler surfaces, it suffices to 
take $g_1= \sigma_2\circ\sigma_3$ and $g_2=\sigma_3\circ \sigma_1$ and $m=3$. 

%%%%%%%%%%%%%%%%%%%%%%%%%%%%%%%%%%%%%%%%%%%%%%
%%%%%%%%%%%%%%%%%%%%%%%%%%%%%%%%%%%%%%%%%%%%%%
\section{The dynamics of Halphen twists}\label{sec:halphen}
%%%%%%%%%%%%%%%%%%%%%%%%%%%%%%%%%%%%%%%%%%%%%%
%%%%%%%%%%%%%%%%%%%%%%%%%%%%%%%%%%%%%%%%%%%%%%

%\romain{ce serait pas mal dans cette section de distinguer les paragraphes ou $f$ joue un rôle des autres notamment dans le \S \ref{par:singular-fibers} }\serge{Oui}

%\serge{A priori, on peut maintenant figer cette section 3 et ne plus y revenir ... (version du 23 Août 2021 :-))}

\subsection{Parabolic automorphisms and Halphen twists}\label{par:Halphen-intro} 
An automorphism $f$ of a compact K\"ahler surface $X$  is said 
\textbf{parabolic} if 
\begin{equation}\label{eq:parabolic_definition}
\int_X (f^n)^*\kappa\wedge \kappa \asymp n^2
\end{equation}
for some (hence  any) K\"ahler form $\kappa$ on $X$; equivalently, some power $(f^m)^*$ of $f^*\in \GL(H^2(X;\Z))$ is unipotent and the 
maximal size of its Jordan blocks is equal to $3$; equivalently, $f^*$ acts on $H^{1,1}(X;\R)$ 
as a parabolic isometry with respect to the intersection form given by the cup  product (see~\cite{Cantat:Milnor}). 
%\begin{rem}[See~\cite{Cantat:Milnor, Cantat:SLC}]
%The notion of Halphen twist extends to bimeromorphic transformations, and to birational 
%transformations of projective surfaces over an arbitrary field $\bfk$; in this latter case, we 
%assume   $\bfk$  algebraically closed for simplicity. Then, any Halphen twist $f\in \Bir(X)$ is conjugated, by some
%bimeromorphic (resp. birational) transformation $\varphi\colon Y\dasharrow X$ to an automorphism
%$f_Y$ of a new compact K\"ahler (resp. smooth projective) surface $Y$. This is the reason why 
%we stay with the most restrictive definition.\serge{Je pense qu'on peut enlever cette remarque}
%\end{rem}

\begin{thm}\label{thm:invariant_fibration}
Let $f\colon X\to X$ be a parabolic automorphism of a compact K\"ahler surface.  
\begin{enumerate}[\em (1)]
\item there exists  a genus $1$ fibration $\pi\colon X\to B$ and an automorphism $f_B$ of the
Riemann surface $B$ such that $\pi \circ f = f_B\circ \pi$;
\item if $E$ is any (scheme theoretic) fiber of $\pi$, and $F$ is a member of the linear system $\vert E\vert$, then $F$ is a fiber of $\pi$;
\item the foliation determined by the fibration~$\pi$ is the unique $f$-invariant complex analytic (smooth or singular) foliation on $X$;
\item $f_B$ has finite order,  unless  $X$ is a compact torus. 
\end{enumerate}
\end{thm}

The existence of the invariant fibration is proven in~\cite{Gizatullin:1980} when $X$ is a rational surface, 
and is easily obtained for other types of surfaces by the Riemann-Roch theorem (see~\cite[Proposition 1.4]{Cantat:Acta}, and \cite{Cantat:Milnor} 
for a survey). The second assertion is not specific to invariant fibrations, but is good to be kept in mind; 
phrased differently, it says  that  the space of sections $H^0(X;{\mathcal{O}}(E))$  is  $2$-dimensional, and 
 the divisor of zeroes $\Div(s)_0$ of any section $s\in H^0(X;{\mathcal{O}}(E))$ is a fiber of $\pi$.
The uniqueness of the invariant foliation and the last assertion are proven in~\cite{Cantat-Favre} (see   Remark~\ref{rem:invariant_foliations} below for  the  Kähler case).

When $f_B   = \id_B$ we say that $f$ is a \textbf{Halphen twist} (this terminology may differ from other references).
From the last item of the theorem, we see that 
 if $X$ is not a torus, for every parabolic $f$  there exists $k\geq 1$ such that 
 $f^k$ is a Halphen twist.

\begin{rem}\label{rem:theorem_invariant_fibration}  Consider a fiber $E$ of $\pi$, as in Assertion (2) of Theorem~\ref{thm:invariant_fibration}.
Its class $[E]\in H^{1,1}(X;\R)$ generates a ray $\R_+[E]$ with the following property: if $D$ is an effective divisor with $[D]\in \R_+[E]$, then 
$D$ is mapped to a point by $\pi$. Thus, $\R_+[E]$ characterizes $\pi$ as the unique fibration contracting these curves. In particular, an automorphism $h$
of $X$ permutes the fibers of $\pi$ if and only if $h^*$ preserves $\R_+[E]\subset H^{1,1}(X;\R)$.
\end{rem}

\subsection{Complex and real analytic structures on fibers}\label{par:smooth-fibers-Betti} Let $\pi\colon X\to B$ be a genus $1$ fibration on a 
compact Kähler surface $X$ ( in this section, we do not assume that $\pi$ is invariant by a parabolic automorphism). Our goal is to describe a foliation 
which is associated to $\pi$ and the choice of a section of $\pi$ (see~\cite[\S 2.1]{cantat-gao-habegger-xie} for further details). 

Denote by $\Crit(\pi)\subset B$ %\romain{VCrit?}\serge{A voir} 
the set of critical values of $\pi$, and 
by $B^\circ$ the complement of this finite set. On $\pi^{-1}(B^\circ)$, $\pi$ is a proper submersion 
and each fiber $X_w:=\pi^{-1}(w)$ 
  is a curve of genus~$1$. 
Let $U\subset B^\circ$ be an open subset, endowed  with 
\begin{itemize}
\item a holomorphic section $\sigma\colon U \to X$ of $\pi$ above $U$,
\item a continuous choice of basis of $H_1(X_w;\Z)$, for $w\in U$.
\end{itemize} 
If  $\sigma(w)$ is taken as the neutral element of $X_w$, then  $X_w$ becomes 
an elliptic curve for each $w\in U$ and there is a unique holomorphic function $\tau$ from $ U$ to the 
upper half-plane $\Hyp_+\subset \C$  such that 
\begin{itemize}
\item for every $w\in U$, $X_w=\C/\Lat(w)$ where  
\begin{equation}\label{eq:lattice_L(w)}
 \Lat(w)=\Z\oplus \Z\tau(w)\simeq H_1(X_w;\Z) 
\end{equation}
\item  the basis $(1, \tau(w))$ of $\Lat(w)$ corresponds to the chosen basis of $H_1(X_w;\Z)$. 
\end{itemize}
Indeed, $X_U:=\pi^{-1}(U)$ is holomorphically equivalent to the quotient of
$U\times \C$ by the action of $\Z^2$    defined  by 
$(p,q)\cdot(w,z)=(w,z+p+q\tau(w))$ for $(p,q)\in \Z^2$  and  $(w,z)\in U\times \C$.

In the real-analytic category all one-dimensional complex tori are equivalent to 
$\R^2/\Z^2$ as real Lie groups. Concretely, there is a unique isomorphism $\Psi_w\colon X_w\to \R^2/\Z^2$
which maps the basis $(1,\tau(w))$ of $\Lat(w)$ to the canonical basis $((1,0),(0,1))$ of $\Z^2$; in coordinates, if 
$\tau(w)=\tau_1(w)+\ii \tau_2(w)$ and $z=x+\ii y$,   then
\begin{equation}\label{eq:psi_explicit_formula}
\Psi_w(z)=\left( x - \frac{\tau_1(w)}{\tau_2(w)}y, \; \frac{1}{\tau_2(w)} y\right).
\end{equation}
The real analytic diffeomorphism $\Psi\colon \pi^{-1}(U)\to U\times \R^2/\Z^2$ defined by 
\begin{equation}\label{eq:Psi}
\Psi(w,z)=(w,\Psi_w(z))
\end{equation}
is the unique homeomorphism such that (1) $\pi_1\circ \Psi=\pi$, where $\pi_1$ is the first projection; (2) $\Psi$ maps the basis of $H_1(X_w;\Z)$
to the canonical basis of $\Z^2$; and (3) $\Psi$ is an isomorphism of Lie groups in each fiber. In particular, $\Psi(w,\sigma(w))=(w,(0,0))$. 

In the following remarks,  $\pi_2\colon U\times \R^2/\Z^2\to \R^2/\Z^2$  denotes the second projection.

\begin{rem}\label{rem:Betti_foliation}
For any $(a,b)\in \R^2/\Z^2$, the holomorphic map 
$w\in U\mapsto a + b\tau(w) \in \C/\Lat(w)
$ determines a local section of $\pi$ above $U$;
this section coincides with 
\begin{equation}
\Psi^{-1}\{ (w,(x,y))\; ; \; (x,y)=(a,b)\}.
\end{equation}
So, if we consider  the real analytic foliation of $U\times \R^2/\Z^2$ whose leaves are the fibers of  $\pi_2$, and if we pull-back this foliation by $\Psi$, we get a real analytic foliation 
${\mathcal{F}}_U$ of $X_U$ with holomorphic leaves, which will be referred to as the (local) {\bf{Betti foliation}}. 
Now, consider the holomorphic map $[k]_U\colon \pi^{-1}(U)\to \pi^{-1}(U)$ given by 
multiplication by some integer $k\geq 2$ along the fibers of $\pi$; by definition, it fixes $\sigma(U)$ pointwise. 
Then, ${\mathcal{F}}_U$ is invariant under the action of $[k]_U$. A leaf is pre-periodic 
if and only if it contains a torsion point of $X_w$, for some and then any $w\in U$.
The union of these preperiodic leaves is dense and each leaf of $\mathcal F_U$ is a limit of such leaves. 
%\serge{J'enlève deux lignes sur les courbes prépériodiques. Du coup, je change un peu la remarque suivante (qui sautera sans doute dans vcourte).}\romain{oui mais attention la définition du feuilletage de Betti est dedans}
%If $m\geq 0$ and $\ell\geq 1$, the equation $[2]_U^{m+\ell} (x)=[2]^{m}_U(x)$ determines a holomorphic curve in $X_U$: this is the curve of  
%pre-periodic points $x$  for which $[2]^{m}_U(x)$ is periodic of period dividing~$\ell$.
\end{rem}

\begin{rem}\label{rem:algebraic_leaves_of_FB} %\romain{déplacé définition du feuilletage de Betti dans la remarque précédente}(see~\cite{cantat-gao-habegger-xie}).
Suppose that $X$ is projective and $\sigma$ is the restriction of an algebraic 
multisection of degree $\ell$. This means that 
there is an irreducible curve $C$ in $X$ intersecting the general fiber of $\pi$ in $\ell$ points such that the graph of $\sigma$ is 
contained in $C$. Set $k=\ell+1$. Then the multiplication map $[k]_U$
extends as a rational transformation $[k]_B\colon X\dasharrow X$. Indeed, if $X_w$ is a general fiber
and $x$ is a point of $X_w$, there is a unique point $y$ such that $(\ell+1)x -y$ is linearly equivalent to the divisor $C\cap X_w$: by definition, $[k]_B(x)=y$.  %(intersected with~$X_U$).  

 If $C$ is a section, {i.e.}\ $\ell=1$,  ${\mathcal{F}}_U$ extends globally to a foliation ${\mathcal{F}}$ of 
$\pi^{-1}(B^\circ)$; we shall also refer to ${\mathcal{F}}$ as the (global) Betti foliation. The leaves of ${\mathcal{F}}$ corresponding to torsion points are, in fact, 
algebraic curves in $X$, since they correspond to the curves defined by $[2]_B^{m+q} (x)=[2]^{m}_B(x)$ for some $m\geq 0$ and $q\geq 1$.
On the other hand, the local projections $\pi_2\circ \Psi\colon \pi^{-1}(U)\to 
\R^2/\Z^2$ are not canonically defined; if $\gamma$ is a loop in $B^\circ$, with base point $w_0\in U$, then the analytic continuation 
of $\pi_2\circ \Psi$ along the loop is $M(\gamma)\circ \pi_2\circ \Psi$, where $M(\gamma)\in \mathsf{SL}_2(\Z)$ is given by the monodromy of the fibration (the determinant of $M(\gamma)$ is $1$ because the orientation of the fibers is preserved).   In other words, the monodromy of the fibration is induced by the holonomy of the Betti foliation.
The section $\sigma$ provides a fixed point $\sigma(w_0)$ 
of the holonomy;   the curves of pre-periodic points of $[2]_B$ correspond simultaneously to finite orbits of the holonomy group and to torsion points 
of the fiber~$X_{w_0}$.

If  $C$ is a multisection of degree $\ell\geq 2$, the Betti foliation ${\mathcal{F}}_U$  does not extend to $B^\circ$:  instead we obtain a web of degree at most $\ell$, which is 
 locally the superposition of the local Betti  foliations  associated to 
  the $\ell$ choices of local sections whose graphs are contained in $C$. 
%  As soon as a section of $\pi$ above some open set $V$ of $B$ is given, one can speak of the foliation 
%above $V\cap B^\circ$, and we shall refer to it as ``{\emph{the Betti foliation}}''. 
\end{rem} 

\begin{rem} (see~\cite{cantat-gao-habegger-xie}).
The form $\pi_2^*(dx\wedge dy)$
 is  smooth and closed. Its pull back to $X_U$ is the local {\bf{Betti form}} %\footnote{If one wants to specify that $\omega_{\rm B}$ is defined above $U$, we shall denote it by $\omega_{{\rm B},U}$.} 
 $\omega_{{\rm B} }=\Psi^*\pi_2^*(dx\wedge dy)$:  
(1) $\omega_{{\rm B}}$    is a closed semi-positive $(1,1)$-form; 
(2) it vanishes along the leaves of ${\mathcal{F}}_U$ (its kernel is $T{\mathcal{F}}_U$); 
and (3) for $w\in B^\circ$, $\omega_{{\rm B} \vert X_w}$ is the unique translation invariant form of type $(1,1)$ such that $\int_{X_w} \omega_{{\rm B},}=1$. These properties characterize $\omega_{{\rm B},}$. If there is a global section, %as in the previous remark, 
these forms patch together to define a global real analytic Betti form $\omega_{\rm B}$ on $X_{B^\circ}$ (the monodromy group is contained in $\SL_2(\Z)$, so it preserves $dx\wedge dy$). 
\end{rem}

Note that the Betti form and Betti foliation depend on the choice of a section, but not on the choice of a basis of $H_1(X_{w_0};\Z)$.

\subsection{Singular fibers}\label{par:singular-fibers} 
Our goal in this section is to collect some facts concerning the 
geometry of a genus $1$ fibration $\pi\colon X\to B$ around one of its singular fibers. 
Furthermore, if $f$ is a Halphen twist preserving $\pi$, we describe how its dynamical properties 
 degenerate at a 
singular fiber, and how they are affected by the stabilization process, which reduces a singular 
fiber to a canonical model (see below). 
Of particular interest to us is the set of points $w\in B$  such that the orbits of $f$ in $X_w$ are
 finite, or dense, or have a closure of dimension $1$ (cf. Section~\ref{par:the_dynamics_twisting}).
A first  instance of stabilization is when    $\pi$ is not relatively minimal, that is when  there is an 
exceptional curve of the first kind $E$ contained in 
a fiber of $\pi$. There are finitely many    such curves, so $f$ permutes them, and 
some positive iterate $f^m$ fixes each of them. 
Thus, one can 
contract $E$ in an $f^m$-equivariant way, to end up with a birational morphism $\e \colon X\to X'$, a fibration 
$\pi'\colon X'\to B$ such that $\pi'\circ \e=\pi$, and an automorphism $f'$
of $X'$ such that $\e\circ f^m=f'\circ \e$. The dynamical properties of $f'$ are  the  same as the ones of $f$: 
for example, the  parameters $w\in B^\circ$ such that each orbit of $f$ in $X_w$ is dense coïncide with the parameters for which the orbits
of $f'$ in $X'_w$ satisfy the  same property.

The local geometry of $\pi $ around a critical value $s\in \Crit(\pi)$  
was described  by Kodaira. 
The reader is referred  to  \cite{BHPVDV} for details, in particular   Sections III.10, V.9, and V.10 there. 
From the above discussion, we may {\emph{assume that $X$ is relatively minimal}}. We fix $s\in \Crit(\pi)$ and 
further assume that $X_s$ is not a multiple fiber; the adaptation to the case  of a multiple fiber
will be described in \S~\ref{par:multiple_fibers}

\subsubsection{Local sections}\label{par:local_sections}  A first %consequence of Kodaira's classification 
observation is that when $X_s$ is not a multiple fiber
there  exists  a local section of the fibration $\pi$ around $s$. 
More precisely, for every component $C$ of multiplicity $1$ of $\pi\inv(s)$, any small disk transverse to $C$ 
is the graph of a section $\sigma$; and by Kodaira's classification, such a component always 
exists (see \cite[\S V.7]{BHPVDV}).  
Let us  fix a small open disk $V\subset B$ around $s$, such that ${\overline{V}}\cap \Crit(\pi)=\{s\}$, together with such a local section $\sigma\colon V\to X$. 
%Let $X_s=\cup_i C_i$ be the decomposition of the central fiber $X_s$ into its irreducible component.  
Set $X_V=\pi^{-1}(V)$ and let $X_V^\sharp$ 
be the complement in $X_V$ of the irreducible
components of $X_s$ that do not intersect $\sigma(V)$ (resp. of the singular point of $X_s$ if $X_s$ is irreducible); 
this set depends the chosen section. 
In other words, we keep from $X_s$ the smooth locus of the unique component intersecting $\sigma(V)$; 
we shall denote by $X_s^\sharp\subset X_V^\sharp$ this residual curve. 

\subsubsection{Type $I_b$}\label{par:type-Ib} The main example of singular fibers are those of type $I_b$, with $b\in \N^*$
(type $I_0$ corresponds to the smooth case). For $b=1$, $X_s$ is a 
rational curve with a unique normal crossing singularity; when $b\geq 2$, $X_s$ is a cycle of $b$ smooth rational curves of 
self-intersection $-2$. So, $X_s^\sharp$ is biholomorphic to
 $\C^\times=\P^1(\C)\setminus\{ 0,\infty \}$. Shrinking $V$ if necessary, 
we can identify $(V,s)$   with a disk $(\disk_R, 0)$ of radius $R< 1$, and
  $X_V^\sharp$ with the quotient of $\disk_R\times \C$ by the family of lattices $\Lat(w)=\Z\oplus \Z \tau(w)$ given by 
\begin{equation}\label{eq:tau_w_Ib}
\tau(w)=\frac{b}{2\ii \pi} \log (w), \quad {\text{for}} \; w\in \disk_R 
\end{equation}
(note that $\log (w)$ is not well-defined but $\Lat(w)$ is).
For $w=0$, the lattice $\Lat(w)$ degenerates to $\Lat(0)=\Z \subset \C$.
If $\gamma$ is a loop making one positive turn around $s$, the  monodromy $M(\gamma)$ maps the basis $(1, \tau(w_0))$ to $(1, \tau(w_0) +b)$.

Let $t\colon V\to \C$ be a holomorphic function. The transformation $g\colon V\times \C\to V\times \C$ defined by 
$g(w,z)=(w,z+t(w))$ induces, by taking the quotient, 
 a holomorphic diffeomorphism of $X_V^\sharp$. By~\cite[Prop. III.(8.5)]{BHPVDV}, it extends to a  diffeomorphism
of $X_V$ that preserves $\pi$. Conversely, if $f$ is a holomorphic diffeomorphism of $X_V$ that preserves each  fiber of $\pi$, some positive iterate $f^m$ of $f$ preserves each component of $X_s$, and then $f^m$ maps $\sigma$ to another section $f^m\circ \sigma$ intersecting $X_s^\sharp$. Lifting to the universal cover, we see that  there is a holomorphic function $t\colon V\to \C$ such that $f$ is induced by $(w,z)\mapsto (w,z+t(w))$; the function 
$t(w)=f(\sigma(w))-\sigma(w)$ can be viewed as a section of the Jacobian fibration associated to $\pi$ (see~\cite[\S V.9]{BHPVDV} for details).  

 Consider the map %\serge{Attention, je change la variable $t$ en $v$ dans $(w,t)$, car sinon on a un petit conflit entre $t$ et $t(w)$. Mais je n'ai pas propager ce changement partout dans le reste du texte pour l'instant.}
\begin{equation}
(w,z)\in \disk_R \times \C \mapsto (w,v):= (w, \exp(2\ii \pi z)) \in \disk_R \times \C^\times.
\end{equation}
It is the quotient map for the action of $\Z\subset \Lat(w)$ on $\C$ by integral translations. 
If $w\neq 0$,  $\{w\} \times \C^\times$ is mapped in $X_V$ to the 
elliptic curve $\C^\times /w^{b\Z} \simeq X_w$ (because $\exp(2\ii \pi \tau(w))=w^b$). The fiber $\{0\}\times \C^\times$ is mapped injectively onto the central fiber $X^\sharp_s$
of $X^\sharp_V$. 

Consider the Betti foliation ${\mathcal{F}}$ defined in $\pi\inv(V\setminus\{s\})$ by the 
choice of the section $\sigma$. In $\disk_R^*\times \C$, the leaves of  ${\mathcal{F}}$ correspond to the curves $(w, c+ d\tau(w))$, for
$(c,d)\in \R^2$. They are mapped in $\disk_R^*\times \C^\times$ to the curves $\gamma_{c,d}(w)=(w, \exp(2\ii \pi c) w^{bd})$; here,
$\abs{\exp(2\ii \pi c)}=1$ because $c\in \R$ and $w^{bd}$ is  multivalued as soon as $bd\notin \Z$. 
Let us describe the local dynamics of ${\mathcal{F}}$ around $X_s$. To simplify the exposition, we contract the components of $X_s$ that do not intersect the neutral section $\sigma(V)$ onto a point $q$; this gives a new surface ${\overline{X}}_V$. 
%\serge{comparer avec la notation fin de section 3.4. On pourrait virer cela, et ne pas dire plus bas que le seul point limite est $q$ quand $d$ est irrationnel.} 
The central fiber of ${\overline{X}}_V$ is 
irreducible and $q$ is its unique singularity; when $b\geq 2$, $q$ is also a singular point of ${\overline{X}}_V$. By construction, $X^\sharp_V$ is biholomorphically equivalent to ${\overline{X}}_V\setminus \{q\}$.
When $d=0$, the leaf defined by the curve $\gamma_{c, 0}$ %(here we identify a leaf with its parameterization in $\disk_R\times \C^\times$) 
extends to a local holomorphic
section of $\pi$, given by $(w,v)=(w,\exp(2\ii \pi c))$; the union of these curves is, locally, an ${\mathcal{F}}$-invariant real $3$-manifold which intersects
the central fiber $X^\sharp_s\simeq \C^\times$ along
 the unit circle $\{ v\in \C^\times\; ; \; \vert v\vert = 1\}$.  When $d$ is  rational, $\gamma_{c,d}(w)=(w, \exp(2\ii \pi c) w^{bd})$ extends to  a local multisection of $\pi$; when $\sigma$ is the restriction of a global section of $\pi$ to the disk $V\subset B$, this local multisection $\gamma_{c,d}$ %with $(c,d)\in \Q^2$ 
 extends to an algebraic curve of $X$ (a pre-periodic curve for $[2]_B$, see Remarks~\ref{rem:Betti_foliation} and~\ref{rem:algebraic_leaves_of_FB}). Finally, when $d\in \R\setminus\Q$,  $\gamma_{c,d}$ is a transcendental and multivalued curve;  in ${\overline{X}}_V$, the singularity $q$ is the unique limit point of this curve on the central fiber.

\begin{rem}\label{rem:betti-form-Ib} 
In $\disk_R^*\times \C^\times$,
\begin{equation}\label{eq:alpha-betti}
\tilde{\alpha}_V = \log( \abs{v} ) \frac{dw}{w} - \log( \abs{w} ) \frac{dv}{v}
\end{equation}
is a real analytic $(1,0)$-form  that vanishes along  the curves $\gamma_{c,d}$. Being invariant 
under the transformation $(w,v)\mapsto (w, w^b v)$, it induces by taking the quotient 
a $(1,0)$-form $\alpha_V$ on $\pi^{-1}(V\setminus \{ s\})$, the kernel of which  
coincides with the tangent space of the Betti foliation. To get the Betti form $\omega_{\rm{B}}$ defined by $\pi$ and $\sigma$, one needs   to multiply $\alpha_V\wedge {\overline{\alpha_V}}$
by a factor $\varphi(w)$ to ensure $\int_{X_w} \varphi(w) \alpha_V\wedge {\overline{\alpha_V}}= 1$ for every $w$ in $V\setminus \{s\}$. The result is%\romain{petite modification}
\begin{align}
\omega_{\rm B} & = \frac{\ii}{2\pi} \cdot \frac{b}{2 (\log \abs{w} )^3} \alpha_V\wedge {\overline{\alpha_V}}.
\end{align} 
%\begin{align}
%\omega_{\rm B} & =  \frac{1}{2\ii (\log(\abs{w}))^3} \alpha_V\wedge {\overline{\alpha_V}} \\ \notag
%& = \frac{\log( \abs{z})^2}{(\log(\abs{w}))^3} \frac{1}{2\ii }\frac{dw}{w}\wedge \frac{d\overline{w}}{\overline{w}} + \frac{1}{\log( \abs{w} )} \frac{1}{2\ii } \frac{dz}{z}\wedge \frac{d\overline{z}}{\overline{z}} \\
%&\notag \quad  - \frac{\log(\abs{z})}{\log(\abs{w})^2}\frac{1}{2\ii }\left(\frac{dw}{w}\wedge \frac{d\overline{z}}{\overline{z}} -  \frac{dz}{z}\wedge %\frac{d\overline{w}}{\overline{w}}\right).
%\end{align} 
\end{rem}

\subsubsection{Multiple fibers (see~\cite[\S III.9 and V.10]{BHPVDV})}\label{par:multiple_fibers} Let us assume in this paragraph that $X_s$ is a multiple fiber; it is necessarily of type $mI_b$ 
for some $b\geq 0$. Let us do a local base change under the map $p\colon \zeta\mapsto \zeta^m =w$; in other words, we consider the surface $X_V'$ 
given locally above $V\simeq \disk_R$ by $X_V'=\{ (\zeta,x)\in \disk_R\times X_V\; ; \; \pi(x)=\zeta^m\}$, 
together with the projection $\pi'\colon X_V'\to \disk_R$ defined by $\pi'(\zeta,x)=\zeta$. 
Then, the map $P\colon (\zeta,x)\in X_V'\mapsto x\in X_V$ satisfies $\pi\circ P=p\circ \pi'$. 
The surface $X_V'$ may be singular, so we let $X_V''$ be the normalization of $X_V'$ and $X_V^{(m)}$ 
be the minimal resolution of $X_V''$;  there is a natural fibration $\pi^{(m)}\colon X_V^{(m)}\to \disk_R$ and a natural map $P^{(m)}\colon X_V^{(m)}\to X_V$ such that $\pi \circ P^{(m)}=p\circ \pi^{(m)}$. Now, it turns out that $\pi^{(m)}$ has no multiple fiber and that its central 
fiber (the unique possible singular fiber above $V$) is of type $I_b$. 

If $f$ is a holomorphic diffeomorphism of $X_V$ such that $\pi\circ f=\pi$, then $f$ can be lifted to a holomorphic diffeomorphism 
$f^{(m)}$ of $X_V^{(m)}$ such that $P^{(m)}\circ f^{(m)}=f$. First, one lifts $f$ to 
$X_V'$ by $(w,z)\mapsto (w,f(z))$ and then  to the normalization and its minimal resolution. Conversely, one recovers $X_V$ 
by taking the quotient of $X^{(m)}$ by the action of a finite group $\Z/m\Z$ that commutes to $f^{(m)}$. Thus, to study the local dynamics
around multiple fibers, one only needs to study the case of fibers of type $I_b$ (including smooth fibers), and take a quotient by such a finite group.

\subsubsection{Unstable fibers (see~\cite[\S III.10 and V.10]{BHPVDV})}\label{par:stable_reduction} 
Let us now  assume that $X_s$ is not multiple and  not of type~$I_b$; it is an unstable fiber of 
type $II$, $III$,  $IV$, or $I_b^*$, $II^*$, $III^*$,  $IV^*$. As in the previous paragraph, a local 
base change can be performed to end up with a local stable fibration; its central fiber will be smooth, 
except for the types $I_b^*$, $b\geq 1$ which lead to a central fiber of type $I_b$. 
To do so, one first blows up the central fiber to ensure that its singularities are nodes, which
 gives rise to a new surface ${\overline{X}}_V$; then, one does a base change to construct a new surface 
 ${\overline{Y}}_V$ (as above with $X^{(m)}$); a priori, the induced fibration on ${\overline{Y}}_V$
 is not relatively minimal anymore, so one contracts curves in the central fiber to construct a surface $Y_V$ with 
 a relatively minimal fibration. Finally, $X_V$ can 
be recovered    from $Y_V$ by taking a finite quotient, 
however  only up to    bimeromorphic equivalence (see~\cite[\S III.10 and V.10]{BHPVDV}). 

After taking some positive iterate $f^m$, so that $f^m$ fixes each irreducible component of the singular fiber, the holomorphic 
diffeomorphism   can be lifted
to  a holomorphic diffeomorphism  of $Y_V$; indeed, $f^m$ induces
a meromorphic map of $Y_V$, and this map is a local diffeomorphism by \cite[Prop. III(8.5)]{BHPVDV}.  
Thus,  to study the  dynamical properties of   $f$
 we can focus locally on 
regular fibrations  
and  singular fibrations  of type $I_b$, and then take a quotient by a finite group. Moreover, this finite group 
acts by multiplication by a root of unity on the base $V\simeq \disk_R$. 
\subsection{The dynamics of Halphen twists: twisting property}\label{par:the_dynamics_twisting}  
%%%
We pursue the study of a    Halphen twist $f\colon X\to X$ and of its invariant fibration $\pi\colon X\to B$. 
Let $U\subset B^\circ$ be an open disk, endowed with a section $\sigma\colon U\to X$   of $\pi$ and a continuous choice of basis of $H_1(X_w;\Z)$,
for $w\in U$. As in Section~\ref{par:smooth-fibers-Betti}, there is a holomorphic function $\tau\colon U\to \Hyp_+$ such that the fibers $X_w$ can be 
idendified to $\C/\Lat(w)$, where $\Lat(w)=\Z+\Z\tau(w)$. Along  each fiber $X_w =\C/\Lat(w)$,  $f$ can be expressed in the coordinate $z\in \C$ as $\xi  z +t(w)$, where  $t\colon U\to \C$ is holomorphic. Here, $\xi$ is a root of unity 
(of order dividing $12$) which is determined by the action of 
$f$ on $ H_1(X_w;\Z)\simeq\Z^2$ and, as such, is locally constant; if $\xi\neq 1$,  $f$ has finite order on 
each fiber $X_w$, $w\in U$, so $f$ is periodic, and this  contradicts the parabolicity assumption. 
Thus, $f(z)=  z +t(w)$ along  $X_w $, and {\emph{$f$ acts by translation along each fiber $X_w$, $w\in U$, of~$\pi$}}.
Now, conjugating $f$ locally by the diffeomorphism $\Psi\colon \pi^{-1}(U)\to U\times \R^2/\Z^2$ introduced in Section~\ref{par:smooth-fibers-Betti}, $f$ becomes 
\begin{equation}
f_\Psi(w,(x,y))=(w, (x,y)+T(w))
\end{equation} 
for some real analytic map  $T\colon U\to \R^2$. 
The following lemma says  that ``{\emph{$t(w)$ varies independently from $\tau(w)$}}''. 

\begin{lem}\label{lem:T-depends-on-w}
The analytic map $w\in U \mapsto T(w)\in \R^2/\Z^2$ is not constant. \end{lem}

This is shown in \cite{cantat_groupes}, proof of Proposition 2.2, but the proof makes use of 
a global multisection of $\pi$, which exists if and only if $X$ is projective\footnote{The proof in~\cite{cantat_groupes} is written for K3 surfaces but extends to other projective surfaces}. 
For completeness, we present an  argument which is   more straightforward and  applies to 
all K\"ahler surfaces.

\begin{rem}\label{rem:prelim-to-lem-T-depends-on-w}
If $\sigma$ is replaced by another local section $\sigma'$, the diffeomorphism $\Psi$ is replaced by $\Psi'=\Phi\circ \Psi$,
with $\Phi(w,(x,y))=(w, (x,y)+S(w))$ for some analytic map $S\colon U\to \R^2/\Z^2$.
Then,  $f_\Psi$ is changed into $\Phi\circ f_\Psi\circ \Phi^{-1}$ and  {\emph{$T$ is unchanged}}. Likewise, if the  basis of $H_1(X_w;\Z)$ is changed, 
$T$ is mapped to 
$A\circ T$ for some $A\in \GL_2(\Z)$. Thus, the property  that $T$ is 
locally constant does not depend on the  choices
we made. 
Moreover, $T$ is constant on $U$ if and only if $f$ preserves the Betti foliation ${\mathcal{F}}_U$ associated to $\sigma$; thus, if this property
holds above $U$ for some choice of section, then it holds above any disk $U' \subset B^\circ$ and for any choice of local section.
\end{rem}

\begin{proof}[Proof of Lemma~\ref{lem:T-depends-on-w}] We fix a K\"ahler form $\kappa$ on $X$, and compute 
the norm of the tangent map $\norm{Df}$ with respect to $\kappa$.
Assume  that $T$ is constant on $U$.  
Then  at each point  $(w,(x,y))$ of
$U\times \R^2/\Z^2$, the tangent map $(Df_\Psi)_{(w,(x,y))}$ is the identity;  this implies that 
$\norm{Df^n}$ is uniformly bounded above $U$, independently of $n$. 
By Remark~\ref{rem:prelim-to-lem-T-depends-on-w}, this property  
propagates over $B^\circ$:  if $K\subset B^\circ$ is any compact subset, there is a constant $C(K)$ such that 
$\norm{Df^n_x}\leq C(K)$ for all $n\geq 0$ and all $x\in \pi^{-1}(K)$.

Let us now study the behavior of $f$ and $Df$ near a singular fiber $X_s$ of type $I_b$, for some $s\in \Crit(\pi)$. 
Let $V$ be a small disk around $s$,   endowed
with a section $\sigma\colon V\to X$, as in \S~\ref{par:type-Ib}. 
We identify $V$ to $\disk_R$ and $X_V^\sharp$ to the quotient of $\disk_R\times \C$ by the family of lattices $\Lat(w)=\Z+\Z\tau(w)$,
with $\tau(w)$ as in Equation~\eqref{eq:tau_w_Ib}. Then, we change $f$ into a positive iterate to assume that it 
fixes each irreducible component of the central fiber~$X_s$. From Remark~\ref{rem:prelim-to-lem-T-depends-on-w} we know that the local 
Betti foliation ${\mathcal{F}}$ defined above $V\setminus\{s\}$ is $f$-invariant. As shown in Section~\ref{par:type-Ib}, the leaves of ${\mathcal{F}}$ which extend to local sections of $X^\sharp_V$ foliate a unique local $3$-manifold (that intersects $X^\sharp_s$ on a circle). 
Since $f$ is a diffeomorphism preserving ${\mathcal{F}}$ and $\pi$, it preserves this $3$-manifold. In the local coordinates $(w,v)\in\disk_R\times \C^\times$
introduced  in Section~\ref{par:type-Ib}, this means that $f$ acts by $(w,v)\mapsto (w, h(w)v)$ for some holomorphic map $  w\in \disk_R\to h(w)\in \C^\times$
taking its values in the unit circle. This shows that $h$ is a constant of modulus $1$, and that the iterates of $f$ are 
locally contained, above $V$, in a compact group of local diffeomorphisms. Thus, $\norm{Df^n}$ is also uniformly bounded 
around every singular fiber of type $I_b$. 

%On the other hand,   the parabolic behavior of $f$ implies that $\norm{Df^n}_X\to \infty$ 
%(see~Equation~\eqref{eq:parabolic_definition}). This contradiction finishes the proof. 

% Assume now that $f$ is a parabolic transformation on $X$ whose associated 
%fibration has    singular fibers of  arbitrary type. 

As explained in \S\ref{par:stable_reduction}, we can contract curves and perform a stable reduction to get an automorphism $g$
of a compact K\"ahler surface $Y$ preserving a fibration $\pi'\colon Y\to B'$, a finite map $\eta\colon B'\to B$, and a meromorphic, dominant
map $\e \colon Y\dasharrow X$ such that $\pi\circ \e  = \eta\circ \pi'$ and $\e \circ g = f \circ \e $.
The fact that $T$ is constant on some open set  is an intrinsic property, 
so it also holds  for $g$. Since all singular fibers of $\pi'$ are now of type $I_b$, the previous argument
shows that $\norm{Dg^n}$ is uniformly bounded on $Y$; in particular,  $\norm{(g^n)^*}_{H^{2}(Y;\R)}$ is uniformly bounded. Let us show that this contradicts the parabolic
behavior of $f$. Write $\e$ as a composition $\varphi\circ \psi^{-1}$ where $\varphi\colon Z\to X$ is regular and $\psi\colon Z\to Y$ is a bimeromorphic morphism.
Set $h=\psi^{-1}\circ g\circ \psi$, which is a bimeromorphic map of $Z$.
By \cite[Prop. 1.15]{diller-favre}, $\norm{(h^n)^*}_{H^{1,1}(Z;\R)}$ is bounded. Then, %with $\kappa$ a Kähler form on $X$,
\begin{equation}
\int_{Z}  ((h^n) ^\varstar \varphi^\varstar\kappa) \wedge (\varphi^\varstar\kappa) = 
 \int_{Z} \varphi^\varstar \lrpar{ (f^n) ^\varstar\kappa\wedge \kappa} = \deg(\varphi)
 \int_{X} (f^n) ^\varstar\kappa\wedge \kappa.
 \end{equation}
The left hand side is bounded because $\norm{(h^n)^*}_{H^{1,1}(Z;\R)}$ is bounded. This contradicts the parabolic behavior of $f$ and concludes the proof. 
\end{proof}

\begin{lem}\label{lem:rank_0_1}
The differential $DT_w\colon T_wU\to \R^2$ of $T$ has
rank $2$ everywhere, except for finitely many points $w_j\in U$ at which $DT_{w_j}=0$. 
The analytic map $T\colon U\to \R^2$ is an open mapping. 
\end{lem}

If $DT_w=0$, we shall say that $f$ {\bf{does not twist}} along the fiber $X_w$; this notion is intrinsic: it does not 
depend on the choice of the local diffeomorphism $\Psi$ (the argument is the same as that of  Remark~\ref{rem:prelim-to-lem-T-depends-on-w}). We define the set of non-twisting points by 
\begin{equation}
\NT_f=\{w\in B\; ; \; DT_w=0\}. %\; {\text{ and }} \; B^{\circ \circ}= B^\circ \setminus \NT_g.
\end{equation}

\begin{proof}[Proof of Lemma~\ref{lem:rank_0_1}]
Since $f$ is holomorphic, its differential commutes with the complex 
structure~$\jj_X$ of~$X$. 
This implies that $DT_w\circ \jj =j(w)\circ DT_w$ where $\jj$ is the complex structure on $U\subset B$ and $j(w)$ is 
a complex structure on $\R^2/\Z^2$ that depends on $w$ (it is the conjugate of the complex 
structure of $\C/\Lat(w)$ by the linear map $D\Psi_w$). As a consequence, the rank of $DT_w$ is
equal to $2$ or $0$. 
Assume that the real analytic set $\{w\in U\; ;\; DT_w=0\}$ contains some connected real analytic curve $C\subset U$. 
Along $C$, $T$ is a constant $T_0$; then, the image $f_\Psi(\{(w,(0,0))\; ;\; w\in U\})$ of 
the zero section intersects the horizontal
disk $\{ (w,(x,y))\; ; \; (x,y)=T_0\}$ on a real analytic curve. 
Now, let us come back to the complex surface $X$. 
The image $f(\sigma(U))$ of the zero section and, by Remark~\ref{rem:Betti_foliation}, 
the set $\Psi^{-1}\{ (w,(x,y))\; ; \; (x,y)=T_0\}$ are two connected complex analytic curves.
Since they intersect along a non-discrete subset, they coincide; this implies that $T$ is constant, in contradiction with Lemma~\ref{lem:T-depends-on-w}. 
Thus, the set of points at which $DT_w$ has rank $<2$ is locally finite. 
The last assertion follows easily from the first.
\end{proof}

\begin{rem}\label{rem:invariant_foliations}
Lemma~\ref{lem:rank_0_1} shows that  {\emph{every $f$-invariant holomorphic foliation ${\mathcal{G}}$ on $X$ is given by the invariant fibration $\pi$}}.
Indeed, if   $T(w)\in \Q^2/\Z^2$ and $DT_w\neq 0$, then a positive iterate  $f^m$ fixes $X_w$ pointwise and,  
at every point $x\in X_w$, ${\mathrm{Ker}}(D\pi_x)\subset T_xX$ is the unique line invariant by the tangent map $Df^m_x$; thus ${\mathcal{G}}$ must be everywhere tangent to $X_w$ and $X_w$ must be a leaf of ${\mathcal{G}}$. Since $T$ is open, $T^{-1}(\Q^2/\Z^2)$ is dense in $U$, so  ${\mathcal{G}}$ coincides with the fibration above $U$, hence everywhere on $X$. This proves Assertion~(3) of Theorem~\ref{thm:invariant_fibration} for compact K\"ahler surfaces (only 
projective surfaces are considered in~\cite{Cantat-Favre}).
\end{rem}

Lemma~\ref{lem:rank_0_1} does not say that $\NT_f$ is finite: it could  cluster 	at a critical value of~$\pi$. To exclude this possibility, let us first reformulate the
twisting property. Consider an open 
set $U\subset B$ endowed with a section $\sigma$ and the induced
 local Betti foliation  ${\mathcal{F}}$.  Its image $f_*({\mathcal{F}})$ is generically transverse to
${\mathcal{F}}$; these foliations  are tangent at  $x\in X_U$ if,
and only if $DT_{\pi(x)}=0$  if, and only if they are tangent along the fiber $X_{\pi(x)}$. So, 
$\pi^{-1}(\NT_f)=\Tang(\mathcal{F}, f_*({\mathcal{F}}))$.

\begin{lem}\label{lem:Betti_singular_fiber}
Let $s\in \Crit(\pi)$ be the projection of a fiber of type $I_b$, $b\geq 1$. 
Let $V$ be an open disk containing $s$, with a section $\sigma\colon V\to X$ of $\pi$; 
let ${\mathcal{F}}$ be the Betti foliation determined by $\sigma$ above $V\setminus \{s\}$. 
If $V$ is small enough,  $f_*({\mathcal{F}})$ is everywhere transverse to ${\mathcal{F}}$ above $V\setminus \{s\}$. \end{lem}

\begin{proof} We make use of Remark~\ref{rem:betti-form-Ib}, and work in $\disk_R\times \C^\times$. The automorphism $f$ is given by 
$f(w,v)=(w,  h(w) v)$ for some holomorphic function 
$h(w)=\exp(2\ii \pi t(w))$ that does not vanish. In these coordinates, the neutral 
section $\sigma$ is given by $\sigma(w)=(w,1)$, and its image under $f$ is the curve $f\circ\sigma\colon w\mapsto (w, \exp(2\ii \pi t(w)))$. 
The points $w$ above which  ${\mathcal{F}}$ is tangent to  $f_*{\mathcal{F}}$ are the points where the image of $f\circ\sigma$ is tangent to ${\mathcal{F}}$; they 
are determined by the equality ${\tilde{\alpha}}_V((f\circ \sigma)'(w))=0$, with ${\tilde{\alpha}}_V$ as in 
Equation~\eqref{eq:alpha-betti}. This leads to  the constraint 
\begin{equation} 
- \ii \log(\abs{w}) w t'(w) = \Ima (t(w));
\end{equation}
writing $t(w)= \alpha w^k + \mathrm{h.o.t.}$, for some $k$ in $\N$, 
we get $-\ii \alpha k \log(\abs{w}) w^k\simeq \Ima(\alpha w^k)$. As a consequence there is no solution close to the origin, and the proof is complete. 
\end{proof}

According to Section~\ref{par:singular-fibers}, the local dynamics of $f$ near an unstable or multiple fiber is covered by the dynamics of a parabolic automorphism near a  curve
of type $I_b$. So, 
the next result  is a corollary of the previous lemmas. 
(Note that the existence of a section in Assertion~(2) implies that
$X$ is projective and that $\pi$ does not have multiple fibers\footnote{ Indeed, if $D\subset X$ is the graph of a section and $F$ is a fiber, then  
$F\cdot D=1$, so $F$ can not be multiple; and if $F$ is reducible, $D$ intersects $F$ along a component of multiplicity $1$. Moreover, if 
$a\in \Z_+$ is large enough, then $aF+D$ is big and nef.}.) 

\begin{pro}\label{pro:NT_finite}   
Let $f$ be a parabolic automorphism of a compact Kähler surface $X$, acting trivially on the base of its invariant fibration $\pi\colon X\to B$. 
\begin{enumerate}[\em (1)]
\item The set of fibers $\pi^{-1}(w)$ along which $f$ does not twist is finite, {i.e.} $\NT_f$ is finite. 

\item If  $\pi$ admits  a global section  $\sigma\colon B\to X$ and  
  ${\mathcal{F}}$ is the associated Betti foliation on $\pi^{-1}(B^\circ)$, then 
 $\Tang(\mathcal{F}, f_*({\mathcal{F}}))$ is a finite union of fibers. 
\end{enumerate}
\end{pro}

%%%
\subsection{The dynamics of Halphen twists: orbit closures}\label{par:the_dynamics_R}  
%%%

We keep  the notation from Section~\ref{par:the_dynamics_twisting}: we let 
 $U\subset B^\circ$ be a disk, on which a continuous choice of basis 
 for $H_1(X_w;\Z)$ and a section of $\pi$ have been chosen, 
so that $H_1(X_w;\Z)\simeq \Z^2$ and $X_w\simeq \R^2/\Z^2$. 

\subsubsection{}\label{par:definition_L_k_p_q}
A  proper and closed   subgroup  of $X_w$ is finite or $1$-dimensional.   
A  closed  $1$-dimensional subgroup $L\subset X_w$ is
characterized by two data: a slope $(p,q)$, given by a primitive vector in $\Z^2$, and 
the number $k\geq 1$  of connected components of $L$. The connected 
component of the identity $L^0\subset L$ is the kernel of the homomorphism 
$X_w\simeq \R^2/\Z^2\to \R/\Z$ defined by 
\begin{equation}
(x,y)\mapsto qx-py ,
\end{equation}
and $L$ is the preimage of the unique cyclic subgroup of order $k$ in $\R/\Z$. Equivalently, $L$ is 
the kernel of  $(x,y)\mapsto k(qx-py)$. We denote this subgroup by $L_w(k, (p,q))$. The integer $k$ is intrinsically defined, but the slope depends
on the basis of $H_1(X_w;\Z)$.

\begin{nota}\label{nota:L_f} For $z\in X$, we denote by $L_f(z)$ the closure of the orbit of $z$, and by $L_f^0(z)$ the connected component of $z$ in
$L_f(z)$. For $w$ in $B$, we denote by $f_w$ the restriction of $f$ to the fiber $X_w$. If $\pi(z)\notin \Crit(\pi)$,  $f_{\pi(z)}$ is a translation 
in $X_{\pi(z)}$; thus $L_f(z)$ is either finite, or a translate
 of a $1$-dimensional subgroup of $X_{\pi(z)}$, or equal to $X_{\pi(z)}$. 
By definition a translate of a connected, closed,  $1$-dimensional subgroup  of $X_{\pi(z)}$ will be called a 
{\bf{circle}}. 
% in the second case, we say that $L_f^0(z)$ is a {\bf{circle}} in $X_{\pi(z)}$.
\end{nota}

\subsubsection{}  Define 
\begin{align}
\Tor(U) &= \{ w\in U\; ; \;  f_{w}\colon X_w\to X_w \; \; {\text{is a periodic translation}}\} \\
\notag &= \{ w\in U\; ; \;  t(w) \; {\text{  has finite order in }} \; X_w=\C/\Lat(w) \} \\
\notag &= \{ w\in U\; ; \;  T(w)\in \Q^2/\Z^2\}.
\end{align}
This  set   is   intrinsically defined (it does not depend on the section or on the basis of $H_1(X_w;\Z)$). By Lemma~\ref{lem:rank_0_1},  $\Tor(B^\circ)$ is  a countable and dense subset of $B^\circ$. 
A point $w$ belongs to $\Tor(B^\circ)$ if and only if the orbit of every $z\in X_w$ is a finite subset  of $X_w$. 

\subsubsection{}\label{par:next_level} The next level of 
complexity is when $t(w)$ belongs to a 1-dimensional subgroup of $X_w$. 
Write $f_\Psi(w,(x,y))=(w,(x,y)+T(w))$. For $(\alpha,\beta)$ in $\Q^2/\Z^2$ and $(p,q)\in \Z^2\setminus\{(0,0)\}$, we set  \begin{align}\label{eq:RabpqT}
 \mathrm R^{\alpha,\beta}_{p,q}(U) & = \{w \in U\; ; \; T(w)\in (\alpha,\beta)+ \R \cdot (p,q)\}.
\end{align}
This set depends only on the slope $(p,q)$ and on $q\alpha-p\beta\in \R/\Z$, 
but not really on $(\alpha,\beta)$ itself.
Using holomorphic coordinates, this real analytic curve ${\Rk}^{\alpha,\beta}_{p,q}(U)$ is alternatively expressed as  
\begin{center}
${\Rk}^{\alpha,\beta}_{p,q}(U)=$ $\{w\; ; \;$ the complex numbers $t(w) - \sigma_{(\alpha,\beta)}(w)$ and $p+q\tau(w)$ are 
$\R$-collinear$\}$, 
\end{center}
where  
 $\sigma_{(\alpha,\beta)}$ is the local holomorphic section of $\pi$ defined by 
 \begin{equation}
 \sigma_{(\alpha,\beta)}(w)=\Psi^{-1}(w,(\alpha,\beta))=\alpha+\beta\tau(w) \mod\Lat(w);
 \end{equation}
 here, as in Equation~\eqref{eq:Psi} ,  $\Psi_w$ maps the basis $(1,\tau(w))$ of $\Lat(w)\subset \C$ onto the basis $((1,0),(0,1))$
 of $\Z^2\subset \R^2$.
This equation of $R^{\alpha,\beta}_{p,q}$ can  be written locally as 
\begin{equation}\label{eq:R_ab_pq}
\Ima\left( \frac{t(w)-\sigma_{(\alpha,\beta)}(w)}{p+q\tau(w)}\right)=0.
\end{equation} 
This curve may have singularities, as does the curve defined by $\Ima(w^k)=0$ in the unit $\disk$ when $k\geq 2$; %in that case, it locally looks like a star with $2k$ branches, or more precisely like the union of $k$ segments through the origin. 
this happens precisely when $w\in \mathrm{NT}_f$.  

Let $\Rk^k_{p,q}(U)\subset U$ be the closure of the set of points $w$ for which $L_f(z)$ 
has $k$ connected components of slope $(p,q)$ for all $z\in X_w$; it is 
the (finite) union of the $\Rk^{\alpha,\beta}_{p,q}(U)$, for all $(\alpha,\beta)$ in 
$\Q^2/\Z^2$ such that $q\alpha-p\beta$ has order $k$ in $\R/\Z$. Thus, $\Rk^k_{p,q}(U)$ is an analytic curve in~$U$.  
For  $w\in  \Rk^k_{p,q}(U)\setminus \Tor(U)$ and  $z\in X_w$, $L_f(z)$ is a translate of $L_w(k,(p,q))$.  

When performing an analytic continuation of $\Rk^{k}_{p,q}(U)$ around a critical value of $\pi$, 
the continuation may hit $U$ again along a component of $\Rk^{k}_{p',q'}(U)$ for some new slope $(p',q')$; 
the vector $(p',q')$ is in the orbit of $(p,q)$ under the action  of the  monodromy group of the fibration.
Since the orbit of $(p,q)$ is typically infinite, the analytic continuation could a priori intersect $U$ on 
infinitely many distinct $\Rk^{k}_{p',q'}(U)$.

Finally, for  each integer $k\geq 1$, we set 
%\begin{align}
%\Rk^k(U) &=   \{ w\in U\; ; \;  {\text{for every}} \; z\in X_w, L_f(z) \; {\text{is a subgroup of dimension}} \; 1  \\
%& \quad \quad \quad \quad \quad   \; {\text{in}}  \; X_w  \; {\text{with}}  \; k  \; {\text{connected components}} \} \notag
% \end{align}
%\begin{align}\label{eq:circles}
%\Rk^k(U) &=   \set{ w\in U\; ; \;  {\text{for every}} \; z\in X_w, L_f(z) \; {\text{is a union of }k\text{ circles}}} \end{align} 
\begin{equation}\label{eq:circles}
\Rk^k(U) = \bigcup_{(p,q)\in \Z^2 \text{ primitive}} \Rk^{k}_{p,q}(U), \;\;\; \text { and } \;\;\; \Rk(U)=\bigcup_{k\geq 1} \Rk^k(U).
\end{equation}
%and $\Rk(U)=\bigcup_{k\geq 1} \Rk^k(U)$; 
%again, the sets $\Rk^k(U)$ and $\Rk(U)$ are  intrinsically defined. 
%Let us express $\Rk^k(U)$ as a   real analytic subset of dimension $1$ in $U$. 
These sets are intrinsically defined. Intuitively $\Rk^k(U)$ should be thought as the set of $w\in U$ such 
that for $z\in X_w$, $L_f(z)$ is a union of $k$ circles; formally, this  is not a correct characterization of $\Rk^k(U)$ because  
$\Rk^k(U)$ contains points of    $\Tor(U)$.

\subsubsection{} Summarizing this discussion, and keeping the same notation, we obtain: 

\begin{lem}\label{lem:basic_dynamics} Let $f$ be a parabolic automorphism of a compact K\"ahler surface acting trivially on 
the base of its invariant fibration $\pi\colon X\to B$. Let $U\subset B^\circ$ be an open disk.
\begin{enumerate}[\em (1)]
\item The set $\Tor(B^\circ)$ is a dense and countable subset of $B$, and $w\in \Tor(B^\circ)$ if and only if the translation $f_w$ is periodic, 
if and only if $L_f(z)$ is finite for every $z\in X_w$.
\item For each slope $(p,q)$, the set $\Rk^k_{p,q}(U)$ is either empty, or a (possibly 
singular) real analytic curve; the set $\Rk^k(U)$ is the union of these curves. 
Moreover,  $w\in \Rk(B^\circ) \setminus \Tor(B^\circ)$
if and only if $L_f^0(z)$ is a circle for each $z$ in $X_w$ if and only if the closure of each orbit of $f_w$ is a union of circles embedded in $X_w$; 
\item  If $w\in B^\circ\setminus \Rk(B^\circ)$, then each orbit of $f_w$ is dense in $X_w$.
\end{enumerate}
In the second case, every  $f_w$-invariant and ergodic probability measure is the Lebesgue measure on $L_f(z)$ for some $z\in X_w$;   in the third case, the only $f$-invariant probability measure on $X_w$ is the Lebesgue measure.
\end{lem}

\subsection{Additional notes  on the curves $\mathrm R^{\alpha,\beta}_{p,q}(U)$}\label{par:notes_on_gamma}
 Fix a finite set of real analytic arcs (slits)  $\gamma_j\subset B$ intersecting transversally, and containing the 
critical values of $\pi$, such that the  complement of $\bigcup \gamma_j$ in $B$ is topologically 
 a disk. Denote this disk by $U$. Then, choose  a section $\sigma$ of $\pi$ and a continuous family of basis for $H_1(X_w;\Z)$ above $U$. Each $\mathrm R^{\alpha,\beta}_{p,q}(U)$ is a real analytic
curve in $U$; but when crossing an arc $\gamma_j$, the analytic continuation of $\mathrm R^{\alpha,\beta}_{p,q}(U)$ must be glued to 
another $\mathrm R^{\alpha',\beta'}_{p',q'}(U)$. 
%\begin{figure}[h]
%\centering\epsfig{figure=curve_R2.pdf}
%\caption{{\small{A curve $\mathrm R^{\alpha,\beta}_{p,q}$.-- There are four critical values of $\pi$, marked with circled
% dots; one of them is a singularity of the curve. One of the irreducible components of the curve  has a non-compact analytic continuation. }}}\label{fig:Rpq}
%\end{figure}

In \S\ref{par:local_Rabpq}, we will have to understand the local structure of these curves near critical values of $\pi$ and,  
as before, the important case is that of   singular fibers of type $I_b$. Here, we content ourselves with a few simple remarks, which help understand some of the subtleties of the problem of the semi-analytic extension of the curves $\sigma_g$ in Section~\ref{sec:semi-analytic}. 

Choose a local coordinate $w$ as in \S~\ref{par:type-Ib} and, to fix the ideas,  suppose $(\alpha,\beta)=(0,0)$.  
 From 
Equations~\eqref{eq:R_ab_pq} and~\eqref{eq:tau_w_Ib},  $R^{0,0}_{p,q}$ is determined by the equation
\begin{equation}\label{eq:R_00_pq}
\Ima\left( \frac{t(w) }{p+q \frac{b}{2\ii\pi}\log (w) }\right)=0, 
\end{equation} 
where $\log (w)$ is viewed as a multivalued function. 
The multivaluedness of the logarithm takes care of the 
monodromy   in the sense that winding around the origin changes $\log (w)$ into $\log (w) + 2\ii\pi$, 
which in this equation corresponds to the monodromy action $(p,q)\mapsto (p+qb,q)$. 
The function 
$t(w)$ is a well-defined local holomorphic function, which may or may not vanish at the origin (see \S\ref{par:type-Ib}). 

If $q=0$, Equation~\eqref{eq:R_00_pq} reduces to $\Ima(t(w))=0$ which is an analytic subset of the disk (including the origin if $\Ima(t(0))=0$). 

Now, we  focus on the case $q\neq 0$, and we assume that    $t(w)=t_0 w^k$. 
Write  $w = e^{-s}$, where $s$ belongs to some right half plane, and write $t_0 = e^{ks_0}$.
 Then, Equation~\eqref{eq:R_00_pq} becomes 
 \begin{equation}\label{eq:R_00_pq_2}
\Ima\left( \frac{ e^{-k(s-s_0)} }{p+q \frac{b}{2\ii\pi} (-s) }\right)=0.
\end{equation}
Writing $s = x+iy$ and $s_0 = x_0+\ii y_0$, and making   a few elementary manipulations, 
 this equation reduces to 
\begin{equation}\label{eq:R_00_pq_final}
 x = \lrpar{y+ \frac{2\pi p}{b q}} \tan k(y-y_0), 
\end{equation}
which after a vertical translation gives $ x = (y-y_1)\tan (ky)$ for some   $y_1\in \R$. %\romain{j'avais dit $\pi\Q$ mais c'est $\R$ à cause du $y_0$}
There are two  different regimes: 
\begin{enumerate}[(1)]
\item if $k=0$ the curve is a line, which descends to a logarithmic spiral in the $w$-coordinate; 
\item if $k>0$, one distinguishes two cases, depending on whether
 $y_1$ is   of the form $\frac{1}{k}(\frac{\pi}{2}+j\pi)$, $j\in \Z$,  or not (see   Figure~\ref{fig:tangente}). In both cases, these curves have infinitely many branches asymptotic to the horizontal lines $ y  = \frac{1}{k}(\frac{\pi}{2}+j\pi)$, $j\in \Z$, as $x\to +\infty$. In the $w$-plane, these branches  have well-defined 
 tangent directions at the origin, so they extend as $C^1$ curves at 0, but they are not semi-analytic. 
 %\romain{est-ce si clair?}\serge{Il se pourrait que ce soit tout de même semi-analytic, non?}\romain{j'ai l'impression que non...}
\end{enumerate}

\begin{figure}[h]
\begin{minipage}{6cm}\includegraphics[width=6cm]{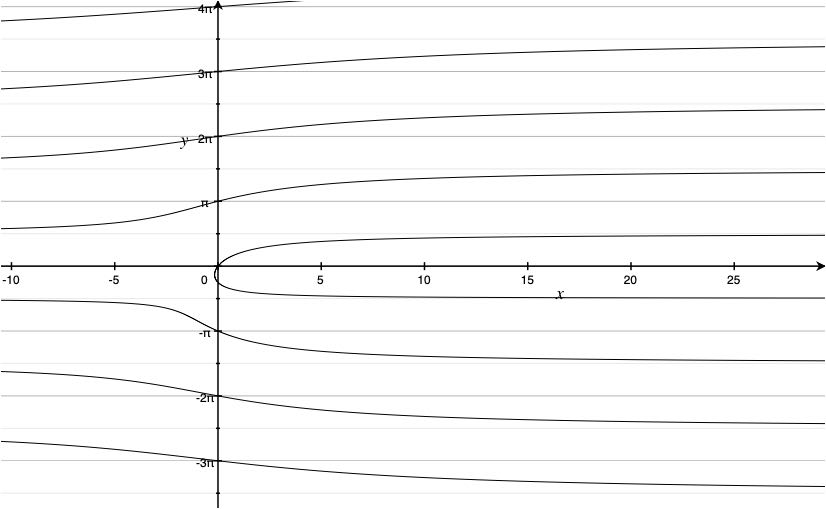}\end{minipage} \hspace{1cm}
\begin{minipage}{6cm}\includegraphics[width=6cm]{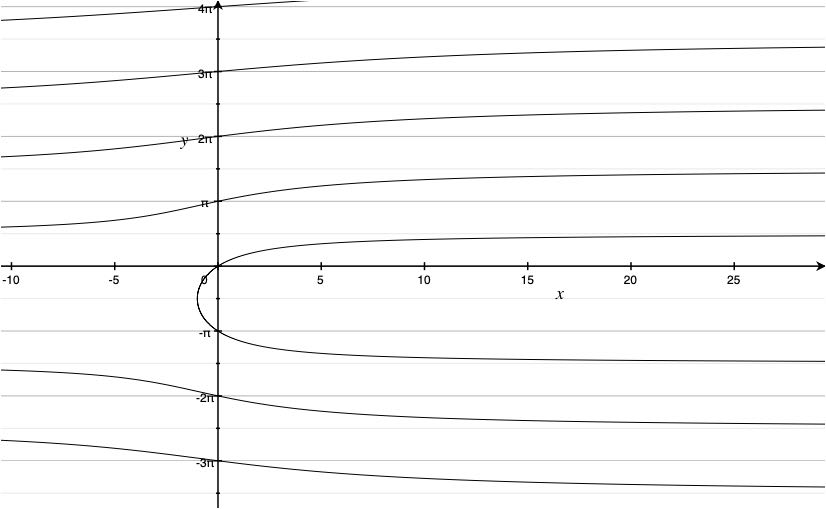}\end{minipage}
\caption{The curves $x = \lrpar{y+ \frac{\pi}{4}}\tan y$ (left) and $x = \lrpar{y+ \frac{\pi}{2}}\tan y$ (right).}\label{fig:tangente}
\end{figure}

%
%
%In polar coordinates $w=\rho e^{2\ii \pi \theta}$, 
%with $0\leq \rho<R\leq 1$ and $-1/2\leq \theta<1/2$, the equation for $R^{0,0}_{p,q}$ becomes (see Equations~\eqref{eq:R_ab_pq} and~\eqref{eq:tau_w_Ib}):
%\begin{equation}
% \log(\rho) \Rea(t(w)) = -2\pi \left( \frac{p}{q\beta} +  \theta \right) \Ima(t(w))
%\end{equation}
%if $q\neq 0$ and $\Ima(t(w))=0$ if $q=0$.
%Now, assume  $q\neq 0$ and, for simplicity, let us consider the case $t(w)=re^{2\ii \pi \alpha} w^k$; 
% then this equation reads 
%\begin{equation}
% \log(\rho)  = -2\pi \left( \frac{p}{q\beta} +  \theta \right) \tan(2\pi (k\theta+\alpha)).
%\end{equation}
%Thus, $\rho$ becomes a multivalued function of $\theta$; changing $\theta$ into $\theta+1$ corresponds to the monodromy action $(p,q)\mapsto (p+q\beta,q)$. When $k=0$ and $\alpha \neq 0$ we obtain $\rho(\theta+1)=\exp(-2\pi \tan(2\pi\alpha))\rho(\theta)$ and the curve spirals down to the origin (clockwise if $0<\alpha <1/2$ and
%counterclockwise if $-1/2 < \alpha < 0$).
%Now assume $k>0$. When $\theta$ goes to $\pm \frac{1}{4k}-\frac{\alpha}{k}+\frac{\ell}{k}$, one branch of this
%curve goes to the origin (except when $ \frac{p}{q\beta} \pm \frac{1}{4k}-\frac{\alpha-\ell}{k}=0$); but the curve does not spiral down to the origin. 
%
%Other types of singular fibers lead to a similar behavior of $\mathrm R^{\alpha,\beta}_{p,q}(U)$ since we can always operate a local stable reduction.  

%%%%%%%%%%%%%%%%%%%%%%%%%%%%%%%%%%%%%%%%%%%%%%
%%%%%%%%%%%%%%%%%%%%%%%%%%%%%%%%%%%%%%%%%%%%%%
\section{Proof of Theorem \ref{thm:main}: preliminaries,  first steps, and corollaries}\label{par:Main Proof1}
%%%%%%%%%%%%%%%%%%%%%%%%%%%%%%%%%%%%%%%%%%%%%%
%%%%%%%%%%%%%%%%%%%%%%%%%%%%%%%%%%%%%%%%%%%%%%

Let $X$ be a smooth, compact, K\"ahler surface. Fix a subgroup $\Gamma$ of $\Aut(X)$ such that:
\begin{enumerate}[(i)]\label{assumptions_Gamma}
\item $\Gamma$ is non-elementary,
\item $\Gamma$ contains a parabolic element.% (equivalently, a Halphen twist) $f$.
\end{enumerate}
Then, as shown in~\cite{stiffness}, $X$ is a smooth complex projective surface. 
%We shall also assume, for simplicity, that $X$ is not an abelian surface (see \S~\ref{}).
Our goal in this section is to prove Theorem~\ref{thm:main} under the stronger assumption: 
\begin{enumerate}[(i)]
\item[(ii')] $\Gamma$ contains a Halphen twist $g$ (see~\S~\ref{par:Halphen-intro}); equivalently $\Gamma$ contains a parabolic automorphism acting with finite order on the base of its invariant fibration.  
\end{enumerate}
%Recall that by definition a Halphen twist is parabolic maps acting trivially on the base  of its invariant fibration.
As explained in  Theorem \ref{thm:invariant_fibration}, 
this hypothesis   automatically follows from (ii) when $X$ is not an Abelian surface.  
  For the easier case of Abelian surfaces, a direct  proof of Theorem \ref{thm:main} is given in the Appendix.

\begin{nota}\label{nota:Hal} If $h$ is a Halphen twist, we denote by $\pi_h\colon X\to B_h$ 
its invariant fibration, and by $\Sing(\pi_h)$ the union of its singular and multiple fibers.
For $U\subset B_h$ (resp. $w\in B_h$), we set $X^h_U=\pi_h^{-1}(U)$ (resp. $X_w^h=\pi_h^{-1}(w)$). Similarly, 
we make use of the notation $\Rk_h(U)$ and $\Tor_h(U)$, with the index $h$.
We denote by $\Hal(\Gamma)$ the set of Halphen twists $h\in \Gamma$   
preserving each  irreducible component of each fiber of $\pi_h$. This set is invariant under 
conjugation: if $f\in \Gamma$ and $h\in \Hal(\Gamma)$, then $f^{-1}\circ h\circ f$ is an element 
of $\Hal(\Gamma)$ and its invariant fibration is given by $\pi_g\circ f$. 
\end{nota}

\begin{rem}\label{rem:pig}
The foliation defined by $\pi_h$ is uniquely determined by $h$, but  $\pi_h$ itself is not canonically defined: post-composition by an automorphism of $B_h$ would give another projection defining the same fibration. Thus, the notation means that a projection $\pi_h$ was chosen for every fibration invariant by an element $h\in \Hal(\Gamma)$, the choice being the same for two twists preserving the same fibration.
Then  $\pi_h\circ f$ equals  $\pi_{f^{-1}\circ h\circ f}$ up to post composition by 
an automorphism of the base.
\end{rem}
 
\subsection{Invariant curves, $\Tang $ and $\STang $}\label{par:invariant_curves}
According to~\cite[Lemma 2.13]{stiffness} and~\cite[Section 3]{finite_orbits}, there is a unique reduced, effective, and 
$\Gamma$-invariant divisor $\mathrm D_\Gamma$ in $X$ such that: 
\begin{enumerate}[(1)]
\item the $\Gamma$-periodic irreducible curves $C\subset X$
are exactly the  irreducible components $C_i$ of $\mathrm D_\Gamma$;
\item the intersection form is negative definite on 
\begin{equation}
V(\mathrm D_\Gamma):=\Vect([C_i], i=1, \ldots, k)\subset H^2(X;\R).
\end{equation} 
\end{enumerate}

\begin{rem}  If $g\in \Hal(\Gamma)$, the divisor  $\mathrm D_\Gamma$ is made of irreducible components of fibers of $\pi_g$, but {\emph{$\mathrm D_\Gamma$ does not contain any 
complete fiber}}: indeed   the intersection form is negative definite on $V(\mathrm D_\Gamma)\subset H^2(X;\R)$, while the self-intersection of a fiber is $0$.
\end{rem}

%Set $\mathrm F_\Gamma=\bigcap_{g\in \Hal(\Gamma)} \Sing(\pi_g)$; 
%it is an algebraic subset of $X$, and we denote by $\mathrm F_\Gamma^1$ and  $\mathrm F_\Gamma^0$ its $1$-dimensional and $0$-dimensional parts, respectively. For $g\in \Hal(\Gamma)$, the fibration $\pi_g$  is determined by any of its fibers $\pi^*(w)$;  as a consequence, 
%$\mathrm F_\Gamma^1$ is  
% the union of all irreducible curves $C$ such that  $\pi_g(C)$ is a point  for every $g\in \Hal(\Gamma)$.  

%\romain{je supprime la définition de $\mathrm{F}_\Gamma$ qui ne sert plus et introduis STang}
For $(g,h)\in \Hal(\Gamma)^2$, we let $\Tang(\pi_g, \pi_h)$ 
be the set of points $x\in X$ such that  $(d\pi_g\wedge d\pi_h)(x)=0$, and define 
\begin{equation}\label{eq:Tang_Gamma}
\Tang_\Gamma = \bigcap_{(g,h)\in \Hal(\Gamma)^2} \Tang(\pi_g, \pi_h)
\end{equation}
In plain words,  $x\notin\Tang_\Gamma$ if one can find  $g$ and $h$ in $\Hal(\Gamma)$ such that $\pi_g$ and $\pi_h$ are transverse projections in a neighborhood of $x$ 
(this is compatible with our convention for $\pi_g$, see Remark \ref{rem:pig}).  
Note that by definition if $F$ is a  multiple component of a fiber of 
$\pi_g$ then $d\pi_g\wedge d\pi_h=0$ along $F$ for all $h\in \Hal(\Gamma)$ (see \S\ref{par:tangencies} for more on this).
%Then, {{$\Tang_\Gamma$ is a $\Gamma$-invariant algebraic subset of $X$ that contains $\mathrm F_\Gamma$}}. 
Let $\Tang^1_\Gamma$  (resp. $\Tang^0_\Gamma$) 
be the union of the 1-dimensional (resp. 0-dimensional) components of $\Tang_\Gamma$. 
 We also put 
\begin{equation}\label{eq:STang}
\STang(\pi_g, \pi_h) = \Tang(\pi_g, \pi_h) \cup \Sing(\pi_g)\cup \Sing(\pi_h)
\end{equation}
and  
\begin{equation}\label{eq:STang_Gamma}
\STang_\Gamma = \bigcap_{(g,h)\in \Hal(\Gamma)^2} \STang(\pi_g, \pi_h), \end{equation}
and define 
$\STang^1_\Gamma$   and $\STang^0_\Gamma$ similarly.

\begin{lem}\label{lem:Tang_DGamma}
The reduced divisors given by $\mathrm D_\Gamma$,% $\mathrm F_\Gamma^1$ and  
 $\Tang^1_\Gamma$ and $\STang^1_\Gamma$ are all equal. The $0$-dimensional parts
 $\Tang^0_\Gamma$ and $\STang^0_\Gamma$ are  finite  $\Gamma$-invariant sets.  
\end{lem}

\begin{proof}
By definition, $\Tang(\Gamma)$ and $\STang_\Gamma$ are $\Gamma$-invariant algebraic sets, and  $\STang^1_\Gamma$ contains~$\Tang_\Gamma^1$. So, 
$ \Tang_\Gamma^1 \subset \STang^1_\Gamma\subset \mathrm D_\Gamma$.
Now, fix  $g\in \Hal(\Gamma)$. If $C$ is a $g$-periodic irreducible
curve, Lemma~\ref{lem:basic_dynamics} entails that $\pi_g(C)$ must be a point. 
%So, any component  of $\mathrm D_\Gamma$ is contained in 
%$\mathrm F_\Gamma$ and, as a consequence,
So, all fibrations $\pi_g$ must be tangent along any component of $D_\Gamma$, and it follows that 
 $ \mathrm D_\Gamma=\Tang^1_\Gamma$. 
The second  assertion is straightforward.
\end{proof}

\begin{pro}[See~\cite{finite_orbits}, Proposition 3.9]\label{pro:contraction_of_DGamma}  
There is a normal projective surface $X_0$ and a birational morphism $\eta\colon X\to X_0$ 
such that \begin{enumerate}[\em (1)]
\item $\eta$ contracts $D_\Gamma$ on a finite number of points;
\item $\Gamma$ induces a subgroup of $\Aut(X_0)$, {{i.e.}} there is an injective homomorphism $f\in \Gamma \mapsto f_0\in \Aut(X_0)$
such that $\eta\circ f = f_0\circ \eta$ for all $f\in \Gamma$. 
\end{enumerate}
\end{pro}

\begin{eg} 
 The Coble surfaces and the surfaces constructed by Blanc (see \S\ref{par:examples_rational}) provide examples of pairs $(X,\Gamma)$ 
 such that $X$ is rational and $\Gamma$ preserves a smooth rational curve or a smooth curve of genus $1$, respectively.  \end{eg}
 %\serge{Remarque 4.7 (singularités quotient) supprimée}
 
%\begin{rem} \romain{pourrait être développé ailleurs cf.ma remarque sur la section \ref{sec:examples}}\serge{Je ne trouve pas cela si mal placé. Mais on pourrait éventuellement l'enlever} 
%Assume that $X$ is a  K3 or Enriques surface. Since the intersection form is negative definite on $V(D_\Gamma)$ and $K_X$ is numerically trivial,
%the genus formula shows that the irreducible components of $D_\Gamma$ are smooth rational curves of 
%self-intersection $-2$. Thus, the singularity obtained by contracting a connected component of $D_\Gamma$
%is a  Klein singularity, with Dynkin diagram $A_n$, $D_n$, $E_6$, $E_7$, or $E_8$ (see~\cite[\S III.7]{BHPVDV} and \cite{Cantat:Acta}): there is a neighborhood 
%of the singularity which is holomorphically equivalent to a quotient of a neighborhood of the origin in $\C^2$
%by a finite subgroup of $\SL_2(\C)$. \end{rem}

\subsection{Analytic subsets of positive mass, $d_\R(\mu)$, and $d_\C(\mu)$} Let $\mu$ be a $\Gamma$-invariant and ergodic 
probability measure. Denote by $d_\R(\mu)$ the minimum of the dimensions $k\in \{0,1,2,3,4\}$ such that 
%$\mu(W)>0$ for some immersed, irreducible, 
%real analytic subset $W\subset X$ of real dimension $k$. 
%(By {\em{local}}, we mean that $W$ is defined in some open subset of $X$; we do not assume $W$ to be a global, closed analytic subset.)
%\begin{lem}\label{lem:dR}
%If $d_\R(\mu)<\dim_\R(X)$, $\mu$ is supported on a $\Gamma$-invariant immersed real analytic subset of dimension $d_\R(\mu)$. 
%\end{lem}
%
%\begin{proof}
%Pick  an immersed irreducible real analytic subset  $W$ of dimension $d_\R(\mu)$ with $\mu(W)>0$. 
%For  $f\in \Gamma$, either $f(W)=W$ or $\mu(f(W)\cap W)=0$
% because if $f$ does not preserve $W$, then $f(W)\cap W$ is a countable union of immersed analytic subsets of 
%smaller dimension, each of them of measure zero by the definition of $d_\R(\mu)$. In this latter case, we obtain $\mu(f(W)\cup W)=\mu(f(W))+\mu(W)=2\mu(W)$. So, we
%deduce from $\mu(X)=1$ that a subgroup of $\Gamma$ of index $\leq \mu(W)^{-1}$ preserves $W$. 
%By ergodicity, the measure $\mu$ is in fact supported on an immersed, $\Gamma$-invariant, possibly  reducible, 
%real analytic subset of dimension $d_\R(\mu)$.
%\end{proof}
there exist an open set $U\subset X$, for the euclidean topology, and a real analytic submanifold  $W\subset U$ 
 of real dimension $k$ with $\mu(W)>0$. 
 
\begin{rem} 
Pick such a local analytic submanifold $W\subset U$ with $\mu(W)>0$ and $\dim_\R(W)=d_\R(\mu)$. 
By ergodicity, $\mu$ gives full mass to  
$\bigcup_{g\in \Gamma} g(W)$. If $W\cap g(W)$ has empty relative interior in 
 $W$ and $g(W)$,  
 then 
$g(W)\cap W$ is contained in at most countably many analytic  submanifolds of lower dimension,   so  $\mu(g(W)\cap W) = 0$.  
\end{rem}

Likewise, let $d_\C(\mu)$ be the minimal  dimension $k\in \{0,1,2\}$ such that 
$\mu(Z)>0$ for some irreducible local {complex analytic subset} $Z\subset X$ of (complex) dimension $k$. 
 
\begin{lem}\label{lem:dC}  Assume that $\Gamma$ satisfies (i) and (ii'), and let 
$\mu$ be an ergodic $\Gamma$-invariant measure.  
\begin{enumerate}[\em (1)]
\item If $d_\C(\mu)=0$, $\mu$ is supported on a  finite orbit of $\Gamma$.
 \item If $d_\C(\mu)=1$, $\mu$ is supported on $\mathrm D_\Gamma$. 
  \end{enumerate}
\end{lem}

So, we see that  $d_\C(\mu)=2$ is equivalent to:  {\emph{$\mu$ gives no mass to algebraic sets}} (or more precisely {\emph{to proper algebraic subsets}}). 

\begin{proof} The zero-dimensional case is left to the reader. So, suppose that $\mu$ has no atom 
and let $Z\subset X$   be a local, $1$-dimensional, complex analytic set   such that $\mu(Z)>0$. If $Z$ is not
contained in $D_\Gamma$,  one can find  $g\in \Hal(\Gamma)$ and a  holomorphic disk $\Delta\subset Z$ of positive measure 
which is transverse to $\pi_g$. By Lemma~\ref{lem:rank_0_1},  for any $k\neq \ell$, $g^k(\Delta)\cap g^\ell(\Delta)$ is a finite set; since  $\mu$ is atomless and $g$-invariant, this implies $\mu(\bigcup_{n\geq 0} g^n (\Delta))=\sum_{n\geq 0} \mu(\Delta) =\infty$.  This contradiction completes the proof. 
\end{proof}

%%%%%%%%%%%%%%%%%%%%%%%%%%%%
\subsection{The smooth case} \label{subs:smooth_case}
%%%%%%%%%%%%%%%%%%%%%%%%%%%%
For $g\in \Hal(\Gamma)$, the (marginal) measure $\mu_g:=(\pi_g)_*\mu$ is a probability measure on~$B_g$. 
 Let $\Rk_g = \Rk(B_g^\circ)$ (cf. Equation~\eqref{eq:circles}).

\begin{pro}\label{pro:smooth_case} 
Let $\mu$ be an invariant and $\Gamma$-ergodic measure which 
  gives no mass to proper algebraic subsets 
(i.e. $d_\C(\mu) = 2)$). Assume that there exists  an element $g\in \Hal(\Gamma)$ such that 
\begin{equation}\label{eq:mug_positive}
\mu_g(B_g^\circ\setminus \Rk_g)>0.
\end{equation}
Then $\mu$ is absolutely continuous with respect to the Lebesgue measure on $X$
and its support is equal to $X$. Moreover, the density of $\mu$ with respect to any real analytic 
volume form on $X$ is a real analytic function in the complement of $\STang_\Gamma$. 
\end{pro}

The proof occupies \S\S~\ref{prelim:choice_of_a_good_group} to 
\ref{subsub:proof_smooth_case_general} below.  
 
\subsubsection{Preliminaries: special subgroups}\label{prelim:choice_of_a_good_group}  

\begin{defi}\label{defi:special} A pair $(g,h)$ of Halphen twists is \textbf{special}  if 
\begin{enumerate}%[\em (1)]
\item the group $\langle g, h\rangle$ is a non-abelian free group on two generators (in particular the fibrations $\pi_g$ and $\pi_h$ are distinct);
\item an element $f \in \langle g, h\rangle$ is
a power of $g$ if and only if it fixes the class of the fibration $\pi_g$ in the N\'eron-Severi group $\NS(X;\Z)$, if and only if it 
permutes the fibers of $\pi_g$, if and only if it maps some smooth fiber of $\pi_g$ to a fiber of $\pi_g$;
\item Property (2) holds also for $h$ in place of $g$.
\end{enumerate}
\end{defi}
In Assertion~(2), the important part is that an $f\in \langle g, h\rangle$ that permutes the fibers of $\pi_g$ is an element of $g^\Z$. Indeed, the following remark shows that the last three properties in Assertion~(2) are always equivalent.

\begin{rem}\label{rem:transport_of_smooth_fiber} 
Consider an arbitrary pair 
$(g, h)$  of Halphen twists, and 
let $E$ be a scheme theoretic fiber of $\pi_g$ (that is, if $w$ is a local coordinate near 
$w_0=\pi_g(E)$, the equation of $E$ is $\pi_g(\xi)=w$).  
If $f\in \Gamma$  maps $E$ in a fiber $E'$ of $\pi_h$, 
then the class $[f(E)]$ must be proportional to $[E']$, because the self-intersection of $f(E)$ 
is zero and in the vector space generated by
the classes of the components of $E'$, all isotropic vectors are proportional to $[E']$. 
Thus the classes of the fibers of 
$\pi_g\circ f$  and $\pi_h$ generate the same ray in $H^{1,1}(X;\R)$; by Remark~\ref{rem:theorem_invariant_fibration},  $\pi_g\circ f^{-1}$ and $\pi_h$ are equal, up to post composition by an isomorphism $B_h\to B_g$. 
In particular, if $g$ and $h$ preserve distinct fibrations, no fiber of $\pi_g$ is entirely contracted by~$\pi_h$.
\end{rem}

\begin{lem}\label{lem:schottky-parabolic} 
If $g$ and $h$ are Halphen twists associated to distinct fibrations, then for large enough $n$, 
the pair $(g^n, h^n)$ is special. 
\end{lem}

\begin{proof}  Consider the action of $\Gamma$ on $H^{1,1}(X;\R)$. It preserves the isotropic cone 
 \begin{equation}
 \{u\in H^{1,1}(X;\R)\; ; \; \langle u\vert u\rangle=0\},
 \end{equation} where $\langle \cdot \vert \cdot\rangle$ is the intersection 
form. Projectively, this cone is a sphere ${\mathbb{S}}$ (which can be considered as the boundary of a hyperbolic space, see~\cite{Cantat:Milnor, stiffness}). 
Now, $g^*$  is a parabolic transformation on ${\mathbb{S}}$, fixing a unique point $u_g$. Let us fix a small neighborhood $U_g$ of $u_g$; 
if $x$ is a point of ${\mathbb{S}}\setminus \{u_g\}$ 
there is small neighborhood $V(x)$ of $x$ in ${\mathbb{S}}$ and a positive integer $m_g(V(x))$ such that $(g^n)^*(V(x))\subset U_g$ 
for every $n\in \Z$ with 
$\abs{n}\geq n_g(V(x))$. 
A similar property is satisfied by $h$. The points $u_g$ and $u_h$ are determined by the classes of the fibers of the invariant fibrations
$\pi_g$ and $\pi_h$.
We choose  $U_g$ and $U_h$  small and disjoint, and   $n_0$ 
such that for all $n\in \Z$ with $\abs{n}\geq n_0$
\begin{equation}\label{eq:pingpong}
(g^n)^*(U_h)\subset U_g\setminus\set{u_g} \text{ and }(h^n)^*(U_g)\subset U_h\setminus\set{u_h}.
\end{equation} 
In the following we fix   such an $n$.  
Then, the first assertion follows from the ping-pong lemma (see~\cite{delaHarpe:Book}, Chapter I). 

Let us show   
that  if $f\in  \langle g^n, h^n\rangle$, then $f^*u_g = u_g$ if and only if $f\in \langle g\rangle$.
 We assume that $f^*u_g = u_g$, and we want to show that $f$ is an iterate of $g$.
 Write $f$ as a  word in $g^n$ and $h^n$: 
$f=g^{nk_1}\circ h^{n\ell_1}\circ \cdots \circ h^{n\ell_s}$, 
the $k_i$ and $\ell_i$ being non-zero integers, except possibly for 
$k_1$ and $\ell_s$ which may vanish. 
The proof is by induction on $\abs{f}=\vert \{i\; ; \; k_i\neq 0\}\vert+ \vert \{i\; ; \;  \ell_i\neq 0\}\vert$. 
The result is obvious when $\abs{f}=0$ or $1$. 
 
 The point $u_g$  corresponds to an isotropic line in $H^{1,1}(X;\R)$, and this line intersects $H^2(X;\Z)$ on 
 a discrete set $\Z [D]$, for some integral class $[D]$ on the boundary of the ample cone.%\serge{J'ai enlevé la parenthèse. Elle n'est pas fausse, mais pour montrer que $D$ est effectif il faut travailler un peu.} %($D$ is a divisor, and some positive multiple of $[D]$ is the class of a fiber of $\pi_g$). 
 Since $f^*$ preserves $H^2(X;\Z)$ and the ray $\R_+ [D]$, it fixes $[D]$. Thus, $f^*$ cannot be loxodromic (see~\cite{Cantat:Milnor}).
 
If $f$ starts with $h$ and ends with $g$
(i.e. $\ell_sk_1\neq 0$), then $f^*$ maps $U_h$ strictly inside itself and $(f^*)^{-1}$ maps $U_g$ strictly inside itself. This implies that $f^*$ is  loxodromic, a
contradiction. Therefore $f$ starts and ends with the same letter. If this letter were $h$, then $u_g$ would in fact be mapped into $U_h$ by $f$; thus, 
$f$ starts and ends by a power of $g$. Conjugating $f$ by a power of $g$, we reduce its length without changing the property $f^*u_g=u_g$. 
Thus,  by induction, $f$ is a power of $g$.
\end{proof}

\subsubsection{Preliminaries: disintegration} 
If $g$ is any Halphen twist, %\romain{petit changement, $g$ n'est plus l'élément distingué de la proposition}
we may disintegrate $\mu$
with respect to $\pi_g$: there is a measurable family of probability measures $\mu_{g,w}$
on the fibers $X_w^g$ such that 
\begin{equation}\label{eq:desintegration}
\int_X \xi(x) \, d\mu(x) 
=\int_{B_g} \int_{X_w^g} \xi(z)\,  d\mu_{g,w}(z)\; d\mu_g(w)
\end{equation}
for every Borel function $\xi\colon X\to \R$. The measures $\mu_{g,w}$ are 
unique, in the following sense: if $\mu'_{g,w}$ is another family of probability measures
satisfying   Equation~\eqref{eq:desintegration}, then $\mu'_{g,w}=\mu_{g,w}$ for $\mu_g$-almost
every $w$. Thus, the measures $\mu_{g,w}$ are invariant under
the action of $g_{\vert X^g_w}$. %\romain{enlevé la dernière phrase}
%We perform a similar disintegration with respect to $\pi_h$, and get measures $\mu_h$ and $\mu_{h,w'}$.

\subsubsection{Proof of Proposition~\ref{pro:smooth_case}: special case}\label{subsub:proof_smooth_case_special}

In this paragraph we  assume that there exists $h\in \Hal(\Gamma)$ such that the pair 
 $(g,h)$ is special in the sense of Definition~\ref{defi:special}, and  let $\Gamma_1 = \langle{g, h}\rangle$.
 Note that $\Gamma_1$ satisfies assumptions~(i) and~(ii') from p.\pageref{assumptions_Gamma}. Let  
 $\mu$ be as in the statement of  Proposition~\ref{pro:smooth_case}. 
Let us show that the conclusions of the proposition hold  under the {\emph{additional 
assumption that  that $\mu$ is $\Gamma_1$-ergodic}} (also in this first part of the proof, the analyticity of the density 
will only be established outside $\STang(\pi_g, \pi_h)$).

Let ${\mathcal{B}}_g$ be a Borel subset of $B_g^\circ$ which is disjoint from $ \Rk_g$ and 
satisfies $\mu_g({\mathcal{B}}_g)>0$.
According to Lemma~\ref{lem:basic_dynamics}, 
the dynamics of $g$ on $X_w^g$ is uniquely ergodic for every $w\in {\mathcal{B}}_g$. Thus, we get: 

\smallskip 

{\bf{Step 1.--}} {\em{If $w\in {\mathcal{B}}_g$, then $g_{\vert X^g_w}$ is uniquely ergodic and $\mu_{g,w}$
is  equal to the Haar measure $\lambda_{g,w}$ on the fiber $X^g_w\simeq \C/\Lat_g(w)$}}.
Here, $\Lat_g(w)=\Z\oplus \Z\tau_g(w)$ and $\lambda_{g,w}= ({\mathrm{Im}}(\tau_g(w)))^{-1}\ii dz\wedge d{\overline{z}}$, 
as in Section~\ref{sec:halphen}.

\smallskip 

{\bf{Step 2.--}} {\em{We have $d_\C(\mu)=2$ and $\mu_g(\Rk_g)=0$}}.  

The first assertion follows directly  from Lemma~\ref{lem:dC} and Lemma~\ref{lem:Tang_DGamma}. 

%Let us prove $\mu_g(\im(\mathrm R_g))=0$.
For the second one, we argue by contradiction, assuming that 
 %Otherwise, 
there is 
 %an open disk $U\subset B_g^\circ$ and 
an analytic arc $\gamma \subset \Rk_g$ such that $\mu_g(\gamma)>0$. Since $d_\C(\mu)>1$, we can shorten $\gamma$ to 
ensure  that it does not contain any critical value of $\pi_g$.
Set  $W_\gamma=\pi_g^{-1}(\gamma)$. Then 
$\mu(W_\gamma)>0$, so  $d_\R(\mu)\leq 3$. %\serge{Je n'ai pas compris l'intérêt de découper $W_\gamma$ en disques plus petits.}
%: indeed, $W_\gamma$ can be written as  a  countable union of embedded
%analytic disks  of dimension at most 3. Let $W\subset W_\gamma$ 
%be a disk of dimension $d_\R(\mu)$ such that $\mu(W)>0$. 
By ergodicity of $\mu$, $\bigcup_{f\in \Gamma_1} f(W_\gamma)$ is a subset of full measure.

Pick  $f\in \Gamma_1$.  If it permutes the fibers of $\pi_g$, then
$f\in g^\Z$ because $(g,h)$ is special, and thus $f(W_\gamma)=W_\gamma$; 
thus $\mu(f(W_\gamma)\cap \pi_g^{-1}({\mathcal{B}}_g))=\mu(\emptyset)=0$ in that case.
Now, suppose $f$ does not permute the fibers of $\pi_g$. Note that if $W_\gamma$ intersects an irreducible curve $C\subset X$ on some non-empty open subset of $C$, then $C$ must be a fiber of $\pi_g$, because
its projection in $B_g$ is locally contained in $\gamma$. 
Thus, if $w\in B_g^\circ$ and $f(W_\gamma)\cap X^g_w$ contains a non-empty open 
subset of $X^g_w$ then  $f^{-1}$ maps  $X^g_w$ into a fiber of $\pi_g$ above $\gamma$, and
this contradicts the fact that $(g,h)$ is special. 
We deduce that $f(W_\gamma) \cap X^g_w$ is contained in a countable union of real analytic submanifolds of dimension $1$ in  $X^g_w$, for
every $w\in B_g^\circ$. In particular  if $w\in {\mathcal{B}}_g$, $\lambda_{g,w}(f(W_\gamma)\cap X^g_w)=0$, and 
we conclude that 
$\mu(f(W_\gamma)\cap \pi_g^{-1}({\mathcal{B}}_g))=0$ in that case too.
%Then $f(W_\gamma) $ is generically transverse to the fibers; more precisely  if  $w\in B_g$, 
%\begin{enumerate}[(a)]
%\item either   $f(W_\gamma) \cap \pi_g\inv(w)$ is contained in a countable union of analytic submanifolds of dimension 1 
%in  $\pi_g\inv(w)$, 
%\item or $f(W_\gamma)\cap \pi_g\inv(w) = \pi_g\inv(w)$. 
%\end{enumerate}
%Case (b) means that $f$ maps some fiber of $\pi_g$ above $\gamma$ to the fiber $\pi_g\inv(w)$, for the only algebraic
%curves in $W_\gamma$ are fibers of $\pi_g$. Since $f$ does not preserve the fibration, this is impossible. 
%In case (1), for every $w\in {\mathcal{B}}_g$, $h(W_\gamma) \cap \pi_g\inv(w)$ is a negligible 
%set  for the conditional measure $\mu_{g,w}=\lambda_{g,w}$.
%On the other hand   the second alternative  only occurs for at most  finitely many fibers: indeed 
%in this case $h\inv(\pi_g\inv(w))$ is an algebraic curve contained in $W_\gamma$, so
% it is contained in  some  fiber of $\pi_g$. Since every smooth fiber of $\pi_g$ determines the fibration (see Theorem~\ref{thm:invariant_fibration}.(2)) and $h$ does not preserve the fibration, 
%  the only possibility   is that this algebraic curve is a component of a singular fiber of $\pi_g$. 
%it is contained in  some  fiber of $\pi_g$. If this happens for an infinite set of fibers $\pi_g\inv(w)$ we would infer that 
%$h$ preserves the fibration, which is not the case. Therefore,  since  $\mu$ gives no
%mass to algebraic curves,   alternative~(2) 
%does not account for   $\mu(h(W_\gamma))$.
Since   $\mu\big(\bigcup_{f\in \Gamma_1} f(W_\gamma)\big) =1$ and simultaneously $\mu(\pi_g^{-1}({\mathcal{B}}_g))>0$ we obtain a 
contradiction, which concludes the proof of the second step. 
 
\smallskip 

{\bf{Step 3.--}} From Step 2 we can suppose $\mu_g({\mathcal{B}}_g)=1$. Let us show that
{\em{$\mu_h:=(\pi_h)_*\mu$ is absolutely continous with respect to the Lebesgue measure. In particular, $h$ satisfies $\mu_h(B_h^\circ\setminus \Rk_h)=1$ too.}}

Let $\Delta\subset B_h$ be a Borel 
set of Lebesgue measure $0$.
Remark~\ref{rem:transport_of_smooth_fiber} shows that if $X^g_w$ is a  smooth fiber of $\pi_g$, then $X^g_w$  is not contracted by $\pi_h$,  
so $\pi_{h\vert X^g_w}\colon X^g_w\to B_h$
is a finite ramified cover, and $\lambda_{g,w}((\pi_{h\vert X^g_w})^{-1}(\Delta))=0$. This shows that 
\begin{equation}
\lambda_{g,w}(\pi_{h}^{-1}(\Delta)\cap X^g_w)=0
\end{equation}
for $\mu_g$ almost every point $w\in B_g$. 
Since $\mu_g(\mathcal{B}_g )=  1$ and $\mu_{g, w}$ coincides with the Haar measure  
$\lambda_{g,w}$ for $\mu_g$ almost every $w\in \mathcal {B}_g$,  
Equation~\eqref{eq:desintegration} implies  $\mu(\pi_h^{-1}(\Delta))=0$; thus $\mu_h(\Delta)=0$, as required.

\smallskip 

{\bf{Step 4.--}} {\emph{The support of $\mu$ is $X$. }} 

Indeed,  $\mu_h$ being absolutely continuous with respect to the Lebesgue measure of $B_h$, $\mu_h( \mathrm R_h)=0$.
 Thus, symetrically, $\mu_g$ is absolutely continuous with respect to the Lebesgue measure of $B_g$, and from Equation~\eqref{eq:desintegration}
we deduce that $\mu$ is absolutely continuous with respect to the Lebesgue measure on $X$. 
If $U$ is an open subset of $B_g$, $X^g_U$ intersects every smooth fiber of $\pi_h$ on a set of positive Haar-measure; thus 
$\mu(X^g_U) = \mu_g(U)$ is positive, and we infer that the support of $\mu_g$ is equal to $B_g$.  Since $ \mu_{g,w}$ is $\mu_g$-almost surely the Haar measure $ \lambda_{g,w}$, we conclude that 
the support of $\mu$ is equal to $X$. 

\smallskip 

{\bf{Step 5.--}} {\emph{The density is analytic outside $\STang (\pi_g, \pi_h)$}}.

 Let $x_0$ be a point of $X\setminus{\STang{(\pi_g, \pi_h)}}$ 
 and $V$ be a small neighborhood of $x_0$ such that 
\begin{itemize}
\item $\pi_g$ and $\pi_h$ are everywhere transverse on  $V$, 
\item $U_g:=\pi_g(V)$  and $U_h:=\pi_h(V)$ are small disks in $B_g $ and $B_h$, respectively,
\item $\pi_g$ (resp. $\pi_h$) is a proper submersion above $U_g$ (resp. $U_h$).
\end{itemize}
In a chart $\Psi_g$ mapping $X^g_{U_g}$ to $U_g\times \R^2/\Z^2$,  Equation~\eqref{eq:desintegration} 
implies that $(\Psi_g)_*(\mu_{\vert X^g_{U_g} })$ is invariant 
under the action of {\emph{all}} vertical translations; the same property holds with respect to  $\pi_h$. Coming back to $X$, these 
translations act analytically and locally transitively on $V$: for every $y$ in $V$ there is a pair of such translations
such that their composition maps $x_0$ to $y$ ($V$ is not invariant under such translations). Following the proof of  \cite[Proposition 1]{Eremenko:1989}, we deduce
that $\mu$ has a real analytic density on $V$, with respect to the analytic structure of~$X$.
Indeed, embed $V$  in $\R^4$ and denote by $\vol$ 
a volume form on $\R^4$. In these coordinates, $\mu=\xi \, \vol$, for some non-negative integrable function $\xi$. 
Changing $x_0\in V$ if necessary, we suppose that $x_0$ is a Lebesgue density
point for~$\mu$; this means that $\mu(\e  K)\vol(\e  K)^{-1}$  converges to  $\xi(x_0)$ when $\e\to 0$,
 for every ellipsoid $K$ centered at $x_0$. If $x_0$ is mapped to $y$
by a diffeomorphism $\varphi_y$ preserving $\mu$, then the boundedness of the distortion of $\varphi_y$ shows that 
$y$ is also a Lebesgue density point of $\mu$, 
with density $\xi(y)=(\jac(\varphi_y)(x_0))^{-1}\xi(x_0)$. Now, choosing $\varphi_y$ as a composition of two translations, 
we can assume that $y\mapsto \jac(\varphi_y)(x_0)$ is an analytic function; thus, the density $\xi$ is real-analytic in a neighborhood of $x_0$.\qed
 
%Now, for every point $x\in X\setminus \STang_{\Gamma_1}$ there is a pair of elements $(g',h')\in \Hal(\Gamma)^2$, depending on $x$, %that satisfies the above transversality assumption. 
%to which the above reasoning can be applied. 
%So, the density of $\mu$ is real analytic on the complement of $\Tang_{\Gamma_1}$, and we are done.\hfill $\square$

\subsubsection{Proof of Proposition~\ref{pro:smooth_case}: general case}\label{subsub:proof_smooth_case_general}

Let $\mu$ and $g$ be as in the statement of the proposition. Fix $h\in \Hal(\Gamma)$ such that 
$d\pi_g\wedge d\pi_h\not\equiv0$. Then by Lemma~\ref{lem:schottky-parabolic}, we find $n\geq 1$ such that $(g^n, h^n)$ is special; set $\Gamma_1 = \langle g^n, h^n\rangle$. Since
$\Rk_{g^n} = \Rk_g$, the assumption~\eqref{eq:mug_positive} holds with $g^n$ instead of $g$. 
However $\mu$ is not necessarily $\Gamma_1$-ergodic, 
so the results of \S~\ref{subsub:proof_smooth_case_special} cannot be applied directly. To get around  this 
difficulty, we use the  $\Gamma_1$-ergodic decomposition of $\mu$ (see~\cite{varadarajan}): $\mu=\int_X \beta_x \; d\mu(x)$ 
for an essentially unique, $\Gamma$-invariant, Borel map 
$\beta:x \mapsto \beta_x$ such that $\beta_x$ is $\mu$-almost surely a $\Gamma_1$-ergodic probability measure.   

Set $\Omega_j=\{ x\in X\; ; \; d_\C(\beta_x)=j\}$, for $j=0,1$.
By ergodicity, $\beta_x(D_{\Gamma_1})=1$ for every $x\in \Omega_1$. Since $\mu(D_{\Gamma_1})=0$, we deduce that $\mu(\Omega_1)=0$.  For $x\in \Omega_0$, $\beta_x$ gives full mass 
to the union of the (fixed) countable set $\pi_g\inv(\Tor(B_g^\circ))\cap \pi_h\inv(\Tor(B_h^\circ))$ and the Zariski closed set $\Sing(\pi_g)\cup \Sing(\pi_h)$; since $d_\C(\mu)=2$, we also get $\mu(\Omega_0)=0$. Thus $\dim_\C(\beta_x)=2$ for $\mu$-almost every $x$.  

So, for $x$ in a subset $\Omega\subset X$ with $\mu(\Omega)=1$, the results of \S~\ref{subsub:proof_smooth_case_special}
apply to $\beta_x$. There are two possibilities: 
\begin{itemize}
\item either $\beta_x(\pi_g\inv(B_g^\circ\setminus \Rk_g))>0$, $\beta_x$ is absolutely continuous, and its support is $X$;
\item or $\beta _x(\pi_g\inv(\Rk_g))=1$. 
\end{itemize}
Denote by $\Omega_\mathrm{ac}$ (for ``absolutely continuous'') and by $\Omega_\mathrm{si}$ (for ``singular'') the set of points such that the first or second alternative holds, respectively. 
Both  are Borel subsets and  
$\mu(\Omega_\mathrm{ac}\cup \Omega_\mathrm{si}) = 1$.  
Assumption~\eqref{eq:mug_positive} implies that $\mu(\Omega_\mathrm{ac})>0$.
If $\mu(\Omega_\mathrm{si})>0$, then $\mu(\pi_g\inv(\Rk_g))>0$, and since $\mu$ is 
$\Gamma$-ergodic, we infer that $\mu(\Gamma\cdot \pi_g\inv(\Rk_g))=1$. But the Lebesgue measure of 
$\Gamma\cdot \pi_g\inv(\Rk_g)$ is zero, so $\int_{x\in \Omega_ \mathrm{ac}} \beta_x(\Gamma\cdot \pi_g\inv(\Rk_g)) \; d\mu(x)=0$ and we deduce that $\mu(\Omega_\mathrm{ac})=0$. This contradiction shows that $\mu(\Omega_\mathrm{si})=0$ and  $\mu(\Omega_\mathrm{ac})=1$.

Finally, if  $x$ and $y$ are in $\Omega_\mathrm{ac}$, then
$\beta_x$ and $\beta_y$ are $\Gamma_1$-ergodic measures 
of full support, the densities of which are analytic on the complement of proper analytic sets. 
So $\beta_x=\beta_y$, and this implies that $\beta_x = \mu$ almost surely. In particular, $\mu$ is $\Gamma_1$-ergodic and satisfies the conclusions of 
 \S~\ref{subsub:proof_smooth_case_special}. 
 
 \begin{rem}\label{rem:any_Halphen}
This argument shows that if $\mu$ is a $\Gamma$-ergodic probability measure  $\mu$ gives no mass to 
algebraic subsets, and  $\mu_g(B_g^\circ\setminus \Rk_g) >0$ for some $g\in \Hal(\Gamma)$, then 
$\mu_{g'}(B_{g'}^\circ\setminus \Rk_{g'}) =1$ for \emph{any} $g'\in \Hal(\Gamma)$: this follows from the absolute continuity of~$\mu$.
\end{rem}

 To conclude the proof of  
 Proposition~\ref{pro:smooth_case}, it remains to show that  the 
 density of $\mu$ is analytic outside $\STang_\Gamma$. For every   $x\in X\setminus \STang_\Gamma$ 
 there exists a pair  $(g',h')\in \Hal(\Gamma)^2$ such that 
  $x\notin \STang(\pi_g, \pi_h)$,
 %$d\pi_{g'}\wedge d\pi_{h'}\neq 0$ in a neighborhood of $x$, 
 and by Lemma~\ref{lem:schottky-parabolic} we can assume that this pair is special. By the previous remark, the results of  \S~\ref{subsub:proof_smooth_case_special} apply to $(g',h')$. So,   
 $\mu$ is smooth near $x$, and we are done.  \qed

%So, the density of $\mu$ is real analytic on the complement of $\Tang_\Gamma$, and we are done

%
%\subsubsection{K3 and Enriques surfaces} 
%\textcolor{red}{Dans ce cas, la densité est lisse au sens orbifold, et en fait c'est la mesure donnée
%par la forme volume (voir section dédiée ci-dessous, ou le mettre ici ?)}\romain{je vote pour la section dédiée mettre un pointeur}\serge{OK}
%
\subsection{Dimension $\leq 1$}

\begin{lem}\label{lem:dR1}
If $d_\R(\mu)\leq 1$ then $\mu$ is either supported on   a finite orbit, or on $\mathrm D_\Gamma$.
\end{lem}

\begin{proof}  By Lemma~\ref{lem:dC}, we may assume that $d_\C(\mu)=2$, in particular $\mu_g$ is atomless, and we seek a contradiction. Pick $g\in \Hal(\Gamma)$.
By assumption  there is  a (local) real analytic curve $W\subset X$ 
  such that $\mu ( W)>0$ and $W$ is Zariski dense in $X$. 
Shrinking it, we can assume that $W$ is an analytic path,   transverse to the fibration $\pi_g$, that
 intersects each fiber above $\pi_g(W)$
in a unique point. By Proposition~\ref{pro:smooth_case},  $\pi_g(W)$ 
is contained  in $\Rk_g$. 
Let us disintegrate $\mu$ with respect to $\pi_g$. Since $\mu_g$ is atomless, Lemma~\ref{lem:basic_dynamics} shows that, for $\mu_g$-almost every  $w$, the measures which are invariant under $g_w$  are atomless; thus, the conditional 
measure $\mu_{g,w}$ is almost surely atomless  and $\mu_{g,w}(W\cap X^g_w)=0$. This contradicts $\mu(W)>0$. 
\end{proof}

\subsection{The totally real case}

We are now reduced to the case where $d_\R(\mu)\geq 2$ and $\mu_g$ 
is supported on $\Rk_g$ for  every $g\in \Hal(\Gamma)$. 
The properties of $\mu$ are summarized in the following proposition, the proof of which will 
be given in Section~\ref{sec:Main_Proof}. 

%\begin{pro}\label{pro:analytic_sigma} Assume that $\mu$ is not supported on a proper algebraic subset of~$X$, and that
%$\mu_g(\Rk_g)>0$ for some $g\in \Hal(\Gamma)$. Then $d_\R(\mu)=2$, and
%there is  a closed semi-analytic set  
%$\Sigma\subset X$ of pure dimension 2 such that 
%\begin{enumerate}[\em (1)]
%\item $\mu(\Sigma)=1$ (i.e. $\mu$ is supported on $\Sigma$);
%\item $\pi_h(\Sigma)\cap B_h^\circ \subset  {\Rk_h}$ for every $h\in \Hal(\Gamma)$;
%\item outside $\mathrm D_\Gamma$ the singular locus of $\Sigma$ is a finite subset of $\Tang_\Gamma^0$;  \serge{attention ausi à l'orbite de $\Sigma$ par $\Gamma$; il faudra écrire un lemme supplémentaire pour gérer les intersections $f(\Sigma)\cap \Sigma$}
%\item on the regular part of $\Sigma$, $\mu$ is given by a real analytic density with respect to any real analytic area form;
%\item $\Sigma$ is totally real.
%\end{enumerate}
%In particular, $\Sigma$ is the support of $\mu$ and is $\Gamma$-invariant. 
%\end{pro}

\begin{pro}\label{pro:analytic_sigma}  
Let $\mu$ be an ergodic $\Gamma$-invariant measure. 
Assume that $\mu$ %is not supported on a proper algebraic subset of~$X$, 
gives no mass to algebraic subsets and that
$\mu_g(\Rk_g)>0$ for some $g\in \Hal(\Gamma)$. Then $d_\R(\mu)=2$, and
there exists 
%a $\Gamma$-invariant proper algebraic subset $Y\subset X$ and 
  a  totally real  analytic subset  
$\Sigma$ of $X\setminus  \STang_\Gamma$  
 of pure dimension 2 such that: 
\begin{enumerate}[\em (1)]
%\item the 1-dimensional part of $Y$ coincides with  $\mathrm D_\Gamma$;
\item $\mu(\Sigma)=1$ and  $\supp(\mu) = \overline \Sigma$; 
\item $\Sigma$ has finitely many irreducible components; 
%\item $\pi_h(\Sigma)\cap B_h^\circ \subset  {\Rk_h}$ for every $h\in \Hal(\Gamma)$;
\item the singular locus of $\Sigma$ is locally finite;  %\romain{je vois pas pourquoi ce serait fini}\serge{Cela va résulter de $NT_f$ fini, non? Les problèmes ne viennent pas que de NT mais aussi des auto-intersections}
\item on the regular part of $\Sigma$, $\mu$ has a 
real analytic density with respect to any real analytic area form;
%\item $\Sigma$ is totally real.
\end{enumerate}
%In particular, $\Sigma$ is the support of $\mu$ and is $\Gamma$-invariant. 
\end{pro}

\begin{eg} The main example 
 is when the projective surface $X$ is defined over $\R$, and  $\mu$ is a measure 
 supported on $X(\R)$ giving no mass to algebraic subsets. 
  This  occurs for instance when $\Gamma\subset \Aut(X_\R)$ 
preserves an area form on $X(\R)$ 
(see Section \ref{sec:examples}). An example of a different kind 
 is given in Section~\ref{sec:not_real_parts}.  
\end{eg}

 A noteworthy consequence of Propositions~\ref{pro:analytic_sigma} and~\ref{pro:smooth_case} is:

\begin{cor}\label{cor:dimension}
The dimension $d_\R(\mu)$  cannot be equal to 3. 
\end{cor}

 \subsection{Proof of the main theorem, and consequences} \label{par:proof_main_Hal}

\begin{proof}[Proof of Theorem~\ref{thm:main}] 
If $\mu$ gives positive mass to  a proper algebraic subset of $X$, then according to 
Lemma~\ref{lem:dC}, either Assertion~(a) or~(b) of Theorem~\ref{thm:main} is satisfied.  Otherwise, 
 we know from Lemma~\ref{lem:dR1} and Corollary~\ref{cor:dimension} that $d_\R(\mu)$
is either equal to $2$ or $4$, and exactly one of Proposition~\ref{pro:analytic_sigma} or Proposition~\ref{pro:smooth_case} applies. 
If $d_\R(\mu)=2$, Proposition~\ref{pro:analytic_sigma} 
shows that the conclusions of case~(c) of the theorem hold.  
The remaining case~(d) is covered by Proposition~\ref{pro:smooth_case}. In both cases 
the exceptional set   $Z$ is equal  to 
 $\STang_\Gamma$.
\end{proof}

\begin{proof}[Proof of Corollary~\ref{cor:no_invariant_subset}] 
When $\Gamma$ does not preserve any proper algebraic subset, then we must be in one of the 
cases~(c) (if $d_\R(\mu)=2$) or~(d) (if $d_\R(\mu)=4$) of Theorem \ref{thm:main}. 
Furthermore %the respective exceptional sets $Y$ and $Z$ are empty. 
 $\STang_\Gamma$ is empty.
Thus if  $d_\R(\mu)=4$ we infer that $\mu$ has an analytic  positive density on $X$. If 
$d_\R(\mu)=2$, note that the analytic surface $\Sigma$ is smooth everywhere, 
since otherwise its singular locus 
would be a finite invariant subset (see Assertion~(3) of Proposition~\ref{pro:analytic_sigma}). 
% $\mathrm D_\Gamma$ and  
%$\Tang_\Gamma$ are empty, and so is ${\mathrm{Sing}}(\Sigma)$  in the case $d_\R(\mu)=2$ are all empty.\serge{Pour ce dernier argument il faut $\Sing(\Sigma)$ fini. Mais ici c'est plus simple (pas d'accumulation sur des courbes invariantes)} 
%If $d_\R(\mu)=4$, Proposition~\ref{pro:smooth_case} shows that Assertion~(b) is verified. From Lemma~\ref{lem:dR1}, 
%Proposition~\ref{pro:analytic_sigma}, and the absence 
%of invariant algebraic subset, we deduce that $\mu$ is supported on a smooth compact real analytic surface $\Sigma$, and that
%the the density of $\mu$ on $\Sigma$ is smooth. Thus, Assertion~(a) is satisfied.
\end{proof}

Let us point out the following immediate consequence of Theorem~\ref{thm:main}.  

\begin{cor}
If $X$ is a real projective surface, $X(\R)$ is nonempty, $\Gamma\subset \Aut(X_\R)$, and $\mu$ is a $\Gamma$-invariant ergodic probability measure on $X(\R)$, then 
\begin{enumerate}[\em (a)]
\item either $\mu$ is supported by a $\Gamma$-invariant proper real algebraic subset of $X$, 
\item or there is a $\Gamma$-invariant proper real algebraic subset $Z$ of $X$ such that $\mu$ is supported by a union of connected components of $X(\R)\setminus Z(\R)$ and $\mu$ has a real analytic density on that set.
\end{enumerate} 
In particular if $\Gamma$ 
does not preserve any proper real algebraic subset of $X$, $\mu$ is given by a real analytic area form on $X(\R)$, restricted
to a $\Gamma$-invariant union of connected components of $X(\R)$. \end{cor}

%\romain{nouveau paragraphe qui remplace l'ex section 8 (voir aussi la remarque \ref{rmk:real_area_form})}
Let $X$ be (a blow up of) an abelian,  K3, or Enriques surface, and let $\vol_X$ be the natural probability
 measure on $X$ (see~\ref{par:examples_K3surfaces}). Likewise, if $\Sigma$ is a 
totally real submanifold in $X$, let $\vol_\Sigma$ be the   measure  induced by
 the normalized 2-form $\Omega_X$ (see Remark~\ref{rmk:real_area_form}). 
By   positivity of the density of $\mu$ and ergodicity, we obtain: 
 
 \begin{cor}
 If in Theorem~\ref{thm:main}, $X$ is a blow-up of an abelian, K3, or Enriques surface, then in case~(c) $\mu$ is a multiple of $\vol_\Sigma$, and in case~(d) it coincides with $\vol_X$.  
 \end{cor}

 In case~(c), this implies in particular that $\Sigma$ has finite area.  Similarly,   case~(d) does not appear in 
 Blanc's example (see \S~\ref{par:examples_rational}). 

%\subsection{Remarks on the   classification of orbit closures}
%Under the assumptions of Theorem~\ref{thm:main},  a  classification of orbit closures should be achievable 
% along the same lines. Let $x\in X$ be such that $\Gamma\cdot x$ is Zariski dense. If there exists 
% $g\in \Hal(\Gamma)$ such that $x\notin \pi_g\inv(\mathrm R_g)$, then it is easy to show exactly as in 
% \S~\ref{subs:smooth_case} that $\Gamma\cdot x$ is dense in $X$. Otherwise we know that 
% $\overline {\Gamma\cdot x}$ is a closed subset contained in 
% $\bigcap_{g\in \Hal(\Gamma)} \pi_g\inv(\mathrm R_g)$, 
% so outside $\STang_\Gamma$ it is contained in a countable union of analytic totally real surfaces. As in 
% Proposition~\ref{pro:analytic_sigma}, the difficulty is   to show that this union is locally finite. Since 
% the proof of  Proposition~\ref{pro:analytic_sigma} makes use of 
%  the measurable structure, in particular of the notion of 
% ergodicity, a non-trivial adaptation is necessary.   
% Still, we expect that the general scheme of the proof carries over to this  topological setting.  

%%%%%%%%%%%%%%%%%%%%%%%%%%%%%%%%%%%%%%%%%%%%%%
%%%%%%%%%%%%%%%%%%%%%%%%%%%%%%%%%%%%%%%%%%%%%%
\section{Proof of Proposition \ref{pro:analytic_sigma}}\label{sec:Main_Proof}
%%%%%%%%%%%%%%%%%%%%%%%%%%%%%%%%%%%%%%%%%%%%%%
%%%%%%%%%%%%%%%%%%%%%%%%%%%%%%%%%%%%%%%%%%%%%%

In this section, we prove Proposition~\ref{pro:analytic_sigma}. Thus, unless otherwise specified, we study 
ergodic invariant measures with $d_\C(\mu)=2$ and 
$\mu_g(\gamma)>0$ for some $g\in \Hal(\Gamma)$ and some local, smooth, analytic arc $\gamma\subset \Rk_g$. 
The main step of the proof  is to show that $\gamma$ extends to 
an  analytic curve $\sigma$ in  $B_g^\circ$.  From Proposition~\ref{pro:smooth_case}, we already know  that 
the support of $\mu_g$ is contained in the closure of $\Rk_g$; however the analytic continuation of
 $\gamma$ and the support 
of $\mu_g$ still could a priori be a complicated subset  in $B_g$ (see Remark~\ref{rem:saturation}). 
The main point is  to exclude such  a phenomenon.  
%Unless $\pi_g$ be isotrivial, we shall also obtain a finiteness result for the possible curves $\sigma$: 
%this will be the main ingredient towards the proof of Theorem~\ref{thm:finiteness}.

After some  preliminaries in \S\S~\ref{par:semi-analytic-vocabulary} and \ref{par:preliminaries_conventions}, 
the proof of Proposition~\ref{pro:analytic_sigma} is carried out in 
\S\S~\ref{par:geometry_of_g-orbits} to~\ref{par:proof_pro_analytic_sigma}. 
%In \S\ref{par:complements} we 
%study the extendability of $\sigma$ (resp. $\Sigma$) to a semi-analytic curve in $B_g$ (resp. a semi-analytic surface in $X$). 
 
%%%%% 
\subsection{Vocabulary of analytic and semi-analytic  geometry}\label{par:semi-analytic-vocabulary} 
%%%%%
Let us recall  briefly a few basic facts from analytic and 
semi-analytic geometry.  
A {\bf{semi-analytic}}  set $E$
in a ($\sigma$-compact) real-analytic manifold  $M$
is a subset such that   every $x\in M$ admits an open neighborhood $U$ in which  $E\cap U$ 
is   defined by finitely many inequalities of the form $f\geq 0$ or $f>0$, where  $f\colon U\to \R$ is analytic. 
The class of semi-analytic sets is stable under many operations such as taking a finite union or intersection, 
the closure, the boundary, the connected components, or the preimage under an analytic map. 
However the image of a semi-analytic set 
by a proper analytic map needs not be semi-analytic: adding such projections, one obtains the class of {\bf{subanalytic sets}}. 
The main fact  that we will need from  subanalytic geometry is: 
{\em any subanalytic set of dimension $\leq 1$ is semi-analytic}. 
In this way we will only have to deal with semi-analytic sets 
(see \cite{bierstone-milman, lojasiewicz} for more details on these facts). 
 
 %Semi-analytic sets can be {\em stratified}, that is 
 %a semi-analytic set $E$ can be written as a disjoint union $E = \bigcup A_i$ 
% where each $A_i$ is a semi-analytic set and an analytic manifold. 
 A point $x$ in $E$ is {\bf{regular}} if $E$ is an analytic submanifold 
 in the neighborhood of $x$, otherwise $E$ is 
  {\bf{singular}} at $x$.  The 
 {\bf {dimension}} of $E$ is the maximal dimension of $E$ at its regular points. %;  it is also the maximal dimension of  its strata.
 We say that $E$ is of pure dimension if its dimension is the same at all regular points. 
 % A semi-analytic set of dimension $k$  has locally  finite $k$-dimensional Hausdorff measure (see~\cite{herrera} for instance). 
 %\romain{peut être pas ref standard}\serge{Je ne sais pas}\romain{je ne sais pas si on a encore besoin de cette remarque de toute façon}) 
 
  A delicate point in semi-analytic and real-analytic geometry is that the notion of irreducible component is not 
  well behaved (we will discuss only the analytic case). To get around this problem, the following notion was introduced by Cartan~\cite{cartan}
  and Whitney and Bruhat~\cite{whitney-bruhat}: a subset $E$ of a real analytic manifold $M$ 
  is \textbf{C-analytic}
   (or global
   analytic) if it the common zero set of a finite number of analytic functions  defined on the whole of $M$. 
   Equivalently, there is a coherent analytic sheaf whose zero set is $E$ (see~\cite[Prop. 10]{whitney-bruhat}). 
   This class is stable by union and  intersection.  Every C-analytic set $E$ admits a unique 
   locally finite    decomposition $E = \bigcup_i E_i$ into irreducible components;
   here, the  irreducibility means that $E_i$ is not the union of  two distinct C-analytic sets (beware that it might be 
   reducible as an analytic set, see~\cite[\S 11.a]{whitney-bruhat}).  
If $E\subset M$   is a   smooth analytic submanifold, or if $E$ is 
  locally a finite union of smooth plaques of the same 
dimension, then $E$ is C-analytic. Indeed, in this case it is easy to see that 
 for every $x\in E$ there exists an open neighborhood
$U\ni x$ and a finite family of analytic functions $f_i$ on $U$ such that for every open subset $V\subset U$, the intersection of the zero sets of the $f_i$ coincides with $E$ in $V$: this implies that $E$ is the zero set of a 
coherent analytic sheaf, hence it is C-analytic. 
Another useful fact is that every  1-dimensional analytic set is C-analytic; more generally, if $E$ is analytic, 
  the canonical ideal sheaf of analytic functions vanishing on $E$ is coherent 
  outside a codimension 2 subset of $E$,
(see~\cite{galbiati}).

 A semi-analytic subset $\Sigma\subset X$ of pure dimension $2$ is {\bf{totally real}} if  
 at every regular point  $x\in \Sigma$, $T_x\Sigma$ is a totally real subspace of $T_xX$, that is, 
 ${\mathsf{j}}_x(T_x\Sigma)\oplus_\R T_x\Sigma=T_xX$, where ${\mathsf{j}}_x$ is the complex structure (multiplication by $\sqrt{-1}$).

%%%%%
\subsection{Preliminaries and conventions} \label{par:preliminaries_conventions}
%%%%%
%%%
\subsubsection{Choice of Halphen twists}\label{par:choice_halphen_twist} 

Recall from Remark~\ref{rem:any_Halphen} that under the assumptions of 
Proposition~\ref{pro:analytic_sigma}, $\mu_g(\Rk_g)=1$ for any $g\in \Hal(\Gamma)$. 
We fix a pair of elements $(g,h)\in \Hal(\Gamma)^2$ %such that $(g,h)$ is special in the sense of 
associated to   different    fibrations. %\romain{je retire la référence à spécial et conjugué, qui ne sert semble-t-il à rien}
This pair will be kept fixed until 
Subsection~\ref{par:proof_pro_analytic_sigma} .
Recall from   Notation~\ref{nota:Hal}, that we denote by $X^g_w$ the fiber $\pi_g\inv(w)$, %we use similar notations for $h$ and the other $g_i$.  
and similarly for~$h$. 

\subsubsection{A decomposition of the tangency locus}\label{par:tangencies}
 By definition, the tangency locus $\Tang(\pi_g,\pi_h)$ is the locus where the map $(\pi_g,\pi_h)\colon X \to B_g\times B_h$
is not a local diffeomorphism; $\Tang(\pi_g,\pi_h)$ is a curve that contains all multiple components of fibers of $\pi_g$ and $\pi_h$, 
as well as the curves along which the foliations determined by these fibrations are tangent. To be more precise, we split $\Tang(\pi_g,\pi_h)$ 
into four parts, 
\begin{equation}
\Tang^{\mathrm {ff}}(\pi_g, \pi_h)\cup \Tang^{\mathrm {ft}}(\pi_g, \pi_h) \cup \Tang^{\mathrm {tf}}(\pi_g, \pi_h) \cup 
\Tang^{\mathrm {tt}}(\pi_g, \pi_h)
\end{equation}
which are defined as follows. An irreducible component $C$ of $\Tang(\pi_g,\pi_h)$ is 
\begin{itemize}
\item contained in $\Tang^{\mathrm {ff}}(\pi_g, \pi_h)$ if and only if $C$ is contained in a fiber of $\pi_g$ and in a fiber of $\pi_h$ (in that
case, $C$ is both $g$ and $h$-invariant, and its self intersection  is negative, see \S~\ref{par:invariant_curves}); 
\item contained in $\Tang^{\mathrm {ft}}(\pi_g, \pi_h)$ if and only if $C$ is a multiple component of a fiber of $\pi_g$ but is generically transverse to $\pi_h$;  
\item contained in $\Tang^{\mathrm {tf}}(\pi_g, \pi_h)$ if and only if $C$ is a multiple component of a fiber of $\pi_h$ but is generically transverse to $\pi_g$;
\item contained in $\Tang^{\mathrm {tt}}(\pi_g, \pi_h)$ if it is generically transverse to both fibrations.
\end{itemize}
The superscripts $\mathrm f$, $\mathrm t$ stand for fiber and transverse.

\begin{lem}\label{lem:isotrivial}
If $\Tang^{\mathrm{tt}}(\pi_g,\pi_h)$ is empty, the fibrations $\pi_g$ and $\pi_h$ are both isotrivial. 
\end{lem}

The isotriviality of $\pi_g$ means that the $j$-invariant of the fibers $X^g_w$ is constant on $B_g^\circ$.  
 In this case, the discussion of \S~\ref{par:type-Ib} shows that  no fiber of $\pi_g$ is 
of type $I_b$ or  $I_b^*$, $b\geq 1$,   and that after a finite
base change, $\pi_g$ becomes birationnally equivalent to a trivial fibration. 

\begin{proof}
If $\Tang^{\mathrm{tt}}(\pi_g,\pi_h)$ is empty, the foliation associated to $\pi_h$ is transverse to $\pi_g$ on $\pi_g^{-1}(B_g^\circ)$ (though a 
 multiple fiber of $\pi_h$  may  intersect every fiber of $\pi_g$).  If $\beta\colon [0,1]\to B_g^\circ$ is 
a smooth path, the holonomy of this foliation determines a holomorphic diffeomorphism ${\mathrm{hol}}_\beta\colon X_{\beta(0)}^g\to 
X_{\beta(1)}^g$. Thus, all fibers of $\pi_g$ are isomorphic, and likewise for $\pi_h$. 
\end{proof}

\subsection{Geometry of $g$-orbits}\label{par:geometry_of_g-orbits}
 Let us fix a K\"ahler metric on $X$, given by a K\"ahler form $\kappa$, as 
well as a K\"ahler form $\kappa_g$ on $B_g$ (hence also on $B_h$, see \S~\ref{par:choice_halphen_twist}). Lengths, areas, and diameters will be computed with respect to these metrics.

According to the Notation~\ref{nota:L_f}, a {circle} in an elliptic curve is  a translate of a $1$-dimensional, 
closed, and connected subgroup.  
Being invariant under translation, this notion is well defined on $X_w^g$, for every $w\in B_g^\circ$. %, independently of the choice of a local section. 
If $w\in \Rk^k_g$ and 
$z\in X_w^g$, the closure of its $g$-orbit is a union of $k$ circles; the circle containing $z$ is denoted $L_g^0(z)$. In the next lemma we use the notation introduced in Section~\ref{par:the_dynamics_R}; the norm of a slope $(p,q)$ is, by definition $\norm{(p,q)}=(p^2+q^2)^{1/2}$.
Since circles are homotopically non-trivial and long circles become asymptotically dense, we get:

\begin{lem}\label{lem:circles_are_epsilon_dense}
Let $U\Subset B_g^\circ$ be a disk, endowed with a continuous choice of basis for $H_1(X_w^g;\Z)$ and a local section of $\pi_g$.  
\begin{enumerate}[\em (1)]
\item  There is a real number $\ell(U)$ such that the length of every circle
of every fiber $X_w^g$, for $w\in U$, is bounded from below by~$\ell(U)$.
\item For every $\e  >0$, there is a real number $D>0$ such that for all $(p,q)$ with  $\norm{(p,q)}>D$, and all $w\in U$,  every circle of 
slope $(p,q)$ is $\e$-dense in $X_w^g$.
In particular, for every $\e  >0$, there is a real number $D>0$ such that  if $\norm{(p,q)}>D$, and 
 $w\in \Rk^k_{p,q}(U)\setminus \Tor(U)$ for some $k$ then  the circle $L_g^0(z)$ is $\e$-dense in $X_w^g$ for every $z\in X_w^g$.
 \item For every $\e  >0$, there is an integer $k_0>0$ such that for every $k\geq k_0$, every $w\in \Rk^k_{p,q}(U)\setminus \Tor(U)$, and every 
 $z\in X_w^g$, the orbit closure $L_g(z)$ is $\e$-dense in $X_w^g$.
 \end{enumerate}

 \end{lem}

As a consequence, if  $K$ is a compact subset of $B_g^\circ$,   there is a real number $\ell(K)$ such that the length of every circle
of every fiber $X_w^g$, for $w\in K$, is bounded from below by~$\ell(K)$. Another consequence is:

\begin{lem}\label{lem:circle_escapes}
For every $\delta_1>0$ there exists $\delta_2>0$ such that if 
$w\in B_g$ is $\delta_2$-far from $\Crit(\pi_g)$, any circle in $X_w^g$   escapes the  $\delta_1$-neighborhood of 
 $\Tang(\pi_g, \pi_h)$. 
\end{lem}

%Let us now study the local geometry of the projections of these circles by $\pi_h$. 
 Let $\eta$ denote half of the injectivity radius of 
the metric $\kappa_g$. For every $w\in B_g$, the (riemannian) exponential map is a diffeomorphism from the disk of 
radius $\eta$ in $T_wB_g$ to some open subset of $B_g$.
By definition, the diameter of an interval $J\subset T_wB_g$ is its length 
with respect to $\kappa_{g,w}$; its radius is half its diameter. 

 Now, let $I\subset B_g$ be a smooth real analytic arc, and let 
$w$ be a point of $I$. The tangent direction to $I$ at $w$ determines an orthogonal decomposition of $T_wB_g$ into 
a direct sum $T_wI\oplus_\R (T_wI)^\perp$. Let
$r$ be a positive number $\leq \eta$.
We say  that $I$ {\bf{is of size}} (at least) $r$ at 
$w$  if its preimage in $T_wB_g$ by the exponential map contains the graph $\{s+\psi(s)\; ; s\in J\}$ of a function $\psi\colon J\subset  T_wI \to (T_wI)^\perp$ such that:
\begin{enumerate}[(i)]
\item $\psi$ is defined on  an interval $J$
of radius $r$ around $0$ in $T_wI$;
\item $\psi(0)=0$,  its first   derivative $\psi'$ is bounded by $1$, and its  second derivative $\psi''$ is bounded by $1/r$.
\end{enumerate} 
Note that since  $\psi'$ is bounded by $1$, $\psi$  takes its values in an interval of length $\leq r$ in $(T_w I)^\perp$. 

This definition 
is scale invariant. Similar notions can be defined on $B_h$, or  on $X$, or along the fibers $X_w^g$. When 
talking about the size of an arc, it should be clear from the context whether we are working in $X$,   $B_g$, or  $B_h$.  
  
 Since circles on a torus are geodesics for the  flat metrics, we get: 
  
\begin{lem}\label{lem:circle_size}
For every $\delta>0$ there exists $r_1 = r_1(\delta)>0$ such that if  $w\in B_g$ is such that 
 $\dist(w,\Crit(\pi_g))>\delta$, then any circle of $X_w^g$ is of size  at least  $r_1$ at any of its points. 
 \end{lem}

The ramification points of the restriction of $\pi_g$ to the leaves of the foliation induced by $\pi_h$ are located in 
$\mathrm{Tang}^{\mathrm{tt}}(\pi_g, \pi_h)\cup \mathrm{Tang}^{\mathrm{ft}}(\pi_g, \pi_h)\cup 
 \mathrm{Tang}^{\mathrm{ff}}(\pi_g, \pi_h)$. This implies:

\begin{cor}\label{cor:circle_size}
For every $\delta>0$ there exists $r_2 = r_2(\delta)>0$ such that if $\xi\in X$ is $\delta$-far from $\mathrm{Tang}^{\mathrm{tt}}(\pi_g, \pi_h)\cup \mathrm{Tang}^{\mathrm{ft}}(\pi_g, \pi_h)\cup 
 \mathrm{Tang}^{\mathrm{ff}}(\pi_g, \pi_h)$, then the projection under $\pi_g$ of any circle in 
$X^h_{\pi_h(\xi)}$ has size at least $r_2$ at $\pi_g(\xi)$. %\romain{J'avais noté "problème" ici, mais je crois que ça signifie juste "à vérifier"}
\end{cor}

%\serge{No comprendo }\romain{je crois que j'ai juste raconté n'importe quoi}\serge{Don à enlever a posteriori} Observe that in this statement   $\xi$ can be close to $\Ram(\pi_g)$, as soon as $\pi_h$ is regular near $\xi$. This observation will be important in \S\ref{par:gamma_does_not_spiral}. 
Of course   a  similar  result holds by swapping $g$ and $h$.

\subsection{Local structure of $\mu$}
\label{par:local_structure}
Recall that we work under the hypotheses of Proposition~\ref{pro:analytic_sigma}. By assumption, $\mu_g(\Rk_g)$ is positive and  $\mu(\Tang(\pi_g, \pi_h))=0$, because
the tangency locus $\Tang(\pi_g, \pi_h)$  is a proper algebraic subset of $X$. 
And by Remark~\ref{rem:any_Halphen},  
$\mu_g(\Rk_g)= \mu_h(\Rk_h)=1$. %\romain{petit changement}

Pick an analytic arc $\gamma\in \mathrm R_g$ such that $\mu_g(\gamma)>0$; shrinking $\gamma$ if necessary,
we  choose an open subset $U$ as in Section~\ref{par:the_dynamics_R}, and parameters $(\alpha,\beta)$ and $(p,q)$ such that
$\gamma$ is a smooth, real analytic subset of $R^{\alpha,\beta}_{p,q}(g;U)$ diffeomorphic to an interval.
%recall that $\gamma\subset B_g^{\circ\circ}$ is the analytic continuation 
%${\tilde{p}}_g({\tilde{R}}_{(p,q)}^{(\alpha,\beta)})$ of some (connected component of) $R_{(p,q)}^{(\alpha,\beta)}(g;U)$,  with notation as in  Section~\ref{par:the_dynamics_R}. 
Then%\serge{Je change un peu les lignes suivantes, j'aimais pas trop la somme sur des $\gamma'$.}
%\begin{equation}\label{eq:4.4}
%0 < \mu(\pi_g^{-1}(\gamma))= \sum_{\gamma'\in \mathrm R_h}\, 
%\mu(\pi_g^{-1}(\gamma)\cap \pi_h^{-1}(\gamma')) \leq 1;
%\end{equation}
\begin{equation}\label{eq:4.4}
0 < \mu(\pi_g^{-1}(\gamma))= \mu(\pi_g^{-1}(\gamma)\cap \pi_h^{-1}(\Rk_h)) \leq 1;
\end{equation}
consequently, there is an open subset $U'$ of $B_h^\circ$ and a smooth analytic arc $\gamma'$ in some $R^{a',b'}_{p',q'}(h;U')$ 
such that  $\mu(\pi_g^{-1}(\gamma)\cap \pi_h^{-1}(\gamma'))>0$. 
 On the complement of 
$\Tang(\pi_g,\pi_h)$,  the intersection $\pi_g^{-1}(\gamma)\cap \pi_h^{-1}(\gamma')$ 
is transverse; thus reducing $\gamma$, $U$, $\gamma'$ and $U'$ again if necessary, we may assume that 
$(\pi_g, \pi_h)$ is a local diffeomorphism from $X_U^g\cap X_{U'}^h$ to $U\times U'$; in particular the intersection $\pi_g^{-1}(\gamma)\cap \pi_h^{-1}(\gamma')$ 
is transverse. Then, 
the  set 
\begin{equation}
\pi_g^{-1}(\gamma)\cap \pi_h^{-1}(\gamma')= S_1 \sqcup \cdots \sqcup S_l 
\end{equation}
is a disjoint union of  small ``squares'' -- diffeomorphic to $\gamma\times \gamma'$ -- which, for $w\in \gamma$,
 intersect the  fiber $X_w^g$ along pieces of circles (of the same slope). In what follows, 
 we denote by $S$ any one of the  squares $S_j$ such that  $\mu(S_j)>0$; hence 
 \begin{equation}
S\subset \pi_g^{-1}(\gamma)\cap \pi_h^{-1}(\gamma'),  \quad \mu(S)>0, 
\end{equation}
and   $ (\pi_g, \pi_h)\colon S\to \gamma\times \gamma'$ is a diffeomorphism.
Thus,   $d_\R(\mu)\leq 2$, and Lemma~\ref{lem:dR1} implies 

\begin{lem} $d_\R(\mu)=2$. \end{lem}

\begin{figure}[h]
\input{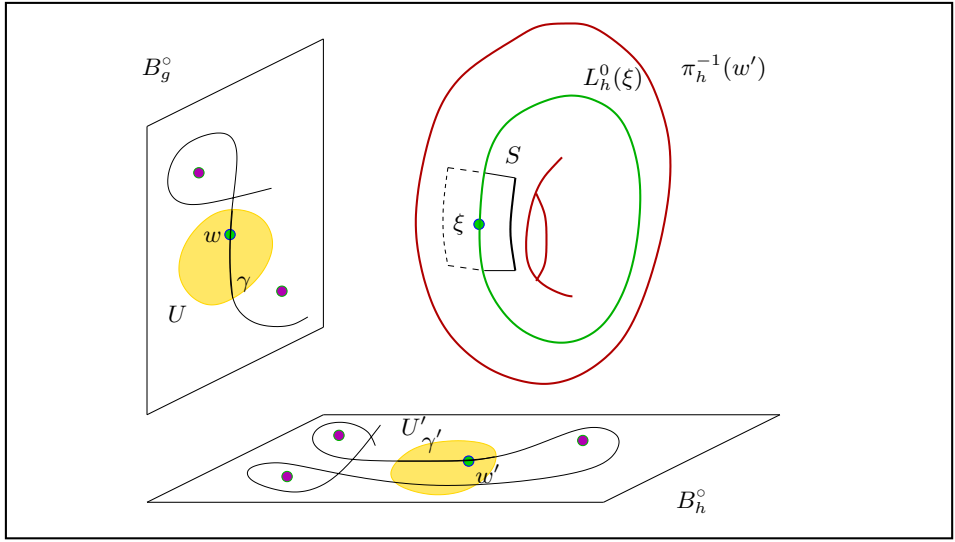}
\caption{{\small{This figure represents (one of) the square(s) $S$. The intersection of $S$ with  $X^h_{w'}$ is contained in the  orbit closure $L_h^0(z)$.
The points $\xi$, $w$, $w'$ are in green; the magenta points are critical values of $\pi_g$ and $\pi_h$. The pre-image of 
$\gamma$ in $X^h_{w'}$ is locally contained in $L_h^0(\xi)$.}}}\label{fig:two_fib}
\end{figure}

We now initiate the study of the analytic continuation of $\gamma$. 
 We say that an arc $J\subset B_g$ is \emph{evenly charged} by $\mu_g$  if  
 $\mu_g(J')>0$ for every non-empty, open interval $J'\subset J$. 
Likewise, we say that a surface $S$ in $X$ is evenly charged by $\mu$ if $\mu(S')>0$ for any 
non-empty relatively open 
 subset $S'\subset S$.    
 
For  $w'\in \gamma'\setminus \Tor(B_h^\circ)$ and $\xi'\in X^h_{w'}$, $L_h(\xi')$   is a finite union of  circles of $X^h_{w'}$, and   
$L_h^0(\xi')$ is invariant under some iterate of $h$. 
For $\mu_h$-almost every $w'$, the conditional measure $\mu_{h,w'}$ is invariant under $h_{w'}$;
as such, it is supported on a union of such orbit closures. 
The set $S\cap X^h_{w'}$ is an interval, which is a piece of the circle $L_h^0(\xi')$, for $\xi'\in S\cap X^h_{w'}$. The set\footnote{We   must take the   closure in Equation~\ref{par:annulus_h} because   $L_h^0(z')$ is reduced to $\{z'\}$ when $\pi_h(z')\in \Tor(B_h^\circ)$.} %\romain{petit changement de notation ici, et distinction entre $L_h(S)$ et  $L_h^0(S)$}
\begin{equation}\label{par:annulus_h}
L_h(S)={\overline{\bigcup_{\xi'\in S} L_h(\xi')}}\ \text{ (resp. }
L_h^0(S)={\overline{\bigcup_{\xi'\in S} L_h^0(\xi')}} \text{ )},
\end{equation}
is a finite union of real analytic annuli -- each of which   a 
circle bundle above $\gamma'$--  which is $h$-invariant % under some positive iterate of $h$  
 (resp.  is an  $h^k$-invariant annulus,  for some $k> 0$). 
For $\xi'\in S$ such that $w':=\pi_h(\xi')\notin \Tor_h(B_h^\circ)$ the restriction of the conditional measure $\mu_{h,w'}$ to $L_h^0(\xi')$ is the Haar measure. Projecting under   the real analytic map $\pi_g$, we get:

\begin{lem}\label{lem:proj_annulus} 
With the above notation, $\pi_g(L_h^0(S))$ is a semi-analytic curve containing $\gamma$, 
which is evenly charged by 
$\mu_g$.  In particular any 
(semi-)analytic continuation of $\gamma$ contains $\pi_g(L_h^0(S))$. 
\end{lem}
  
\begin{rem} \label{rem:saturation}
Pick $\xi'$ in $S$ and set $w'=\pi_g(\xi')$.
If $L_h^0(\xi')$ is disjoint from the  ramification points
of  $\pi_{g}\rest{ X^h_{w'}} \colon X^h_{w'} \to B_g$, 
 then $\pi_g(L_h^0(S))$ is an analytic loop in $B_g$, and 
this loop provides a complete description of the analytic continuation of $\gamma$. 
On the other hand, if $L_h^0(\xi')$ hits a ramification point of $\pi_{g}\rest{ X^h_{w'}}$, 
  its image $\pi_g(L_h^0(\xi'))$ should  be thought of as a segment which is
strictly contained in the analytic continuation of $\gamma$, and whose 
% the extremities of such a segment are projections of the intersection points of $L_h^0(\xi')$ 
%with $\Tang(\pi_g, \pi_h)$; equivalently, they are projections of the ramification points of 
endpoints are contained in the projections of the ramification points. 
%$\pi_{g\vert X^h_{\pi_h(\xi')}}$ contained in $L_h^0(\xi')$. 
As $\xi'$ moves in $S$, the  endpoints and the length of $\pi_{g}(L_h^0(\xi'))$ 
may vary within a curve contained 
in $\mathrm R_g$ and after successively saturating as in~\eqref{par:annulus_h}, $\mu_g$ might fill up 
a complicated, possibly dense, curve in $B_g$. 
\end{rem}

\begin{lem}\label{lem:basic_facts_gamma}~
 The restriction of $\mu$ to $S$   is absolutely continuous with respect to the $2$-dimensional 
Lebesgue measure on $S$ and 
  its density with respect to any real analytic area form on $S$ is given by a positive real analytic function. 
%\item the same properties hold for the restriction of $\mu$ to the annuli $L_h^0(S)$ and $A^g_S$.
\end{lem}

The argument is the same as in Step 5 in the proof of Proposition~\ref{pro:smooth_case}. Now,
from Section~\ref{par:geometry_of_g-orbits}, we get the following 
  {a priori} bound on the analytic continuation of~$\gamma$.

\begin{lem}\label{lem:local_extension_gamma}
 For every $\delta>0$ there exists $r = r(\delta)>0$ such that if 
$\gamma\in \Rk_g$  is an (analytic) arc with $\mu_g(\gamma)>0$ and 
 $w$ is a point of $\gamma$ such that $\dist(w, \Crit(\pi_g))>\delta$, then $\gamma$ admits an  analytic continuation to an analytic 
arc   of size $r$ at $w$, which is evenly charged by $\mu_g$.
\end{lem}

\begin{proof} 
We may assume that  $w\in \gamma\setminus \Tor(B_g)$. With notation as above, pick  $\xi\in X_w^g\cap S$. By Lemma~\ref{lem:circle_escapes}, there exists $\xi'\in L_g^0(\xi)$ such that $\dist (\xi', \Tang(\pi_g, \pi_h))>\delta_2 
= \delta_2(\delta)$. 
By $g$-invariance of $\mu$, $\xi'$ plays the same role as $\xi$. 
Now by  Corollary~\ref{cor:circle_size}, $\pi_g(L_h^0(\xi'))$ has size $r_2(\delta_2)$ at $\pi_g(\xi') = w$ and 
applying Lemma~\ref{lem:proj_annulus} concludes the proof.  
 \end{proof}

\begin{rem} It may happen that the local equation for some ${\Rk}^{\alpha,\beta}_{p,q}(B_g^\circ)$ is of type $\Ima(w^k)=0$, with $w$ in a small disk $\disk_\varepsilon \subset B_g^\circ$ (see the definition of $\NT_g$ in \S~\ref{par:the_dynamics_twisting}). At the origin of such a 
 disk, the curve is singular, with several branches going through the origin.  Lemma~\ref{lem:local_extension_gamma} says that if $\e\ll r $ and $\gamma$ is a smooth analytic arc 
 in one of these branches, 
 then it can be continued to an evenly charged analytic arc accross the origin. 
So, if one of the branches is charged by $\mu$, its symmetric with respect to the origin is charged too.
\end{rem}

\subsection{Analytic continuation of $\gamma$: the isotrivial case}\label{par:analytic_continuation_gamma_isotrivial}
From this point the proof splits in two separate arguments 
according to the isotrivial or non-isotrivial nature of the fibrations (more precisely, according to the emptyness, or not, of 
$\Tang^{\mathrm{tt}}(\pi_g, \pi_h)$).  
%Recall that  $g$ and $h$  are  conjugate so the associated  fibrations  are of the same nature (isotrivial or not). 
%In the isotrivial case, the analytic continuation result is the following. Recall that we are given an ergodic invariant measure $\mu$ and an analytic arc $\gamma\subset B_g^\circ$ such that $\mu_g(\gamma)>0$. 

\begin{lem}\label{lem:analytic_continuation_gamma_isotrivial}
If  $\pi_g$ is isotrivial  there is an analytic subset $\tilde \sigma$ of $B_g$ such that: %\romain{Attention! à ce stade je ne sais pas démontrer  que $\mu_g(\sigma)=1$ donc formulation nécessaire pour démontrer la finitude du nombre de composantes dans la démo finale. Remarquer qu'on ne peut pas montrer directement $\mu_g(\sigma) = 1$ ici car ça doit utiliser l'ergodicité d'une manière ou d'une autre}

%\begin{enumerate}[\em (1)]
%\item there exists a semi-analytic subset $\sigma$ of $B_g$ of dimension 1 \romain{nouvelle notation $\sigma$; je trouve ça meilleur que $\hat\gamma$, $\tilde\gamma$, etc.}
%extending $\gamma$,  such that $\mu_g(\sigma) = 1$; 
%\item 
%$\sigma$ is a finite  union of   immersed real analytic loops in $B_g^\circ$  and 
%immersed real analytic arcs  with endpoints in $\Crit(\pi_g)$.
%\end{enumerate} 
\begin{enumerate}[\em (1)]
\item  $\tilde \sigma$ is of pure dimension $1$ and extends $\gamma$;
\item $\mu_g(\tilde \sigma)>1/2$ and 
 $\supp(\mu_g\rest{\tilde\sigma})$ is a semi-analytic set $\sigma\subset \tilde\sigma$; it is a finite  union of 
  immersed real analytic loops in $B_g^\circ$  and 
immersed real analytic arcs  with endpoints in $\Crit(\pi_g)$. 
\end{enumerate} 
\end{lem}

\begin{proof} As   observed in \S~\ref{par:tangencies},  $\pi_g$ becomes birationally  equivalent to a trivial 
fibration after a finite 
base change $B'\to B_g$. In particular its monodromy is finite and the curves $R_{p,q}^k$ define 
 global analytic subsets of $B'$. 
 Coming back to $\pi_g\colon X\to B_g$,  the local curves $R_{p,q}^k(U)$
extend as (singular) global  analytic subsets of $B_g$. Since $\mu_g(\Rk_g) = 1$ we can find a finite 
number of smooth analytic arcs $\gamma_1, \ldots, \gamma_\ell$ contained in $\Rk_g$  such 
that $\mu_g (\gamma_1\cup\cdots\cup \gamma_\ell)>1/2$. 
 Since every 
$\gamma_j$ is contained in some $R_{p,q}^k(U)$, 
 it is contained in a global analytic subset $\tilde \sigma_j$ of $B_g$. We put $\tilde \sigma = 
 \tilde \sigma_1\cup \cdots \cup\tilde \sigma_\ell$. 
By Lemma~\ref{lem:local_extension_gamma}, every irreducible component of 
$\tilde \sigma \cap B_g^\circ$ of positive mass is evenly charged by $\mu_g$. 
To conclude, we  define 
$\sigma$ to be   the closure of the union of the components of $\tilde \sigma\cap B_g^\circ$ charged by $\mu$.\end{proof}

\subsection{Analytic continuation of $\gamma$: the general case}\label{par:analytic_continuation_gamma_general} 

By Lemma~\ref{lem:isotrivial}, if $\Tang^{\mathrm{tt}}(\pi_g,\pi_h)= \emptyset$ then $\pi_g$ and $\pi_h$ are isotrivial; in that case, 
Lemma~\ref{lem:analytic_continuation_gamma_isotrivial} applies.   Now,       we   
 assume   $\Tang^{\mathrm{tt}}(\pi_g,\pi_h)\neq \emptyset$, and our goal is to establish the following lemma.
 
\begin{lem}\label{lem:analytic_continuation_gamma_general}
If $\Tang^{\mathrm{tt}}(\pi_g,\pi_h)\neq  \emptyset$, then there exists 
a unique analytic curve 
$\sigma_g$ in $B_g^\circ$  such that
\begin{enumerate}[{\em (1)}]
\item if $\mu$ is any
 ergodic $\Gamma$-invariant probability measure such that 
$\mu_g(\mathrm R_g)>0$, then $\mu_g(\sigma_g) = 1$; 
\item if $\gamma\subset \sigma_g$ is any   arc,
then $\mu(\gamma)>0$ for at least one such measure. 
\end{enumerate}
\end{lem}

 This result is both stronger and weaker than 
 Lemma~\ref{lem:analytic_continuation_gamma_isotrivial}: indeed its conclusion holds for {\emph{all
   invariant measures with  $\mu_g(R_g)>0$}} (this fact will be important for Theorem~\ref{thm:finiteness});
    on the other hand it gives no information on the structure of the 
   analytic continuation $\sigma_g$ 
   near $\Crit(\pi_g)$ (this issue will be investigated in the next section). 

The curve  $\sigma_g$ is {\bf{defined}} by this lemma. Note that, at this stage of the proof, 
it could  contain for instance a  sequence of small topological circles 
converging to a critical value of~$\pi_g$. We shall exclude such a possibility later. 

%%%
\subsubsection{An elementary lemma} 
%%%
 
Let $M_k:\C\to \C$ be the monomial $M_k(z) = z^k$.  

\begin{lem}\label{lem:simple_curvature}
Let $r$ be a positive real number. For any $0 < \e  \ll r$ there is a constant $C_r(\e)>0$ with the following property. 
If $z_0\in \C$ satisfies $\abs{z_0} \leq \e$, and if $\gamma$ is an analytic arc  of size $\geq r$ at $z_0$ with
 $0\notin  \gamma$, then $M_k( \gamma)$ contains a point at which the 
 curvature   is $\geq C_r(\e)\e^{-k}$.   
\end{lem}

 \begin{figure}[h]\label{fig:curvature}
\includegraphics[width=8cm]{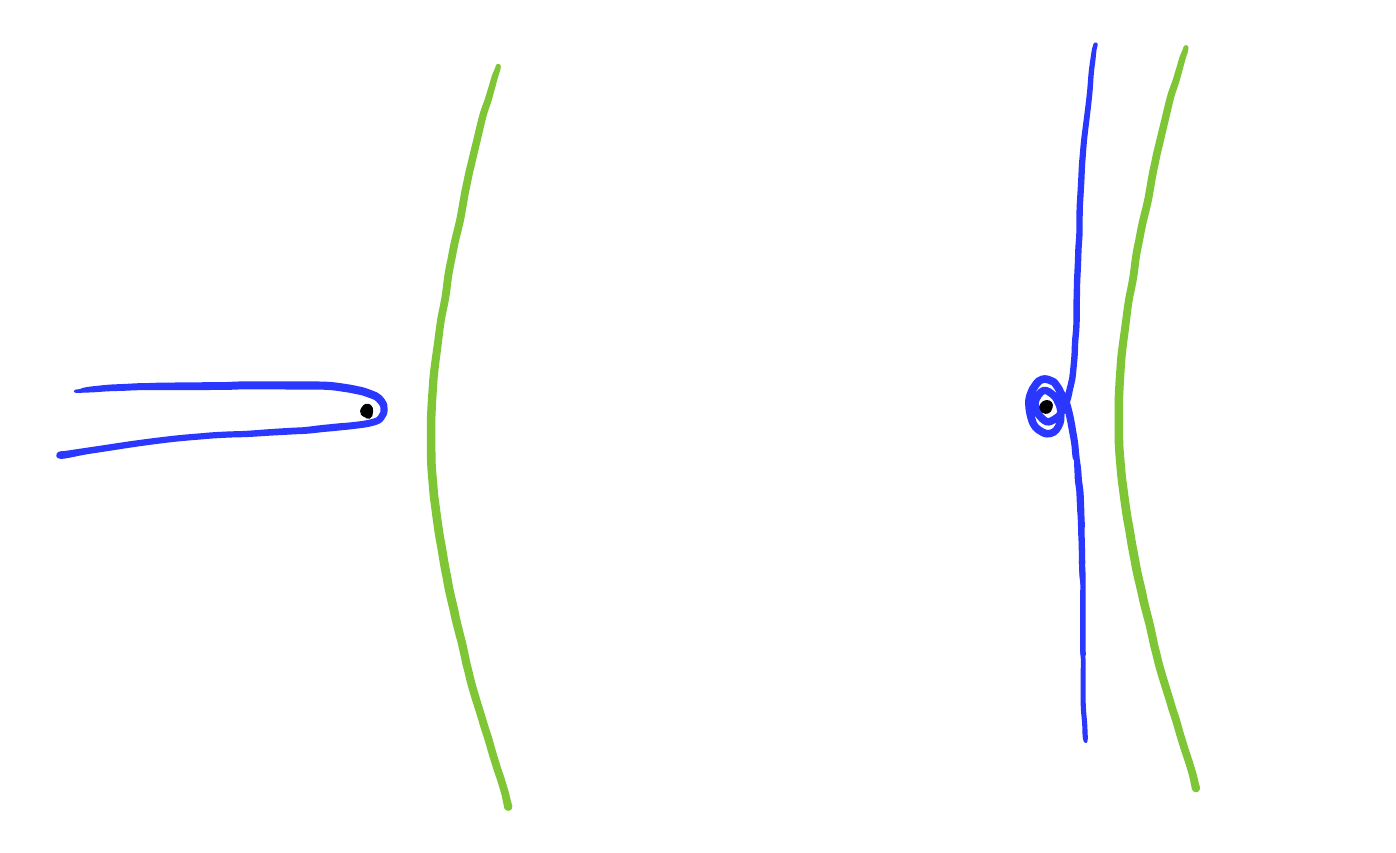}
\caption{A (green) curve of bounded geometry   passing near the origin and its image (in blue) by $z\mapsto z^2$ on the left, and $z\mapsto z^5$ on the right.}
\end{figure}

\begin{proof} We prove it via an explicit computation; see Figure~\ref{fig:curvature} for a graphical explanation. 
We may assume that $r=1$, so that $0<\e \ll 1$, and that $\gamma$ is of size $1$ (we truncate $\gamma$ if necessary). 
 Let $z_1$ be the point 
closest to 0 on $\gamma$.  
Applying  a rotation  we may assume that $z_1 = \eta\in \R_+$ for some $\eta\leq\e$. %, and without loss of generality we rename  $\e'$ into $\e$. 
Since $z_1$ is closest to the origin, $\gamma$ is a graph over the $y$-axis: it is of the form $x = \varphi(y)$, with $\varphi(y) = \eta+ \alpha y ^2+ O(y^3)$. Since $\gamma$ has size $1$ at $z_0$ and $\e\ll 1$, it has size $\geq 1/2$ at $z_1$,  in particular $\abs{\alpha}\leq 2$. In 
 polar coordinates $(\rho, \theta)$,   the equation of $\gamma$ is of the form 
 \begin{equation}
  \rho = \psi(\theta) = \eta\lrpar{1+\lrpar{\frac12+\alpha\eta}\theta^2  + O(\theta^3)} 
  \end{equation}
 for  $\theta$  small.  Thus the polar equation of 
 $M_k(\gamma)$ near $\theta = 0$ is 
  \begin{equation}
\rho = \psi(\theta/k)^k = \eta^k \lrpar{1+\unsur{k} \lrpar{\frac12+\alpha\eta}\theta^2  + O(\theta^3)},
 \end{equation}
 and finally the value of the curvature at $\theta = 0$ is  
 \begin{equation} 
 \left.\frac{\rho^2 + 2(\rho')^2 - \rho\rho''}{\lrpar{\rho^2+ (\rho')^2}^{3/2}} \right|_{\theta = 0} = 
 \eta^{-k} \frac{k-1-2\alpha\eta}{k}\geq \eta^{-k}\geq \e^{-k}.
 \end{equation} This completes the proof. 
\end{proof}

%%%
\subsubsection{Extension of $\gamma$ in $B_g^\circ$: proof of Lemma~\ref{lem:analytic_continuation_gamma_general}}
\label{par:no_accumulation_of_gamma_in_B_g} 
%%%
%Here, we show that {\emph{the image of the analytic 
%prolongation $\widetilde{\gamma}$ of $\gamma$ in $B_g^\circ$
%is an analytic subset of $B_g^\circ$}}. 

Let $\mu$ be an arbitrary $\Gamma$-invariant and ergodic measure, giving no mass to algebraic subsets, 
and assume that $\mu_g(\mathrm R_g)>0$. 
Then, %from  \S\ref{par:local_structure},  
by Remark~\ref{rem:any_Halphen},  
$\mu_g$ gives full mass to $R_g$.  

Fix a small open disk  $U \Subset B_g^\circ$ in which we have fixed a section of $\pi_g$ and a 
continuous choice of basis for $H^1(X_w^g, \Z)$. 
In $U$,  $\Rk_g$ is the  union of the analytic curves $\Rk^{k}_{p, q}(U)$ (see \S\ref{par:the_dynamics_R}).
To prove the lemma, we seek a uniform bound on $\max(k, \norm{(p, q)})$ for the indices $(k, (p,q))$ such that $\mu_g(\Rk^{k}_{p, q}(U))>0$ for at least one $\Gamma$-invariant ergodic measure~$\mu$. Note that  
even  for  a single $\mu$, $\mu_g$ could charge infinitely many of the $\Rk^{k}_{p, q}(U)$; due to the monodromy of the fibration, the analytic continuation of an arc $\gamma$ of positive mass could come back infinitely many times in $U$, each time along a new $\Rk^{k}_{p, q}(U)$. 

Suppose that there is a sequence of curves $\gamma_n=\Rk^{k_n}_{p_n,q_n}(U)$ with $\max(k_n, \norm{(p_n, q_n)})\to \infty$
such that $\mu_{n,g}(\gamma_n)>0$ for some $\Gamma$-invariant measures $\mu_n$ (that depend a priori on $n$).
 The $\gamma_n$ are evenly charged (by $\mu_n$), and by Lemma~\ref{lem:local_extension_gamma} they
have uniformly bounded geometry (the constant $r(\delta)$ in Lemma~\ref{lem:local_extension_gamma} does not depend on $\mu_n$). 
In particular the accumulation locus of $(\gamma_n)$
is uncountable.   

Over $U$, there exists 
$\e = \e(k,p,q)$, with $\e(k,p,q)\to 0$ as $\max(k, \norm{(p, q)})\to \infty$ such that 
for every $w\in U$, any translate of 
$L_{g,w}(k, (p,q))$ is $\e$-dense in $X_w^g$ (see \S~\ref{par:the_dynamics_R} and Lemma~\ref{lem:circles_are_epsilon_dense}). Hence, there is a sequence $(\e_n)\in (\R_+^*)^\N$ converging to $0$ 
such that $L_g(\xi)$ is $\e_n$-dense in $X_w^g$ for any $w\in \gamma_n\setminus \Tor_g(U)$ and $\xi\in X_w^g$. 

Since   $\Tang^{\mathrm{tt}}(\pi_g, \pi_h)$ is non-empty, it 
intersects every  $X_w^g$ along some non-empty finite subset, which by definition is 
not persistently contained in $\Sing(\pi_h)$. Furthermore $\Tang^{\mathrm{tt}}(\pi_g, \pi_h)$ is generically 
transverse to $\pi_g$. So we can pick $w_0$ in the accumulation set of $(\gamma_n)$ together with 
$\xi_{\mathrm t}\in \Tang^{\mathrm{tt}}(\pi_g, \pi_h) \cap X_{w_0}^g$ such that $\pi_h(\xi_{\mathrm t})$ 
is $2\delta$-far from 
$\Crit(\pi_h)$ for some $\delta>0$, and $\xi_{\mathrm t}$ can be locally 
holomorphically followed as a point  $\xi_{\mathrm t}(w)$ in 
$\Tang^{\mathrm{tt}}(\pi_g, \pi_h) \cap X_{w}^g$ for $w$ near $w_0$. 

Pick a sequence $w_n\in \gamma_n\setminus \Tor_g(U)$ converging to $w_0$, and consider the 
corresponding sequence $\xi_{\mathrm t}(w_n)$. For each $n$ there is a finite union of annuli 
$A_n\subset X_{\gamma_n}^g$ such that  $\pi_g$ maps $A_n$ onto $\gamma_n$, 
$\pi_g\colon A_n\to \gamma_n$
 is a fiber bundle whose fibers are unions of $k_n$ circles of type $L_g^0(\xi)$, and 
 $\mu$ charges $A_n$ evenly.
We can choose $\xi'_n\in A_n$ %\cap V$ 
such that  
 $0< \dist(\xi'_n,\xi_{\mathrm t}(w_n)<2\e_n$. 
 In  $X_{w_n}^g$ the curve $L_g^0(\xi'_n)$ is a circle, with uniformly bounded geometry.
 Near $\xi_{\mathrm{t}}(w_n)$ the restriction of $\pi_h$ to $X_{w_n}^g$ 
 is locally conjugate to $z\mapsto z^k$
 where $k\geq 2$ is  
 the order of tangency between $\pi_g$ and $\pi_h$ along 
 $\Tang^{\mathrm{tt}}(\pi_h,\pi_h)$ at $\xi_{\mathrm{t}}(w_n)$. 
The local change of coordinates that transforms $\pi_h$ into $z\mapsto z^k$ depends on the fiber, 
hence on $w_n$, 
 but the first and second derivatives of these changes of coordinates are uniformly controlled, independently of $n$. Thus, Lemma~\ref{lem:simple_curvature} shows that the curvature of 
 the projections $\pi_h(L_g^0(\xi'_n))$ goes to $+\infty$ as $n\to\infty$. 
But $\pi_h(L_g^0(\xi'_n))$ is a piece of  $\mathrm R_h$ which is 
charged by $\mu_h$ %(that is, with notation as above, it is included in    $\widetilde {\mathrm R}_h$) 
and is $\delta$-far from 
$\Crit(\pi_h)$ for large $n$; so by Lemma~\ref{lem:local_extension_gamma}, 
its curvature must be  uniformly bounded (with respect to $n$). This is a contradiction, and the proof of 
Lemma~\ref{lem:analytic_continuation_gamma_general} is complete. \qed

 \begin{rem}\label{rem:kpq_bounded}
 The proof shows that under the assumptions of Lemma~\ref{lem:analytic_continuation_gamma_general}, 
 if $\gamma$ is an arc of a curve $\mathrm R^k_{p,q}$ such that $\mu_g(\gamma)>0$ for some invariant probability measure $\mu$, then 
 $\max(k,\norm{(p,q)})$ is uniformly bounded. 
 \end{rem}

\subsection{Conclusion of the proof of Proposition~\ref{pro:analytic_sigma}} \label{par:proof_pro_analytic_sigma} 

\begin{proof}[Proof of Proposition~\ref{pro:analytic_sigma}]  
%Let $g$ be as in the statement of the proposition.  
%\serge{J'ai changé l'introduction de $Y$ et des $g_i$. C'est bon ?}\romain{je rechange un peu la définition et ce paragraphe car dans $Y$  il faut ajouter la contrainte de conjugaison} 
%Consider the following proper algebraic subset $Y$ of $X$, whose definition is similar to that 
%of $\STang_\Gamma$:  
%\begin{equation}
%Y=\bigcap_{(h_1, h_2)\in \Hal(\Gamma)^2\atop h_1, h_2 \text{ conjugate to } g} \left(\Sing(\pi_{h_1})\cup \Sing(\pi_{h_2})\cup \Tang(\pi_{h_1},\pi_{h_2})\right)
%\end{equation}
%%A point $x$ is in $X\setminus Y$ if and only if there is a pair $(g,h)\in \Hal(\Gamma)^2$ such that 
%%the fibers of $\pi_g$ and of $\pi_h$ through $x$ are smooth, and the fibrations $\pi_g$ and $\pi_h$ are
%%transverse in a neighborhood of $x$. 
%Since the set of Halphen twists is invariant under conjugacy, 
%$Y$ is $\Gamma$-invariant. Since, by definition, $\STang_\Gamma$ is contained in $Y$, by arguing as 
%in Lemma~\ref{lem:Tang_DGamma}
%we get  
%that the $1$-dimensional part of $Y$ coincides with $D_\Gamma$.  
Let $g$ be as in the statement of the proposition.  
Recall the definition of $\STang_\Gamma$ from \eqref{eq:STang_Gamma}:
\begin{equation}
\STang_\Gamma=\bigcap_{(h_1, h_2)\in \Hal(\Gamma)^2} \left(\Sing(\pi_{h_1})\cup \Sing(\pi_{h_2})\cup \Tang(\pi_{h_1},\pi_{h_2})\right)
\end{equation}
By the Noether property, this infinite intersection  
%defining $Y$ 
can be written as a finite intersection. 
Thus, we can choose a finite set of Halphen twists $(g_i)_{1\leq i\leq s}$ 
%with the following properties:
%\begin{enumerate}[(a)]
%\item for every $i\neq j$, the pair $(g_i, g_j)$ is special;
%\item the $g_i$ are pairwise conjugate; more precisely for every $i$ there exist $f_i\in \Gamma$ such that 
%$g_i=f_i\circ g_1\circ f_i^{-1}$, $\pi_{g_i}=\pi_{g}\circ f_i^{-1}$ (of course, we take $f_1=\id_X$). 
%\item if $x\notin Y$, there exist $i\neq j$ such that the fibers of 
%$\pi_{g_i}$ and $\pi_{g_j}$  at $x$ are smooth fibers and are  transverse at $x$.
%\end{enumerate}
%To justify this, we consider the set $\mathcal S$   of singular fibers of 
%$\set{fgf\inv, \ f\in \Gamma}$, which  is the $\Gamma$-orbit of 
%$\Sing (\pi_g)$. Therefore the set $Y_{\mathrm{tang}}$ (resp. $Y_{\mathrm{sing}}$) of points $x\in X$ such that all curves  of $\mathcal S$ are tangent at $x$ (resp.  belonging to all curves of $\mathcal S$) 
%is a $\Gamma$-invariant algebraic set  containing $\Tang_\Gamma$ (resp. $F_\Gamma$). Set $Y = Y_{\mathrm{tang}} \cup Y_{\mathrm{sing}}$. 
%The  1-dimensional part of $Y$ is an invariant algebraic curve containing $\Tang_\Gamma ^1 = F_\Gamma^1 = \mathrm D_\Gamma$ so it coincides with 
%$\mathrm D_\Gamma$. This shows that we can find a family of conjugate Halphen twists $g=g_1, g_2, \ldots , g_s$,
%whose associated fibrations are distinct, and satisfying~(b) and~(c).
%By Lemma~\ref{lem:schottky-parabolic}, replacing each $g_i$ 
%by a sufficiently large iterate, we get property~(a). 
with $g_1=g$ and such that if $x\notin \STang_\Gamma$, there exist $i\neq j$ such that the fibers of 
$\pi_{g_i}$ and $\pi_{g_j}$  at $x$ are smooth fibers and are  transverse at $x$.

For  $x\notin \STang_\Gamma$, there are $k<\ell$  and a (Zariski) neighborhood $V\ni x$ 
disjoint from $\Sing(\pi_{g_k})\cap \Sing(\pi_{g_\ell})$ such that 
$\pi_{g_k}$ and $\pi_{g_\ell}$ are transverse in $V$.  Let us  show that in $V$, 
$\supp(\mu)$ is a real analytic set 
satisfying the conclusions of the proposition. There are two possibilities: 
\begin{enumerate}
\item[(A1)] either there exists 
$k\neq \ell$ such that $\Tang^{\mathrm{tt}}(\pi_{g_k}, \pi_{g_\ell}) = \emptyset$; 
then by Lemma~\ref{lem:isotrivial}
$\pi_{g_k}$ and $\pi_{g_\ell}$ are isotrivial;
%, and since by assumption the $g_i$ are conjugate, $\pi_{g_i}$ is isotrivial for every $1\leq i\leq s$;
\item[(A2)] or for every $k\neq \ell$, $\Tang^{\mathrm{tt}}(\pi_{g_k}, \pi_{g_\ell})  \neq \emptyset$.
\end{enumerate}

%\serge{J'ai l'impression qu'on pourrait d'abord construire $\tilde\sigma$ pour (A1) et pour (A2); puis ensuite montrer les propriétés à démontrer pour les deux à la fois...}\romain{pas clair les 2 cas sont assez différents: pour les arguments globaux (autour des composantes irréductibles), dans le 1er cas on utilise les prolongements globaux $\tilde \sigma$ et dans le 2e }
$\bullet$ Let us first complete the proof in case~(A1). Fix $k\neq \ell$ given by Condition (A1) and  apply Lemma~\ref{lem:analytic_continuation_gamma_isotrivial}: there exist 
analytic curves $\tilde\sigma_{g_k}$ in $B_{g_k} $ and $\tilde\sigma_{g_\ell}$ in $B_{g_\ell}$ such that 
 $\mu_{g_k}(\tilde \sigma_k)>1/2$ and $\mu_{g_\ell}(\tilde \sigma_\ell)>1/2$. Then
 $ \tilde  \Sigma_{kl} := \pi_{g_k}\inv(\tilde\sigma_{g_k})\cap \pi_{g_\ell}\inv(\tilde\sigma_{g_\ell})$  is a   
 real analytic subset of~$X$ such that $\mu (\tilde  \Sigma_{kl})  >0$. 
 Observe that $ \tilde  \Sigma_{kl}$ is C-analytic in $X$ 
 (see \S~\ref{par:semi-analytic-vocabulary} for this notion): indeed the curves $\tilde\sigma_{g_k}$  
 and $\tilde\sigma_{g_k}$, being of dimension 1, are defined by global equations in 
 $B_{g_k} $ and  $B_{g_\ell}$, hence so does $ \tilde  \Sigma_{kl} = \pi_{g_k}\inv(\tilde\sigma_{g_k})\cap \pi_{g_\ell}\inv(\tilde\sigma_{g_\ell})$ on $X$. 
 Let $\tilde \Sigma_0$ be an irreducible 
 component   of $ \tilde  \Sigma_{kl} $ such that  $\mu(\tilde \Sigma_0)>0$.  If $f$ is an element of $\Gamma$, 
$\mu(f(\tilde \Sigma_0))=\mu(\tilde \Sigma_0)$, 
and $\mu(f(\tilde \Sigma_0)\cap \tilde \Sigma_0)=0$, unless $ f(\tilde \Sigma_0)\cap \tilde \Sigma_0$
contains a non-empty, relatively open subset. In the latter  case 
 $ f(\tilde \Sigma_0)= \tilde \Sigma_0$.
Thus,  a  finite index subgroup $\Gamma_{0}$ of $\Gamma$  fixes $\tilde \Sigma_0$.
Define 
 \begin{equation}\label{eq:components}
\tilde \Sigma= \bigcup_{f\in \Gamma} f(\tilde \Sigma_0)=\bigcup_{f\in \Gamma/\Gamma_0} f(\tilde \Sigma_0)
\end{equation}
This is a  C-analytic subset of $X$ (hence of $X\setminus \STang_\Gamma$)
 such that $\mu(\tilde \Sigma) = 1$. We finally define $\Sigma$ to 
be the union of the irreducible components of 
$\tilde \Sigma\setminus \STang_\Gamma $ that are charged by $\mu$. 
Using the $\Gamma$-invariance of $X\setminus \STang_\Gamma$ 
and repeating the previous argument shows that 
$\Sigma$ has finitely many irreducible components which are permuted by 
$\Gamma$ (Assertion~(2) of the proposition), and by construction 
 $\Sigma$ (hence $\overline \Sigma$) 
is semi-analytic in $X$. 
By definition $\mu(\Sigma) = 1$. 

 Now choose any irreducible component $\Sigma_1$  of $\Sigma$ and any $x\in \Sigma_1$. Around $x$, there are two   transverse   projections $\pi_{g_k}$ and $\pi_{g_\ell}$. The results of 
\S\S~\ref{par:geometry_of_g-orbits} and ~\ref{par:local_structure} imply 
that  $\Sigma_1$ is locally a finite union of ``squares''  of the form $\pi_g\inv(\gamma) \cap \pi_h\inv(\gamma')$ (where $\gamma$ and $\gamma'$ are smooth analytic curves), along which $\mu$ has a positive real analytic density.
In particular $\mu(\Sigma_1)>0$ and $\Sigma_1$ is evenly charged by~$\mu$. Thus, Assertions~(1)  and~(4) of the proposition are established. 

 It remains to prove Assertion~(3). For this,  note that if non-empty, the $1$-dimensional part of $\Sing(\Sigma)$ is $\Gamma$-invariant. 
Suppose we can find an arc $\beta$ in $\Sing(\Sigma)$. If $\pi_g(\beta)$ were
 not a point, the closure of the $g$-orbit of $\beta$ would contain a
$2$-dimensional annulus that would be contained in $\Sing(\Sigma)$: contradiction. 
Thus  $\pi_g(\beta)$ is a point, and so is $\pi_{g'}(\beta)$
 for any $g'\in \Hal(\Gamma)$.  This contradicts $\beta \subset X\setminus \STang^1_\Gamma$. Thus, $\Sing(\Sigma)$ is a discrete 
subset of $X\setminus \STang^1_\Gamma$.

%The above description of $\Sigma_1$ shows that any  1-dimensional analytic subset of    
 %$\Sing(\Sigma)$  must   contain  a circle in a  smooth fiber $X_w^g$ of (say) $\pi_g$. 
 %But then for a generic point of this circle, its $h$ orbit it dense in a circle in a fiber of $h$.     By invariance, this yields a 
  %2-dimensional component of   $\Sing(\Sigma)$, 
 %which is impossible. This shows that $\dim (\Sing(\Sigma)) = 0$  and we are done with case~(A1).
  %\romain{point délicat ici: $\Sigma$ n'est pas un ensemble analytique de $X$ mais de 
 %$X\setminus (\Tang_\Gamma\cup \mathrm F_\Gamma)$, donc  l'argument précédent utilisant le fait qu'une composante de dim 1 de $\Sing(\Sigma)$ devrait être invariante par un sg d'indice fini ne me paraît pas marcher} 
%The above description of $\Sigma_1$ shows that if $c$ is a  1-dimensional component of    
% $\Sing(\Sigma)$, then it must be  contained in a  smooth fiber $X_w^g$ of $\pi_g$ for some $g\in \Hal(\Gamma)$
% and invariant by a finite index subgroup of $\Gamma$
%But then $\mathrm{Zar}(c) =  X_w^g$ is invariant under a finite index subgroup of $\Gamma$, which is contradictory. 

$\bullet$ In case~(A2),   we fix $k\neq \ell$ and apply Lemma~\ref{lem:analytic_continuation_gamma_general} to 
$\pi_{g_k}$ and $\pi_{g_\ell}$: it yields 
analytic curves $ \sigma_{g_k}$ in $B_{g_k}^\circ $ and $ \sigma_{g_\ell}$ in $B_{g_\ell}^\circ$ such that 
 $\mu_{g_k}(\sigma_k)=1$ and $\mu_{g_\ell}(\sigma_\ell)=1$. Set 
 $ \tilde  \Sigma_{kl} := \pi_{g_k}\inv( \sigma_{g_k})\cap \pi_{g_\ell}\inv( \sigma_{g_\ell})$; it is a 
 2-dimensional totally real  C-analytic subset of 
 $X\setminus(\Sing(\pi_{g_k})\cup\Sing(\pi_{g_\ell}))$. 
 We then further restrict it to 
 %\begin{equation}
 %X_{k\ell}:=X\setminus(\Sing(\pi_{g_k})\cup\Sing(\pi_{g_\ell})\cup \Tang(\pi_{g_k}, \pi_{g_\ell}))
 %\end{equation} and 
 $X_{k\ell}:=X\setminus \STang(\pi_{g_k}, \pi_{g_\ell}))$ and 
 define $ \Sigma_{k\ell}$ to be the union of 
  irreducible components of  $ \tilde  \Sigma_{k\ell} \cap X_{k\ell}$ of positive $\mu$-mass.  
  Note that  $\mu(\Sigma_{k\ell})=1$. 
   From the analysis of \S~\ref{par:local_structure}, we know 
   that any irreducible component of $\Sigma_{k\ell}$ is evenly charged by $\mu$. 
So, for any other pair $(k',\ell')$, the equality 
  $\mu(\Sigma_{k'\ell'}\cap \Sigma_{k\ell})=1$ implies that the analytic sets 
  $\Sigma_{k\ell}$ and $\Sigma_{k'\ell'}$ coincide on  $X_{k\ell}\cap X_{k'\ell'}$. Thus 
  the $\Sigma_{k\ell}$  patch together to form 
  a real  analytic subset  $\Sigma$  of     $\bigcup_{k\neq \ell}X_{k\ell} = 
  X\setminus \STang_\Gamma$. Since it is locally a finite union of $2$-dimensional  real analytic (and totally real) plaques, we infer that 
  $\Sigma$ is C-analytic in $X\setminus  \STang_\Gamma$. 
Using the $\Gamma$-invariance of $X\setminus  \STang_\Gamma$, we see that 
any component $\Sigma_0$ of positive mass is invariant under a finite index subgroup of $\Gamma$ (see Equation~\eqref{eq:components} and the lines preceding it); hence, 
$\Sigma$ has  finitely many irreducible components. 
The remaining properties of $\Sigma$ are obtained exactly as in case~(A1), and the proof   is complete.    \end{proof}

\begin{rem}\label{rem:semi-analytic}
The proof shows that if the curve $\sigma_g$ constructed in 
Lemma~\ref{lem:analytic_continuation_gamma_general} 
is semi-analytic in $B_g$, then the surface $\Sigma$ is semi-analytic in $X$. This holds automatically in case~(A1). 
\end{rem}

%%%%%%%%%%%%%%%%%%%%%%%%%%%%%%%%%%%%%%%%%%%%%%%%%%%%%%%%%%%
%%%%%%%%%%%%%%%%%%%%%%%%%%%%%%%%%%%%%%%%%%%%%%%%%%%%%%%%%%%
\section{Semi-analyticity of $\overline \Sigma$  and complements}\label{sec:semi-analytic}
%%%%%%%%%%%%%%%%%%%%%%%%%%%%%%%%%%%%%%%%%%%%%%%%%%%%%%%%%%%
%%%%%%%%%%%%%%%%%%%%%%%%%%%%%%%%%%%%%%%%%%%%%%%%%%%%%%%%%%%

In this section we continue the  investigation  of case~(c) of Theorem~\ref{thm:main} by studying   
the   semi-analyticity of $\overline\Sigma$.  
This leads to 
Theorem~\ref{thm:semi-analytic}  in \S~\ref{par:global_Rabpq}, 
 and also prepares the ground for  Theorem~\ref{thm:finiteness}.
We keep the choice of Halphen twists $g,h$ from  \S~\ref{par:choice_halphen_twist}.
By Remark~\ref{rem:semi-analytic} the semi-analyticity of $\Sigma$ 
 is already established in the isotrivial case,
so: 
\begin{center}
\emph{throughout this section we assume that $\Tang^{\mathrm{tt}}(\pi_g, \pi_h)\neq\emptyset$.} 
\end{center}
By Remark~\ref{rem:semi-analytic}, we only need to show that the curve  $\sigma_g$  from  
Lemma~\ref{lem:analytic_continuation_gamma_general} 
admits a semi-analytic continuation to $B_g$. So, the work takes place near the singular fibers of~$\pi_g$.

\subsection{Preparation and strategy} \label{subs:preparation_strategy}Recall that locally  in $B_g^\circ$, $\sigma_g$ 
is  a union of smooth branches with uniformly bounded geometry.
% In particular  the 
% irreducible components of $\sigma_g$ are well-defined and obtained by analytic continuation. \serge{Je ne suis pas sûr de comprendre ce que 
% veut dire ``obtained by analytic continuation''}
Recall also from \S~\ref{par:semi-analytic-vocabulary} that analytic curves have well-defined irreducible components. 

\begin{lem}\label{lem:preparation_strategy}
Assume that $\Tang^{\mathrm{tt}}(\pi_g, \pi_h)\neq\emptyset$. Any irreducible component of 
$\sigma_g$  is  either an analytic loop  in $B_g^\circ$ or an immersed line converging to 
$\Crit(\pi_g)$ at its two endpoints. 
\end{lem}

Here by analytic loop we mean an analytic immersion of the circle, with possible self-intersections.  
And by an immersed line we mean an analytic immersion of the real line $\R \to B_g^\circ$.

\begin{proof}
Let $\sigma$ be a component of $\sigma_g$. If  $\sigma$  is compactly contained in $B_g^\circ$ then 
by Lemmas~\ref{lem:local_extension_gamma} and~\ref{lem:analytic_continuation_gamma_general}
it is an analytic curve in $B_g$  and we are in the first situation.  Otherwise, there is a semi-infinite branch  $\sigma^+$ of $\sigma$; since $\sigma$ is analytic outside the finite set $\Crit(\pi_g)$ and the accumulation set of $\sigma^+$ is connected, we deduce that    the accumulation set of $\sigma^+$  is reduced to a singleton $\set{c_0}\subset \Crit(\pi_g)$. Then the second branch must accumulate $\Crit(\pi_g)$ as well (otherwise $\sigma$ would be an analytic loop) and we are done. 
\end{proof}

%\begin{pro}\label{pro:finite_number_components}
%Assume that $\Tang^{\mathrm{tt}}(\pi_g, \pi_h)\neq\emptyset$. 
%There exists a compact subset $K\Subset B_g^\circ$ such that any irreducible component of $\sigma_g$ 
%intersects $K$. In particular $\sigma_g$ has finitely many irreducible components. 
%\end{pro}

From now on we study the structure of $\sigma_g$  locally near  a fixed  $s \in \Crit(\pi_g)$. 
%and show that there  exists a neighborhood $V$ of $s$ such that any  component of $\sigma_g$ escapes $V$. 
We already know that $\sigma_g$ is locally a union of smooth branches of 
the form  $\mathrm{R}^{\alpha,\beta}_{p,q}$, with $(\alpha,\beta,p,q) \in \Q^2/\Z^2\times \Z^2$. 
The study of %the semi-analytic continuation of 
$\sigma_g$ 
will employ two types of arguments: 
\begin{itemize}
\item first, we 
only use  the dynamics of  $g$ and analyze the curves  $\mathrm{R}^{\alpha,\beta}_{p,q}$ near $s$, as 
started in Section~\ref{sec:halphen}. The rest of the group $\Gamma$ is not taken into account.  This allows us
to make some operations (base change, blowing  down  $(-1)$-curves in fibers of $\pi_g$, etc) 
which are not  $\Gamma$-equivariant but preserve the curves $\mathrm{R}^{\alpha,\beta}_{p,q}$. 
This analysis, which is also crucial for Theorem~\ref{thm:finiteness}, 
 is developed in  \S\ref{par:local_Rabpq}.  
\item  then  in \S\ref{par:global_Rabpq}
we take into account the whole action of $\Gamma$ on the singular fibers; doing so,
we have to work on $X$, without  simplifying the fibration $\pi_g$. 
This is where condition (AC) enters into play. 
\end{itemize}

\subsection{Geometry of the curves $\mathrm{R}^{\alpha,\beta}_{p,q}$}\label{par:local_Rabpq}
As explained above, in this paragraph, the 
only information that we retain from the dynamics of $\Gamma$ is the existence 
of the curve $\sigma_g$, its geometric properties, and 
the fact that it is  locally a union of smooth branches of 
the form  $\mathrm{R}^{\alpha,\beta}_{p,q}$. 
After base change and some birational modification, as described in \S\S~\ref{par:multiple_fibers}
and~\ref{par:stable_reduction}, which we simply refer to as the ``reduction'' of the singular fiber, 
we only have  to consider
 a central fiber $X_s^g$  of type $I_0$ 
(that is, a regular fiber) or  $I_b$ with $b\geq 1$.    

\subsubsection{Case 1: type $I_0$.} This corresponds to the stable reduction of 
 fibers of type ${}_mI_0$, $II$, $III$, $IV$, $I_0^*$, $II^*$, $III^*$, $IV^*$ and their blow-ups.

\begin{lem}\label{lem:typeI0}
If the reduction of $X_s^g$ is of type $I_0$, then 
$\sigma_g$ admits a semi-analytic extension at $s$. In particular it admits only finitely 
many irreducible components in a neighborhood of $s$. 
\end{lem}

\begin{proof} 
We may, and do assume that $X_s^g$ is of type $I_0$ (but the results obtained so far for $\sigma_g$ 
hold only outside $s$).

Fix a pair of disks $V\Subset V'$ centered at $s$, with a local section of $\pi_g$ and a fixed basis for $H_1(X_{V'}^g,\Z)$. 
Recall the real-analytic map $T\colon V'\to \R^2$ introduced in \S\ref{par:smooth-fibers-Betti} 
and the definition~\eqref{eq:RabpqT} of $\mathrm{R}^{\alpha,\beta}_{p,q}(V)$ in terms of $T$, which says that locally
$\mathrm{R}^{\alpha,\beta}_{p,q}(V)$   is the preimage of a straight line under $T$, in particular every branch of 
 $\mathrm{R}^{\alpha,\beta}_{p,q}(V)\setminus\set{s}$ admits a  semi-analytic extension at $s$. 
In $V$,
$\sigma_g$ is a union of branches  of $\mathrm{R}^{\alpha,\beta}_{p,q}(V)\setminus\set{s}$. We need
to show that $\sigma_g$ includes only finitely many such branches. 
For this, observe that by Lemma~\ref{lem:analytic_continuation_gamma_general}, $\sigma_g$ is analytic in $V'\setminus V$. 
So if we can show that any irreducible component of $\sigma_g$ in $V$ reaches $\fr V$, the finiteness follows. 
This relies on a topological argument that avoids explicit computations and will be used again below to prove  Lemma~\ref{lem:no_compact_component}.
%If $s$ is not a critical point of $T$, reducing $V'$ if necessary, $T$ induces a diffeomorphism $V'\to T(V')\subset \R^2$,  and the result is obvious. 

By Lemma~\ref{lem:rank_0_1} we may assume that either $T$ is a diffeomorphism from $V'$ to $V'$ or that $s$ is the unique critical point of $T$  in $V'$. 
%    From \S\ref{par:next_level} we know that in this case 
% $\mathrm{R}^{\alpha,\beta}_{p,q}(V)$ looks locally near $s$ like the curve defined by $\Ima(w^k) = 0$, 
%however the size of the corresponding neighborhood might not be uniform\serge{``might not be uniform'':  je ne comprends pas pourquoi, j'ai l'impression qu'on se complique la vie ici.} with respect to $(\alpha,\beta,p,q)$. 
 Suppose that some branch $\gamma$ of 
$\mathrm{R}^{\alpha,\beta}_{p,q}(V)\setminus\set{s}$ is completely contained in $V$ (see 
Figure~\ref{fig:T} below):
it is either   an analytic loop in $V\setminus\set{s}$ or  an immersed 
arc clustering at $s$ at its two ends. We claim that $\gamma$ contains a critical point of $T$, which is a contradiction. %, which yields a  contradiction, for $s$ is the unique critical point of $T$ in $V$. 
Indeed, parameterize $\gamma$
 by a smooth immersion $\varphi:(0, 1)\to V$, which 
extends continuously to $[0,1]$,  with $\varphi(0) = \varphi(1) = c$ (with $c=s$ if $\gamma$ is not a loop).
\begin{figure}[h]
\includegraphics[width=12cm]{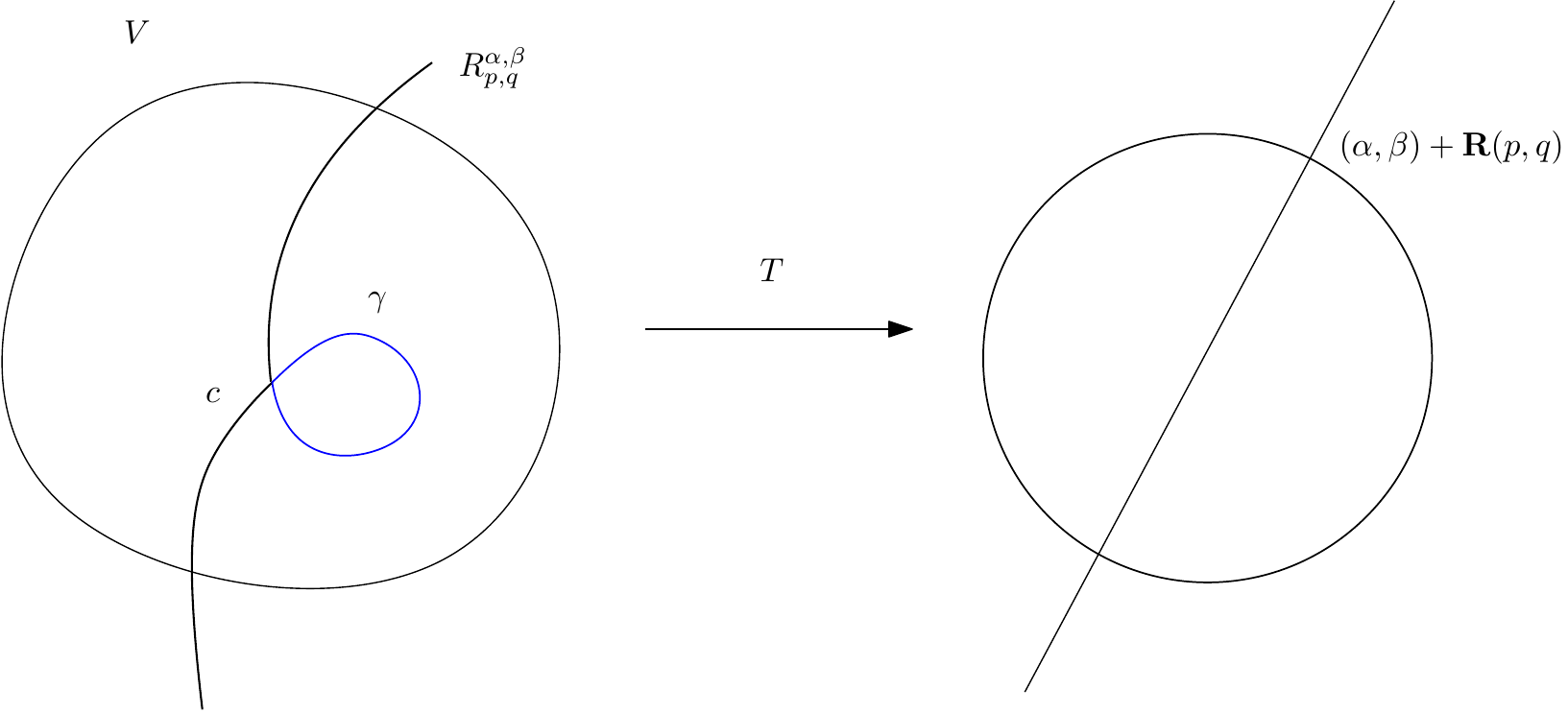}
\caption{}\label{fig:T}
\end{figure}
 Then, $T\circ \varphi$ is a  smooth map from $(0,1)$ to the line $(\alpha,\beta)+\R(p,q)$, it is continuous up to the boundary,    and 
 $T\circ\varphi(0) = T\circ\varphi(1)$;   
by Rolle's Theorem it admits  a critical point in $(0,1)$, as claimed. 
\end{proof}

  \subsubsection{Case 2: type $I_b$, $b\geq 1$.} This corresponds to the stable reduction of 
 fibers of type $I_b$, ${}_mI_b$, $I_b^*$ (with $b\geq 1$), and their blow-ups. 
 As before we fix a pair of disks  $V\Subset V'$ centered at $s$.
The real analytic map $T$ is only locally defined in $V\setminus\set{s}$, but its critical points 
form a well defined subset of $V'\setminus\{s\}$ and,
by Proposition~\ref{pro:NT_finite}, we can assume  this set to be empty. 
We resume the computations from \S\ref{par:notes_on_gamma}.  From Equations~\eqref{eq:R_ab_pq} and~\eqref{eq:tau_w_Ib}, 
there is a coordinate $w$ around $s$ in which $s=0$ and the equation of 
$R^{\alpha,\beta}_{p,q}$ is  
\begin{equation}
\Ima\left( \frac{t(w)  - \lrpar{\alpha+ \beta \frac{b}{2\ii\pi}\log w } }{p+q \frac{b}{2\ii\pi}\log w }\right)=0; 
\end{equation}
here $\log w$ is viewed as a multivalued function and $t(w)$ is a well-defined 
holomorphic function near the origin (see \S~\ref{par:type-Ib}). As already observed, the curve  $R^{\alpha,\beta}_{p,q}$ depends on $(\alpha,\beta)\in \Q^2$ only
 through $q\alpha-p\beta\in \Q/\Z$; we denote by $k$ the order of this torsion point $q\alpha-p\beta\in \Q/\Z$.  %The function $t(w)$ depends on the choice of a local section of $\pi_g$ which is fixed from now on.\romain{peut être pas très bien dit}\serge{Cela dépend vraiment de la section? Je ne pense pas}
%We have to show that every branch of this curve escapes a neighborhood of 0 which is independent of $(p,q)$.  
%From this point  to conclude the proof we have to show that there cannot exist a component which is an immersed arc with both endpoints at $s$. 
%Several subcases will  be considered, depending on the values of $(\alpha,\beta)$, $(p, q)$, and $t(0)$. 
%Here is a preliminary result. 
%\serge{Phrases enlevées, à voir}\romain{me paraît OK}

\begin{lem}\label{lem:no_compact_component}
If $R^{\alpha,\beta}_{p,q}$   admits a compact component in $V\setminus \set{s}$, then 
$q=0$ and this component winds non-trivially around $s$. 
%If $q\neq 0$, the curve  $R^{\alpha,\beta}_{p,q}$ does not admit a compact component in $V\setminus \set{s}$. 
\end{lem}

%\serge{Phrase ajoutée}\romain{je modifie légèrement}
Note that the notion of slope $(p,q)$ depends a priori on our choice of local coordinates, but the property $q=0$ does
not: indeed the action of the monodromy around the loop is given by 
$(p, q)\mapsto (p+bq, q)$ (geometrically, circles of slope $(1,0)$ correspond to the ``vanishing cycle'').

\begin{proof}
Let $C$ be such a component. It can be parametrized by a degree $1$ immersion $\varphi: \mathbb{S}^1 \to C$  
   (not necessarily injective, for $C$ may have self-crossings).
%There are two possibilities: either it winds non-trivially around $s$ or not. 
If $C$ does not wind around the origin,  %then $\varphi: \mathbb{S}^1\to V\setminus \set{s}$ lifts to the universal cover, hence  
there is a continuous determination of $T\circ\varphi$ on $\mathbb{S}^1$. In this way, 
$T\circ\varphi$ is a smooth map from $\mathbb{S}^1$ to a line of slope $(p,q)$ in $\R^2$, so it admits a critical point. 
This contradicts our choice of $V'$. So, the winding number $\ell_C$ around the origin is not zero.  
 Consider a small open disk $U\subset V\setminus\{s\}$
 containing a point $w_0$ of $C$. 
If we follow the local determination of $T$ along $C$ by  turning counterclockwise around the origin, 
$T$ is composed with  the monodromy $(u,v)\mapsto (u+\ell_C bv,v)$. Thus, in $U$, 
$C$   locally satisfies an equation of the form 
 $kT(x,y)\in \R (p,q)$  as well as $kT(x,y)\in \R (p+n\ell_C bq, q)$ for every $n\in \Z$. 
 Since $C$ has only finitely many components in $U$,
this implies $q=0$.  
\end{proof}

%\textit{Subcase 2.1: $q=0$.} 

\begin{lem}\label{lem:q=0}
Assume that the reduction of $X_s^g$ is of type $I_b$, $b\geq 1$. Let $C$ be a branch of $\sigma_g$ in 
$V\setminus \set{s}$ with $q=0$ along some open subset of $C$. Then $C$ admits a semi-analytic extension at $s$. %\serge{Reprendre fin de la partie 3 pour définir/préciser ``type $q=0$'' et ``type rabpq''}
\end{lem}

\begin{proof}
Since  $(p,q)$ is primitive,  $p=1$ and in some small disk $U\subset V\setminus\set {s}$, 
$C$ is of the form $\mathrm R^{\alpha,\beta}_{1,0}$. 
We can choose $(\alpha,\beta)$ of the form $(0,\beta)$, and the local equation of $C$ becomes 
\begin{equation}\label{eq:R_beta_0q}
\Ima\left(  {t(w)  -  \beta \frac{b}{2\ii\pi}\log w }  \right)=0. 
\end{equation}

If $\beta = 0$ this defines an analytic subset of $V$. This set contains the origin if and only if the
imaginary part ot $t(w)$ vanishes at $w=0$, in which case $C$ coincides with a branch of $\set{\Ima\left(  {t(w)  }  \right)=0}\setminus\set{0}$.

If $\beta \neq 0$,  we write $w= e^{-s}$ as in \S\ref{par:notes_on_gamma}, where 
 $s = x+iy$ ranges  in some right half plane $x\geq x_0$. 
 %We have to show that in the $s$-plane,  every branch $\tilde C$
 %of the curve  defined by  Equation~\eqref{eq:R_beta_0q} reaches $x=x_1$ for some uniform 
 %$x_1$ independent of $\beta$. 
 The Equation~\eqref{eq:R_beta_0q} rewrites as 
 \begin{equation}\label{eq:R_beta_0q2}
\Ima (t(e^{-s})) -   \frac{\beta b}{2\pi} x = 0, \quad \text{ that is, }\quad  \tilde{t}(x,y) - \frac{\beta b}{2\pi} x = 0,
\end{equation} 
where $\tilde{t}(x,y)$ is $(2\pi)$-periodic in $y$ and admits a finite limit $\Ima (t(0))$ as $x\to +\infty$; this defines an analytic 
curve $\tilde C$ in the $s$-plane. 
 In particular 
$\tilde t $ is uniformly bounded along  $\tilde C$, 
hence  $\tilde C$ is contained in a vertical strip $\set{x_0\leq x\leq x_1}$. %There are two possibilities. 
The branch $C$ is the projection under $s\mapsto e^{-s}$ 
of   a connected component $\tilde C_0$ of $\tilde C$. It   is contained 
in the annulus $  \set{ \exp(-x_1)\leq\abs{w}\leq\exp(-x_0)}$ and it 
is an analytic subset of this annulus because it is contained in $\sigma_g$.
According to Lemma~\ref{lem:no_compact_component}, $C$ must then  be a loop that winds around the critical value $s$ of $\pi_g$. %If $\tilde C$ contains a compact, connected component $\tilde C_0$, its  projection under $s\mapsto e^{-s}$ is a loop   in $V\setminus \set{s}$ that does not wind around the origin: this loop must coincides with $C$, and this contradicts  Lemma~\ref{lem:no_compact_component}. Thus, one can choose a non-compact connected component $\tilde C_0$ of $\tilde C$ (unbounded in the $y$-direction). 
%Since $\tilde{t}(x,y)$ is $2\pi$-periodic in $y$ and $\sigma_g$ is analytic in $V\setminus\set{s}$, 
%the projection of $\tilde C_0$ is a compact component 
%of   $\mathrm R^{0, \beta}_{1,0}$ in the $w$-plane  which winds around the origin and coincides with $C$. 
\end{proof}

\begin{rem}\label{rem:loops}
If $\Ima(t(0))\neq 0$ and $\beta\neq 0$ is small,   Equation~\eqref{eq:R_beta_0q2} defines a    small loop around the origin. A priori, 
at this stage of the proof, 
$\sigma_g$ might have arbitrarily many small components of this type, converging to $s$. We shall rule out this phenomenon in  Proposition~\ref{pro:finite_number_components}.  
\end{rem}

\begin{lem}\label{lem:qneq0}
If the reduction of $X_s^g$ is of type $I_b$, $b\geq 1$ and if $C$ is a branch of $\sigma_g$ in 
$V\setminus \set{s}$  such that $q\neq 0$ along some open subset of $C$, 
then one end of $C$ converges to $s$  and the other one escapes   $V$. 
There are at most finitely many such branches in $\sigma_g$.  
\end{lem}

\begin{proof}
The last statement follows from the first because every such branch reaches $V'\setminus V$ and $\sigma_g$ 
is analytic in $V'\setminus V$. 

Since $q\neq 0$, we can choose $(\alpha,\beta)$ of the form $(\alpha, 0)$ for some $\alpha \in \Q$, and the equation of $C$ becomes 
\begin{equation}\label{eq:R_alpha0_pq}
\Ima\left( \frac{t(w)  - \alpha }{p+q \frac{b}{2\ii\pi}\log w }\right)=0
\end{equation}
for some (primitive) slope $(p,q)$. The denominator of this expression does not vanish because $\abs{w}<1$. 
Again, we shall write $w= e^{-s}$, $s=x+\ii y$, with $x > x_0\geq 0$.

By Lemma~\ref{lem:no_compact_component}, $C$ contains a branch that accumulates towards $s$. If the other
branch escapes $V$, we are done. So, by Lemma~\ref{lem:preparation_strategy},
 all we have to do is to show that the second branch does not converge 
towards $s$. %\serge{Faut-il ajouter qu'il n'y a pas de spirale infinie à distance finie pour justifier que les deux branches tendent vers $s$?}\romain{ajouté une ref au Lemme 6.1, je pense que ça suffit}
 We argue by contradiction, and parameterize $C$ by an analytic immersion $u\in \R\mapsto \varphi(u)\in V\setminus\{s\}$ such that $\varphi(u)$ goes to $s=0$ as $u$ goes to $+\infty$ and $-\infty$. 
Along that curve, the function $u\mapsto t(\varphi(u))$ converges towards $t(0)$ as $u$ goes to $+\infty$ and to $-\infty$.
Writing $w=\rho \exp(2\ii \pi \theta)$,   Equation~\eqref{eq:tau_w_Ib} 
yields  $\frac{1}{b}\tau(w)=\theta - \frac{\ii}{2\pi}\log(\rho)$; hence, the real part $\tau_1$ of $\tau$
remains bounded while its imaginary part $\tau_2$ goes to $+\infty$ along both ends of $C$. Now, let us consider the function $T$ on the curve $C$.
We start with a local definition of $T$ in a small open subset $U\subset V\setminus\{s\}$ that intersects $C$; locally, $T\rest{C}$ takes values
in a line $L\subset \R^2$ of slope $(p,q)$.
Now, since $\R$ is simply connected, the function $T\rest{C}$ can be analytically continued
 along $C$; since its values are
locally contained in  $L$, they are indefinitely contained in that line. Using Equation~\eqref{eq:psi_explicit_formula}, we can write
$T=(t_1-\frac{\tau_1}{\tau_2}t_2, \frac{1}{\tau_2}t_2)$ in $U$, and this local equality propagates along $C$ by analytic continuation. 
Thus, $T\rest{C}$ converges to $(t_1(0),0)$ at both ends. By Rolle's theorem, we deduce that
the derivative of $T\rest{C}$ vanishes at least once and, by~Lemma~\ref{lem:rank_0_1}, $T$ has a critical point in $V\setminus\{s\}$. This contradiction 
concludes the proof.
\end{proof}

% Ancienne preuve copiée/collée après end{document}
   
\subsubsection{A general finiteness result}
The above results give a rather complete account of the geometry of the branches of $\sigma_g$ near a critical 
value of $\pi_g$. However, for fibers of type $I_b$, $b\geq 1$, our results do not yet imply that $\sigma_g$ is 
semi-analytic at $s$: this is already apparent 
for  the case of curves  $\mathrm{R}^{\alpha,\beta}_{p,q}$ with $q\neq 0$ in the 
toy calculations of \S~\ref{par:notes_on_gamma}. 
Still, we have the following  finiteness result. It does not rely on the choice of a particular invariant measure, so it is stronger 
 than the finiteness of  the number of components of $\Sigma$ in 
Proposition~\ref{pro:analytic_sigma}. 
This will play a key role in Theorem~\ref{thm:finiteness}. 

\begin{pro}\label{pro:finite_number_components}
If $\Tang^{\mathrm{tt}}(\pi_g, \pi_h)\neq\emptyset$, then $\sigma_g$ has finitely many irreducible components. 
\end{pro}

\begin{proof}
Since $\sigma_g$ is analytic in $B_g^\circ$, only finitely many 
components of $\sigma_g$ intersect any given compact  $K\Subset B_g^\circ$. So,  we can work locally   near some fixed $s\in \Crit(\pi_g)$. 
%By Lemma~\ref{lem:typeI0} there are finitely many components near a singular fiber reducing to type $I_0$, and likewise  for fibers reducing to  type $I_b$,  when $q\neq 0$, by Lemma~\ref{lem:qneq0}.  Therefore 
By Lemmas~\ref{lem:typeI0} and~\ref{lem:qneq0}, 
  the only case to deal with is that of fibers reducing to type $I_b$ and branches of type $\Rk^k_{1,0}$ ({i.e.\ }with $q= 0$). According to (the proof of) Lemma~\ref{lem:q=0}, we have to rule out 
  the existence of an infinite sequence of small loops of the form $\mathrm{R}_{1, 0}^{0, \beta_n}$, with $\beta_n\to 0$,   winding  around the origin $s$ and converging to it (see Remark~\ref{rem:loops}). 
  For this we  come back to the analysis of the 
  local structure of   invariant measures from  \S~\ref{par:local_structure}. If an arc $\gamma\subset \Rk_{1, 0}^{0, \beta}$ is locally contained in $\sigma_g$ in some open set $U\subset V\setminus\set{s}$,  there exists an ergodic invariant measure $\mu$ such that for  $w\in \gamma\setminus \Tor_g(U)$, its conditional 
   $\mu_{g,w}$ on $X_g^w$ puts positive weight on  some  circle   $L_g^0(\xi)$
  of slope $(1,0)$ through $\xi$  (see \S~\ref{par:local_structure}). 
 Then, the $g$-invariance shows that $\mu_{g,w}$  puts positive mass on each of the $k$ components of $L_g(\xi)\simeq L_w(k, (1, 0))$, where 
  $k$ is the order of $\beta$ in $\R/\Z$ (see the definition of $L_w(k, (p, q)) $ in~\S~\ref{par:definition_L_k_p_q} and the discussion around Equation~\eqref{par:annulus_h}%\romain{faudra peut être un énoncé auquel se référer, peut être une remarque là bas?}
  ). Now, since $\beta_n$ tends to $0$,  its order $k_n$ tends to infinity. 
Near $s$, this means that when $\sigma_g$ contains a component of type $\mathrm{R}_{1, 0}^{0, \beta_n}$, it contains also a component of 
$\mathrm{R}_{1, 0}^{0, j \beta_n}$ for all $j  = 1, \ldots, \lfloor \beta_n\inv \rfloor $. Fix $\delta > 0$ small and consider the level sets  
$\mathrm{R}_{1, 0}^{0, j \beta_n}$ for $\delta    \leq j  \beta_n \leq 2\delta $. Then, from  Lemma~\ref{lem:q=0} (see Remark~\ref{rem:loops}),
this creates an accumulation of components of $\sigma_g$ away from $\Crit(\pi_g)$,   which is a contradiction.  
\end{proof}

\subsection{Condition (AC) and conclusion}\label{par:global_Rabpq} 
%Recall that in  \S\ref{par:choice_halphen_twist} \romain{à reprendre!!} we started by choosing a 
%Halphen twist $g = g_1\in \Gamma$ and then  a finite set of twists $g_i$ conjugate to $g$,
% satisfying certain properties. The geometric condition (AC) deals with the properties of 
% the singular fibers of $\pi_g$ (in particular it is not affected by changing $g$ into an iterate).

 If $X^g_s $ is a singular fiber of 
 $\pi_g$ with reduction of type $I_b$, $b\geq 1$, we define 
the  {\bf{active components}} of  $X^g_s$ to be the set of  its irreducible components  
which are not contracted during the reduction process. 
More precisely, if $X^g_s$ is not relatively minimal,   there exists a unique birational morphism 
$\e:X \to X'$, with $X'$ smooth, so that $\pi_g\circ\e\inv$ is relatively minimal: no fiber contains a $(-1)$-curve. 
By definition, the components of the exceptional divisor of $\e$ are not active. After this contraction, 
$X^g_s$ becomes a fiber $\e(X^g_s)$ of type ${}_mI_b$, for some $m\geq 1$,  or   $I_{b}^*$. 
In the first case the active components of $X^g_s$
  are the  components which are not contracted by $\e$; equivalently, they are the components of the proper transform of $\e((X')^g_s)$. In the 
 $I_{b}^*$ case, we retain only the proper transform of the $b+1$ components of multiplicity 2 in $\e((X')^g_s)$. 
 
 The Active Components condition reads as follows: 
\begin{itemize}
\item[(AC)] there exists $g\in \Hal(\Gamma)$ such that every  fiber of $\pi_g$ reducing
 to type $I_b$, $b\geq 1$, contains an active component which is not  in $D_\Gamma$.  
\end{itemize}

 \begin{mthmprime}\label{thm:semi-analytic}
If in  Theorem~\ref{thm:main} we further assume the   
non-degeneracy condition (AC), then in the totally real case~$(c)$ we can add the conclusion: (5) $\overline{\Sigma}$ is a semi-analytic subset of $X$. 
\end{mthmprime}

\begin{eg}   \label{eg:condition_(AC)}
If $D_\Gamma$ is empty, then (AC) is satisfied. This is the case for general Wehler surfaces
and general Enriques surfaces if one takes $\Gamma=\Aut(X)$ (see~\cite{finite_orbits} and the references therein). 
If $X$ is minimal and $\pi_g$ does not contain  singular fibers of type $I_b^*$, 
then (AC) is satisfied.  

Let us show that the examples obtained  by Blanc's construction for three points (see \S~\ref{par:examples_rational}) satisfy condition (AC). 
One starts
with a smooth cubic $C\subset \P^2(\C)$ and three general points $p$, $q$, $r$ on $C$. Then, $X$
is the blow-up of $ \P^2(\C)$ at $p$, $q$, $r$, and at the $12$ points $a(p)$, $b(p)$, $c(p)$, $d(p)$, $a(q)$, $\ldots$, $d(q)$, 
$a(r)$, $\ldots$, $d(r)$ of $C$ such that the tangent to $C$ 
at one of these points intersects $C$ in $p$, $q$, or $r$, respectively. The Jonquières involution $s_p$ (resp. $s_q$, $s_r$) that preserves the 
pencil of lines through $p$ (resp. $q$, $r$) and fixes $C$ pointwise lifts to an automorphism of $X$. According to~\cite{Blanc:Michigan}, the subgroup of $\Aut(X)$
generated by $s_p$, $s_q$, $s_r$ is a free product $\Z/2\Z\star \Z/2\Z\star \Z/2\Z$. 
From the formulas given in \cite[Lem. 17]{Blanc:Michigan} for the action of 
$s_\star$ on $\NS(X)$, it can be deduced  that  the composition $g=s_p\circ s_q$ 
is a parabolic automorphism, preserving the pencil of plane quartic curves 
passing through $p$ and $q$ with multiplicity 
$2$ and through the eight points $a(p)$, $\ldots$, $d(q)$ with multiplicity $1$.  
The union of $C$ and the line $L_{p,q}$ through $p$ and $q$ belongs to this pencil; this gives a reducible
fiber of $\pi_g\colon X\to \P^1(\C)$ which, after blowing down the strict transform of $L_{p,q}$, is of type $I_0$.  
It is easy to see that the unique effective divisor 
which is invariant by $s_p$, $s_q$ and $s_r$ is the proper transform of $C$, 
i.e.\ $D_{\langle s_p, s_q, s_r\rangle}=C$. 
Since   condition (AC) deals only with fibers of type $I_b$, $b\geq 1$, we conclude that it holds
  for $\langle s_p, s_q, s_r\rangle$.\end{eg}

\begin{eg} 
 General Coble surfaces with $\Gamma=\Aut(X)$ {\emph{do not satisfy condition}} (AC). More precisely, if $g$ is a Halphen twist, then $\pi_g$ has $10$ singular fibers with
 \begin{itemize}
 \item $8$ singular fibers of type $I_1$, each of them made of a single, active component; 
 \item $1$ fiber made of two rational curves (intersecting transversally 
 in two points), the first is a $(-1)$-curve $E_g$ (contracted by the reduction process), the second is a $(-2)$-curve $S$ which reduces  to a $I_1$ fiber, so it is active;
 \item one multiple fiber $M_g$ of type $2I_0$. 
 \end{itemize}
 The curve $S$ does not depend on $g$, and  $D_\Gamma=S$, so we see that condition~(AC) is violated. 
  On the other hand, if we blow-down the $(-2)$-curve $S$ onto a 
 point, we get a singular surface on which Theorem~\ref{thm:semi-analytic} applies. 
  \end{eg}
  
The relevance of condition (AC)  comes from the following lemma. 
  We resume the context of  Lemma~\ref{lem:qneq0}.
 
 \begin{lem} \label{lem:active_component}
 Let $s\in \Crit(\pi_g)$ be such that $X^g_s$ reduces to type $I_b$, $b\geq 1$. Let
 $C$ be a branch of $\sigma_g$ accumulating $s$, of type 
 $\Rk^{k}_{p,q}$ with $q\neq 0$. Let  $(w_n)$ be a sequence of points of $C\setminus \Tor(B_g^\circ)$
converging toward $s$. Pick an arbitrary  sequence $\xi_n\in X^g_{w_n}$.  If $A$ is any active component 
  of $X_s^g$,  then $L_g^0(\xi_n)$ accumulates $A$ along a non-trivial curve.
  \end{lem}

 \begin{proof}
 Since $A$ is not contracted during the reduction process, we can assume that $X^g_s$ is already of 
  type $I_b$. We now rely on  the description  of fibers of type $I_b$ given in 
  \S\S~\ref{par:local_sections} and~\ref{par:type-Ib}. 
  On a small disk $V\subset B_g$ containing $s$, we pick a local section of $\pi_g$  intersecting $A$ at a 
  smooth point of $X^g_s$ and we construct the surface 
  $X_V^{g,\sharp}$, with central fiber $X^{g,\sharp}_s$ corresponding to $A$. 
 Since $q\neq 0$, the circle $L_g^0(\xi_n)$ is not homotopic to a vanishing cycle, 
 so its length is bounded from below by the injectivity radius of $X$.  
 More precisely, 
 we can extract a subsequence   
 $L_g^0(\xi_{n_i})$ that converges in $X^{g,\sharp}_V$  towards a 
 logarithmic spiral in the central fiber $X^{g,\sharp}_s  \simeq \C^\times$. 
 Here,  by a  {logarithmic spiral}  we mean a   translate of a one
parameter subgroup that goes from $0$ to $\infty$ in the complex multiplicative group $\C^\times$. 
The result follows. 
\end{proof}

\begin{proof}[Proof of Theorem~\ref{thm:semi-analytic}] 
By assumption $\Gamma$ satisfies~(AC). We shall revisit the choice of the Hal\-phen twists $g$ and $h$ from~\S\ref{par:choice_halphen_twist}. 
First, we choose $g$ with the required property in (AC); an extra condition will also be imposed on $h$ (see below).

From the first lines of \S\ref{subs:preparation_strategy}, we may assume 
$\Tang^{\mathrm{tt}}(\pi_g, \pi_h)\neq \emptyset$, and we have to show that  $\sigma_g$ admits a
semi-analytic continuation to $B_g$. Fix $s\in \Crit(\pi_g)$ and small disks  $V\Subset V'$ centered at $s$, as in~\S~\ref{par:local_Rabpq}. If the singular fiber $X^g_s$ reduces to type $I_0$
then $\sigma_g$ is semi-analytic at $s$ by Lemma~\ref{lem:typeI0}. If it reduces to type $I_b$, $b\geq 1$, 
any branch of $\sigma_g$ with $q=0$ is semi-analytic by Lemma~\ref{lem:q=0}. 
By Proposition~\ref{pro:finite_number_components}, 
$\sigma_g$ has   finitely many components near $s$;  thus, we only need to show that any given branch of $\sigma_g$ with  $q\neq0$ is semi-analytic.

Fix a branch $\gamma$ converging to $s$, of type $\Rk^{\alpha,\beta}_{p,q}$ for some $q\neq 0$.  
Since $\Tang_\Gamma$ contains $\mathrm D_\Gamma$, we may assume that there is an active component $A\subset X^g_s$ which is not contained in any fiber of $h$.  Now, fix an invariant, ergodic probability measure $\mu$ such that $\mu_g$ evenly charges $\gamma$.
  We will argue as 
in Lemma~\ref{lem:local_extension_gamma}, except that the 
uniform geometric estimates from \S\ref{par:geometry_of_g-orbits} are replaced by 
Lemma~\ref{lem:active_component}.

The details are as follows. Let $r$ be such that any logarithmic spiral in $A$ contains an arc of size $r$. Fix 
$\e\ll\delta\ll r$.  
Set $k:=\abs{\Tang^{\mathrm{tt}}(\pi_g,\pi_h)\cap X_s^g}$; $k<+\infty$ by definition of $\Tang^{\mathrm{tt}}(\pi_g,\pi_h)$.  Identify $V$ with $\disk_R$, and for $0<R'<R$,
consider the set $\pi_g^{-1}(\disk_{R'})\cap \Tang^{\mathrm{tt}}(\pi_g,\pi_h)$. As $R'$ converges towards $0$, 
this open subset of $\Tang^{\mathrm{tt}}(\pi_g,\pi_h)$ converges towards the finite set 
$\Tang^{\mathrm{tt}}(\pi_g,\pi_h)\cap X_s^g$ in 
the Hausdorff topology. Thus, if $R'$ is small enough, 
$\pi_g^{-1}(\disk_{R'})\cap \Tang^{\mathrm{tt}}(\pi_g,\pi_h)$ is a union of $k$ subsets, 
each of   diameter $\leq \e$; 
its projection under $\pi_h$ is contained in a union of $k$ disks $\Delta_i\subset B_h$, 
each of diameter less than $O(\e)$.

%Set $k:=\abs{\Tang^{\mathrm{tt}}(\pi_g,\pi_h)\cap A}$.  Identify $V$ with $\disk_R$, and for $0<R'<R$,
%consider the set $\pi_g^{-1}(\disk_{r'})\cap \Tang^{\mathrm{tt}}(\pi_g,\pi_h)$. Since $A$ is not contracted 
%by $\pi_h$, as $R'$ converges towards $0$, 
%this open subset of $\Tang^{\mathrm{tt}}(\pi_g,\pi_h)$ converges towards the finite set  $\Tang^{tt}(\pi_g,\pi_h)\cap A$ in 
%the Hausdorff topology. Thus, if $R'$ is small enough, $\pi_g^{-1}(\disk_{r'})\cap \Tang^{\mathrm{tt}}(\pi_g,\pi_h)$ is a union of $k$ subsets, 
%each of which of diameter less than $\e$; its projection under $\pi_h$ is contained in a union of $k$ disks $\Delta_i\subset B_h$, 
%each of diameter less than $O(\e)$.

Fix a sequence  $(w_n)$ in $\gamma\setminus\Tor(B_g^\circ)$  converging towards $s$.
For large $n$, pick a point $\xi_n\in X_{w_n}^g$ which is contained in a small square above $\gamma$ on which 
$\mu$ restricts to a smooth measure, as in Section~\ref{par:local_structure}; 
the orbit $L_g^0(\xi_n)$ is a circle in $X_{w_n}^g$ and by Lemma~\ref{lem:active_component}  a subsequence of $L_g^0(\xi_n)$ converges to a logarithmic spiral in $A$.  
Changing $\xi_n$ into another point $\xi'_n \in 
L_g^0(\xi_n)$, we can therefore assume that (a) $\xi_n$ is in the support of $\mu$, 
(b) $w_n':=\pi_h(\xi'_n)$ is $\delta$-far 
from $\Crit(\pi_h)\cup \bigcup_{i=1}^k \Delta_i$, and (c) $w_n'$ is not in $\Tor(B_h)$; 
then, (d) $L_h^0(\xi'_n)$ is a circle in $X_{w'_n}^h$, whose size at every point is 
bounded from below by a constant that 
does not depend on $n$ (here we apply Lemma~\ref{lem:circle_size} to $h$). 

The set $\pi_g^{-1}(\disk_R)\cap X_{w'_n}^h$ is an open subset of $X_{w'_n}^h$ that
contains $\xi_n'$. If $w'_n$ is close enough to $s$ and $R$ is small, we may assume that the connected
component $V'_n$ of this open set that contains $\xi_n'$ contains a unique point $s'$ of $A$,  that $V'_n$
is a disk, and that $\pi_{g}\rest{V'}\colon V'_n\to V$ is a covering which is ramified at $s'$ only 
(indeed, property (b) above
implies that $\Tang^{\mathrm{tt}}(\pi_g, \pi_h)$ does not intersect $V'$). Consider
the connected component $I_n\subset V'_n$ of $L_h^0(\xi_n')$ that contains $\xi_n'$. The projection 
$\pi_g(I_n)$ is locally contained in ${\gamma}$ around $w_n$.  
If  $I_n$ did not contain $s'$, then $\pi_{g}\rest{I_n}$ would have no ramification point, so $\pi_g(I_n)$ would be an arc with boundary points   in $\partial V$; 
being contained in this arc, $\gamma$ would not converge to $s$, a contradiction.  
Thus $I_n$ contains $s'$ and $\pi_g(I_n)$ is a 
semi-analytic subset of the disk $V$ that contains $s$.  It  is   smooth and analytic  
 in  $V\setminus\set{s}$, and   $\gamma$ is a component of $\pi_g(I_n)\setminus\set{s}$; therefore 
 $\overline \gamma$ is semi-analytic, as desired. 
 \end{proof}

%%%%%%%%%%%%%%%%%%%%%%%%%%%%%%%%%%%%%%%%%%%%%%
%%%%%%%%%%%%%%%%%%%%%%%%%%%%%%%%%%%%%%%%%%%%%%
\section{Finitely many invariant measures: proof of Theorem \ref{thm:finiteness}}\label{sec:finiteness}
%%%%%%%%%%%%%%%%%%%%%%%%%%%%%%%%%%%%%%%%%%%%%%
%%%%%%%%%%%%%%%%%%%%%%%%%%%%%%%%%%%%%%%%%%%%%%

Let as usual $X$ be a compact Kähler surface and $\Gamma$ be a non-elementary subgroup of $\Aut(X)$ containing a parabolic element. 
We want to show the following alternative:
\begin{itemize}
\item either $(X, \Gamma)$ is a Kummer example,
\item or there are only finitely many  $\Gamma$-invariant ergodic  measures  with a Zariski dense support. \end{itemize}
It is shown in \S~\ref{subs:infinitely_many_measures} that Kummer groups can indeed admit infinitely many ergodic totally real measures ({i.e.\ } with $d_\C(\mu)=2$ and $d_\R(\mu)=2$). 

\begin{lem}\label{lem:unique_abso_conti}
There is at most one $\Gamma$-invariant, ergodic probability measure which is absolutely continuous with respect to the Lebesgue measure. \end{lem}
\begin{proof}  Let $\mu$ and $\mu'$ be such measures. 
Fix a real analytic volume form $\omega$ on $X$, for instance $\omega= \kappa\wedge \kappa$ for some Fubini-Study form. 
From Theorem~\ref{thm:main},  $d\mu(x)=\xi(x) \omega$ (resp. $d\mu'(x)=\xi'(x) \omega$) for some function $\xi$ (resp. $\xi'$) which is positive and real analytic on the complement of
some proper real analytic subset $B(\mu)$ (resp. $B({\mu'})$). 
On $X\setminus (B(\mu)\cup B(\mu'))$,  the function $\xi'/\xi$ is continuous and $\Gamma$-invariant; since 
$\mu$ is ergodic, $\xi'/\xi$ is constant, and $\mu'=\mu$. \end{proof}

According to Theorem~\ref{thm:main} and Lemma~\ref{lem:unique_abso_conti}, we are interested only in measures of type~(c): 
those with a totally real support. To prove the finiteness we revisit the proof of Proposition~\ref{pro:analytic_sigma}
and use an alternative similar to that of \S\ref{par:proof_pro_analytic_sigma}; indeed, exactly
 one of the following two situations 
 holds (note that the  quantifiers are switched as compared with the alternative of \S\ref{par:proof_pro_analytic_sigma}):  
 \begin{enumerate}
\item[(A1')]  $\Tang^{\mathrm{tt}}(\pi_{g}, \pi_{h}) = \emptyset$   for every pair $g, h$ in $\Hal(\Gamma)$ such that $\pi_g\neq \pi_h$ (recall   the convention of  Remark~\ref{rem:pig}); 
\item[(A2')]  there exists $g, h$ in $\Hal(\Gamma)$ such that 
  $\Tang^{\mathrm{tt}}(\pi_{g}, \pi_{h})  \neq \emptyset$.
\end{enumerate}
 
 Theorem~\ref{thm:finiteness} is then an immediate consequence of the following two lemmas. 
 
 \begin{lem}\label{lem:A1'}
 If      Alternative~$(A1')$ holds, then $(X, \Gamma)$ is a Kummer group. 
 \end{lem}
 
 \begin{lem}\label{lem:A2'}
  If      Alternative~$(A2')$ holds, then there are only finitely many ergodic, $\Gamma$-invariant probability 
  measures with a Zariski dense support. 
 \end{lem}

 \begin{proof}[Proof of Lemma~\ref{lem:A1'}]
 By Proposition 3.12 in \cite{finite_orbits}, we can choose $g, h\in \Hal(\Gamma)$ which are conjugate in $\Gamma$ and such that every periodic curve for $h\circ g$ is contained in $\mathrm D_\Gamma$. %First,   we can choose $g, h\in \Hal(\Gamma)$   such that (a) the 
%common components of  $\pi_g$ and $\pi_h$ coincide with the components of $\mathrm D_\Gamma$ and
% (b) $h$ is conjugate to $g$. This follows from Proposition 3.12 in \cite{finite_orbits}.  
With such a choice, the 
common components of  $\pi_g$ and $\pi_h$ coincide with the components of $\mathrm D_\Gamma$, 
 and  the foliations ${\mathcal{F}}_g$ and ${\mathcal{F}}_h$ induced by $\pi_g$ and $\pi_h$ are everywhere transverse, 
except along $\Tang^{\mathrm{ff}}(\pi_g,\pi_h)= D_\Gamma$.

\begin{rem} 
If $X\to X_0$ is the  contraction of  $D_\Gamma$ (see Proposition~\ref{pro:contraction_of_DGamma}), the fibrations $\pi_g$ and $\pi_h$ define two genus $1$ fibrations on $X_0$,
and the associated foliations ${\mathcal{F}}_g^0$ and ${\mathcal{F}}_h^0$ are transverse everywhere, except on the finite set $\Sing(X_0)$.
Indeed, a smooth point of the surface can not be an isolated point of tangency between two foliations.
\end{rem}

If $\beta\colon [0,1]\to B_g^\circ$ is a path joining two points $w_0$ and $w$ of $B_g^\circ$, the holonomy of 
${\mathcal{F}}_h$ determines an 
isomorphism $\hol_h(\beta)\colon X_{w_0}^g\to X_{w}^g$; thus, $\pi_g$ is an isotrivial fibration (cf. Lemma~\ref{lem:isotrivial}).  This construction 
defines a representation $\hol_h\colon \pi_1(B_g^\circ; w_0)\to \Aut(X_{w_0}^g)$, the image of which 
fixes the finite
subsets $X_{w’}^h\cap X_{w_0}^g$ for every $w’\in B_h$. Thus, $\hol_h(\pi_1(B_g^\circ; w_0))$ is a finite 
subgroup $H$ of automorphisms 
of the genus $1$  curve $E:=X_{w_0}^g$. A similar argument applies to $\pi_h$ in place of $\pi_g$; we shall denote by 
$E'$ the genus $1$ curve $X_{w_0'}^h$, for some $w_0'$ in $B_h^\circ$, and by $H'$ the corresponding holonomy group. (Note that 
we have $E'\simeq E$ because the 
two fibrations are conjugate by some automorphism of $X$.) We also fix a point $\xi_0\in X$ whose projections are $w_0=\pi_g(\xi_0)$ and $w_0'=\pi_h(\xi_0)$. 

This construction yields  a holomorphic map $\Psi$ from $X\setminus D_\Gamma$, or equivalently   $X_0\setminus \Sing(X_0)$, to $E'/H'\times E/H$, which is defined as follows. To a point $\xi$ in $X\setminus D_\Gamma$, we associate the intersection of the leaf of ${\mathcal{F}}_g$ 
through $\xi$ (i.e. $X_{\pi_g(\xi)}^g$) with the fiber $X_{w_0'}^h$; this gives a unique point  modulo the action of $H'$, hence a point  $\psi'(\xi)\in E'/H'$. Doing the same with respect to ${\mathcal{F}}_h$ and $\pi_g$, 
we get a point $\psi(\xi)\in E/H$, and 
 then we set $ \Psi(\xi)=(\psi'(\xi),\psi(\xi))$. Let $\xi$ be a singularity of $X_0$ and let $U$ be a small 
neighborhood of $\xi$. Then $\psi'(X_{\pi_g(U)}^g)$  is contained in a small disk $V'\subset E'/H'$; similarly,  $\psi(X_{\pi_h(U)}^h)$  
 is contained in a small disk $V\subset E/H$. Thus, $\Psi$ maps $U\setminus \{\xi\}$ in a bidisk $V'\times V$; as a consequence, $\Psi$ extends to $U$, for 
the  singularities of $X_0$ are normal (see \cite[Prop. 3.9]{finite_orbits}). Altogether, this defines a finite ramified cover $\Psi\colon X_0\to E'/H'\times E/H$.
%The map $\Psi$ is a ramified cover from $X_0\setminus \Sing(X_0)$ to $E'/H'\times E/H$; 
%the ramification is contained in the singular or multiple fibers of $\pi_g$ and $\pi_h$. 

We also define a regular map $\Phi$ from $E'\times E$ to $X_0$. 
For this, without loss of generality we 
declare that the neutral element of $X_{w_0'}^h$ is $ \xi_0$ and 
denote by $0$ the neutral element of $E'$; hence,  the pair $(E'\times \{0\}, (0,0))$ is identified 
 to $(X_{w_0'}^h, \xi_0)$ by 
an isomorphism $\varphi_h\colon (E',0)\to  (X_{w_0'}^h, \xi_0)$.
Similarly, we identify $\{0\}\times E$ to $X_{w_0}^g$ via an isomorphism $\varphi_g$ that maps $0$ to $\xi_0$. 
We shall denote by $F'\subset E'$ (resp. $F\subset E$) the finite subset $\varphi_h^{-1}(\pi_g^{-1}(\Crit(\pi_g))$ (resp. $\varphi_g^{-1}(\pi_h^{-1}(\Crit(\pi_h)))$); 
$F'$ corresponds to the intersection of $\varphi_h(E')$ with singular and multiple fibers of $\pi_g$.
If $(u,v)$ is a point of $E'\times E$ close to $(0,0)$, then the fibers $X_{\pi_g(\varphi_h(u))}^g$ 
and $X_{\pi_h(\varphi_g(v))}^h$ have a unique intersection 
point {\textit{near}}~$\xi_0$. This defines  a germ of diffeomorphism $\Phi\colon E'\times E\to X_0$ mapping  
$(0,0)$ to $\xi_0$ and preserving the fibrations; it is defined in a 
   a small bidisk $\disk'\times \disk\subset E'\times E$. Observe
that  the composition $\Psi\circ \Phi$ coïncides with the natural projection 
from $E'\times E$ to $E'/H'\times E/H$. Reducing the bidisk if necessary, this map $\Phi$ extends uniquely to $\disk'\times E$ and provides a local trivialization of the fibration $\pi_g$ above $\pi_g(\disk')$. Similarly, it extends   
to $E'\times \disk$. So $\Phi$ is defined in a ``cross'' of the form  $\disk'\times E \cup E'\times \disk$. 
By analytic continuation, it extends uniquely to 
$(E'\setminus F') \times E$ and to $E'\times (E\setminus F)$, that is, 
  to $(E'\times E)\setminus (F'\times F)$. Moreover, $\Psi\circ\Phi\colon E'\times E\to E'/H'\times E/H$ is the quotient map with respect to the 
action of $H'\times H$. From this, we deduce that the default of injectivity of $\Phi$ is given by a subgroup $G$ of $H'\times H$: $\Phi(p) = \Phi(p')$ if and only if $p'-p\in G$. 
 Since 
$\Psi\colon X_0\to E'/H'\times E/H$ is a finite map, it follows that for each
$(u,v)\in F'\times F$ there is a point $\xi\in X_0$, an open neighborhood $U$ of $\xi$, and an open neighborhood $W$ of $(u,v)$ such that $\Phi$
maps $W\setminus\{(u,v)\}$ into $U$. Embedding $U$ into some affine space, we see by the Hartogs extension theorem that $\Phi$ extends 
to $W$. Thus, $\Phi$ extends to a holomorphic map $E'\times E\to X_0$ which  fits  in a sequence 
\begin{equation}
E'\times  E\xrightarrow{\Phi} X_0 \xrightarrow{\Psi} E'/H'\times E/H,
\end{equation}
such that the composition $\Psi\circ \Phi$ is the natural projection onto the quotient, and the fibers of $\Phi$ are orbits of $G$.

Thus, $X_0$ is a generalized (singular) Kummer surface (see~\cite{finite_orbits}): 
it is a quotient of the abelian surface $E'\times E$ by a finite subgroup 
$G$ of $H'\times H$; the singularities of $X_0$ correspond to the fixed points of elements of $G\setminus\{\id\}$.
Restricting $\Phi$ to the complement of these fixed points, we get a regular finite cover onto the regular part of $X_0$, with $G$ as a group 
of deck transformations. 
Denote $\Lambda$ and $\Lambda'$ lattices in $\C$ such that $E=\C/\Lambda$ and $E'=\C/\Lambda'$; the universal cover of $E'\times E$ is 
$\C^2$, with projection $\C^2\to \C^2/(\Lambda'\times\Lambda)$.
If we think of $X_0$ as an orbifold with quotient singularities, 
its universal cover is $\C^2$. 

From this point, the argument is identical to \cite[Thm. 5.15]{finite_orbits}. 
Let $f$ be an element of $\Aut(X_0)$ and lift it as a holomorphic diffeomorphism 
$F$ of $\C^2$. Its differential 
$DF_{(x,y)}$ at a point $(x,y)\in \C^2$ is an element of $\GL_2(\C)$. Let $L(G)\subset\GL_2(\C)$ be the linear part of 
$G$; it is a finite group, and the class $[DF_{(x,y)}]$ of $DF_{(x,y)}$ in  $\GL_2(\C)/L(G)$ determines a 
holomorphic map that is invariant under translations by the lattice $\Lambda'\times \Lambda$; since  $\GL_2(\C)/L(G)$
is an  affine variety, this map is constant. So, all lifts of all elements of $\Aut(X_0)$ to $\C^2$ are affine maps.

Now, consider the full group $\Gamma\subset \Aut(X)$. It preserves $D_\Gamma$
and induces a subgroup $\Gamma^0$ of $\Aut(X_0)$. By the previous paragraph, 
all elements of $\Gamma^0$ come from affine transformations 
of $E'\times E$. This
proves that $(X_0,\Gamma^0)$ and $(X,\Gamma)$ are  Kummer groups. 
\end{proof}

\begin{proof}[Proof of Lemma~\ref{lem:A2'}] By Lemma~\ref{lem:analytic_continuation_gamma_general},   
there is an analytic  curve  $\sigma_g\subset B_g^\circ$ (resp. 
$\sigma_h \subset B_h^\circ$) such that $\mu_g(\sigma_g) =1$ (resp. $\mu_h(\sigma_h) =1$) for any 
ergodic invariant measure $\mu$. By Proposition~\ref{pro:finite_number_components},  there exists a compact subset $K_g\Subset B_g^\circ$  
such that any component 
 of $\sigma_g$ intersects $K_g$ (and similarly  for $h$). 
In particular any invariant probability
measure gives positive mass to $\pi_g\inv(K_g)\cap \pi_g\inv(K_h)$. In $\pi_g\inv(K_g)\cap \pi_g\inv(K_h)$, the 
regular set of $\pi_g\inv(\sigma_g)\cap \pi_g\inv(\sigma_h)$, which is semi-analytic, has finitely many connected 
components (see~\S\ref{par:semi-analytic-vocabulary} and \cite[Cor. 2.7]{bierstone-milman}). 
By the analyticity of the density in Theorem~\ref{thm:main}.(c), if $\Sigma_0$ is 
such a connected component, there is at most one ergodic invariant probability measure  
giving positive mass to  $\Sigma_0$, and the proof is complete. 
\end{proof}

\section{Orbit closures}\label{sec:closed_invariant}
%%%%%%%%%%%%%%%%%%%%%%%%%%%%%%%%%%%%%%%%%%%%%%
%%%%%%%%%%%%%%%%%%%%%%%%%%%%%%%%%%%%%%%%%%%%%%

\subsection{Sparse subgroups of tori}\label{subs:finitely_many_kpq} 
Let $\e$ be a positive real number. A subset $F$ of a metric space $Z$ is $\e$-dense
if $\{ z\in Z\, ; \, \dist(z,F)\leq \e\} = Z$; thus it is not $\e$-dense whenever 
there is a closed ball of radius $\e$ in $Z$ that
does not intersect $F$. 

Let $G$ be a locally compact abelian group.   
 The space of closed subgroups of $G$  will be denoted by $\mathrm{Sub}(G)$. 
The  Chabauty topology on $\mathrm{Sub}(G)$ is the topology  generated by the sets  
  $$\Omega(K; U_1, \ldots , U_k)=\set{H\in \mathrm{Sub}(G)\; ; \; H\cap K=\emptyset  \text{ and   }
  \forall 1\leq i\leq k, \ H\cap U_i\neq\emptyset},$$
   where  $K$ is an arbitrary    compact subset  of $G$ and the   $U_i $ are  open subsets. 
   A sequence (or net) $H_n$ converges to $H$  for the Chabauty topology if and only if 
   $H_n\cup\set{\infty}$ converges to   $H\cup\set{\infty}$ for the Hausdorff topology in the one-point compactification of $\mathrm{Sub}(G)$. 
  Endowed with this topology, $\mathrm{Sub}(G)$  is a  
   compact Hausdorff space (see~\cite{Cornulier:AGT}).

  Now consider the abelian group $G^\vee = {\mathrm{Hom}}(G;\R/\Z)$  of characters of $G$, which 
  is locally compact when 
   endowed with  the topology of local uniform convergence.
By Pontryagin duality  the natural homomorphism $G\to (G^\vee)^\vee$ is an isomorphism of  topological
groups. According to~\cite[Thm. 1.1]{Cornulier:AGT}, the ``orthogonal'' map
 $ \mathrm{Sub}(G)\to \mathrm{Sub}(G^\vee)$ defined by 
 $H\mapsto H^\vee:=\{ \xi \in G^\vee \; ; \; \xi\rest{H}=0\}$ is a homeomorphism that maps $\{0\}$ to~$G^\vee$.
 
\begin{thm}\label{thm:chabauty}
Let $G$ be a compact, metric, abelian group. For each $\e > 0$, there is a finite set $\{\xi_i ; i\in I\}$ of continuous homomorphisms $\xi_i\colon G\to \R/\Z$ 
that satisfies the following property: if   $H$ is a subgroup of  $G$ which is not $\e$-dense, 
then $H$ is contained in $\mathrm{Ker}(\xi_j)$ for at least one $j\in I$.
\end{thm}

\begin{proof} If a  subgroup is not $\e$-dense, its closure is not $\e$-dense either, so we may
 assume that $H$ is closed.
Since $G$ is compact it is a classical fact that its dual $G^\vee$ is discrete. 
Thus, $\set{0}$ is open in $G^\vee$ and 
any neighborhood $W$ of $\{0\}$ in $\mathrm{Sub}(G^\vee)$ contains  
a set of the form $\Omega(F; \set{0})$, %since any subgroup contains 0)\serge{on peut enlever $\set{0}$ ici}\romain{ça te va comme ça? je trouve ça mieux car $\Omega(F)$ n'a pas été vraiment défini ça pourrait prêter à confusion}
where $F$ is a compact, hence finite, subset of $G^\vee\setminus\{0\}$. 
It follows that 
every  $L\in \mathrm{Sub}(G^\vee)\setminus W$ contains an element of $F$. 
Via the orthogonal map, this   implies that  for every neighborhood $V$
of $G$ in $\mathrm{Sub}(G)$, there is a finite subset $\{ \xi_i ; \; i\in I\}$ of $G^\vee$ such that every element $H$ of $\mathrm{Sub}(G)\setminus V$ is contained in the kernel of one of the $\xi_i$.
The theorem follows. 
\end{proof}

We shall only use the following special case of Theorem~\ref{thm:chabauty}. 

\begin{cor}\label{cor:closed_epsilon_sparse} 
Endow $\R^2/\Z^2$ with the flat riemannian metric coming from the standard euclidean metric of $\R^2$.
Let $\e$ be a positive real number. Then, there exists a finite set of proper closed 1-dimensional subgroups 
$L_j\subset \R^2/\Z^2$ 
such that  every subgroup $H$ of $\R^2/\Z^2$ which is not $\e$-dense is contained in one of the $L_j$.
\end{cor}

\begin{proof}[A direct proof of Corollary~\ref{cor:closed_epsilon_sparse}] 
As before we may assume that $H$ is a closed subgroup. 
We may assume $\e < 1/2$, because every non-trivial subgroup of $\R^2/\Z^2$ is $1/2$-dense. 
 The covering map $\pi\colon \R^2\to \R^2/\Z^2$
provides a diffeomorphism from the disk  $\disk((0,0) ; \e)$ to its image. Set $U'=\disk((0,0) ; \e/2)\subset \R^2$ and $U:=\pi(U')$.
For  any closed subgroup  $H$ of $\R^2/\Z^2$, denote by $V_\e(H)\subset \R^2$  the real vector space generated
by $\pi^{-1}(H)\cap U'=\pi\rest{U'}^{-1}(H\cap U)$. Since $V_\e(H)$ is generated by vectors of norm $\leq \e/2$, 
every point of $\pi(V_\e(H))$ lies within   distance $\leq \sqrt{2}\e/2$ from a point of $H$.
Thus, if $H$ is not $\e$-dense, then $\dim_\R(V_\e(H))\leq 1$.   

If $\dim_\R(V_\e(H))= 1$, $\pi(V_\e(H))$ 
is a closed subgroup that satisfies $\pi(V_\e(H))\cap U=\pi(V_\e(H)\cap U')$ (otherwise $\dim_\R(V_\e(H))$ would be equal to $2$).
This implies that $V_\e(H)$ is defined by a linear equation $px-qy=0$, where  $p$ and  $q$ are
 coprime integers   such that the subgroup of $\R/\Z$ generated by $p/q$ (resp. by $q/p$) is not $\e/2$-dense. 
 There are only finitely many possibilities for such pairs $(p,q)$. Moreover, the projection of $H$ in $(\R^2/\Z^2)/\pi(V_\e(H))$ is not $\e$-dense, which leaves only finitely many possibilities for $H$. 
 More precisely,  $H$ is contained in the kernel of the morphism $(x,y)\mapsto m(px-qy)$ for some $m\leq1/\e$. 
 
 The remaining  case to be considered  is when 
$H\cap U=\{(0,0)\}$. In this case, $H$ has at most $(\pi\e^2/4)^{-1}$ elements, so it 
 is contained in the kernel of $(x,y)\mapsto mx$ for some $m\leq(\pi\e^2/4)^{-1}$.  \end{proof}

%%%%%%%%%%%%%%%%%%%%%%%%%%%%
\subsection{Accumulation sets of orbits}
%%%%%%%%%%%%%%%%%%%%%%%%%%%%
The following result is a topological analogue of the analysis of Section~\ref{sec:Main_Proof}. It somehow says that there is no ``exceptional minimal set'' for the action of $\Gamma$ on $X$ 
(see \S~\ref{subs:minimal} for a more precise discussion on minimal subsets). 

\begin{pro}\label{pro:accumulation}
Let $X$ be a complex projective surface and let $\Gamma\subset\Aut(X)$ be a non-elementary subgroup containing parabolic elements. Let $F$ be an infinite closed invariant set which is not equal to $X$. 
Then, outside  $\STang_\Gamma$, the 
accumulation set $\mathrm{Acc}(F)$ is a  totally real 
analytic subset of dimension 2 with a locally finite singular set; moreover, 
 $F\setminus  \mathrm{Acc}(F)$ is a locally finite set outside   $\STang_\Gamma$. 
\end{pro}

\begin{proof}
Pick $x_0\in \mathrm{Acc}(F)\setminus \STang_\Gamma$. By definition of $\STang_\Gamma$, 
there is an  open neighborhood $U$ of $x_0$  and two elements $g$ and $h$ in $\Hal(\Gamma)$ 
such that $\pi_g(U)$ and $\pi_h(U)$ are open disks in $B_g^\circ$ and $B_h^\circ$, respectively, all fibers of $\pi_g$ and $\pi_h$ intersecting $U$ are smooth, and $\pi_g\times \pi_h$ is a diffeomorphism from $U$ 
to the product of these two disks. Observe that if $F$ 
 contains a fiber $X_w^g$ of $\pi_g$,  then $\pi_h(F)=B_h$, $F$ contains all fibers $X_{w’}^h$ for $w’\in B_h^\circ\setminus \Rk_h$, and finally $F=X$. Thus,  $F$ does not contain any fiber of $\pi_g$. From this 
it follows that $\pi_g(F)\cap B_g^\circ\subset \Rk_g$, and by switching the roles of $g$ and $h$ the same holds 
 for $h$. Furthermore by compactness there exists $\e>0$ such that for every $x\in F\cap U$ the  $g$-orbit of $x$ is not $\e$-dense in $X^g_{\pi_g(x)}$. 
Thus, by Corollary~\ref{cor:closed_epsilon_sparse}, $\pi_g(\Gamma(x))\cap U$ is contained in 
finitely many of the curves $R^k_{(p,q)}(\pi_g(U))$. 
The same property holds with respect to $\pi_h$. We conclude that $F\cap U$ is contained in 
finitely many squares $S_i\subset U$, each of them being of type $(\pi_g\times\pi_h)^{-1}(R^{k_i}_{(p_i,q_i)}\times R^{k'_i}_{(p'_i,q'_i)})$. 

Since $x_0$ is an accumulation point,  permuting $g$ and $h$ if necessary, we can find a sequence 
$(x_n)\subset F$ such that $(x_n)$ converges towards $x_0$ and
$\pi_g(x_n)$ is an infinite subset of $R^{k_i}_{(p_i,q_i)}$; 
it follows that for every $\e >0$, there is a point $x_n\in F$ close to $x_0$ whose $g$-orbit is
$\e$-dense in a circle of slope $(p_i,q_i)$ in $X^g_{\pi_g(x_n)}$; thus, there is a vertical segment of the form
 $S_{i_0}\cap X_y^g$, for some 
$y \in \pi_g(S_{i_0})$, which is contained in  $ \mathrm{Acc}(F)$. 
Using $h$, we see that the whole square $S_{i_0}$ is contained in $\mathrm{Acc}(F)$. 
This argument shows, more generally, that 
if $\vert S_i\cap F\vert = +\infty$, then $S_i\subset \mathrm{Acc}(F)$. Since $F$ is contained in 
finitely many squares $S_i$, we conclude that $\mathrm{Acc}(F)\cap U$ is a finite union of such squares and 
$(F\setminus\mathrm{Acc}(F))\cap U$ is finite, as asserted. 
  
In particular, $\mathrm{Acc}(F)\rest{U}$
is a real analytic, totally real subset of $U$. Its singular locus is the union of the pairwise intersections $S_i\cap S_j$ of distinct squares. If it contained a
curve, then permuting $g$ and $h$ if necessary, this curve would be a vertical segment $S_i\cap X_{w}^g$ for some $w\in R^{k_i}_{(p_i,q_i)}(\pi_g(U))$.
The $h$-orbit of such a segment would be dense in $S_i$, and this would contradict the invariance of  the singular set of $\mathrm{Acc}(F)$. Thus, all intersections $S_i\cap S_j$ 
are finite, and $\mathrm{Acc}(F)\rest{U}$ has only finitely many singularities. \end{proof}

\subsection{Orbit closures} 
When  $\STang_\Gamma = \emptyset$, 
Proposition~\ref{pro:accumulation}  leads 
to a rather satisfactory description of orbit closures. 

\begin{thm}\label{thm:orbit_closures}
Let $X$ be a complex projective surface. Let $\Gamma\subset \Aut(X)$ be a  
subgroup which is  non-elementary, contains    parabolic elements, 
 and %does not preserve any non-empty proper Zariski closed subset.  
 such that $\STang_{\Gamma}$ is  empty. Then there exists 
 a finite $\Gamma$-invariant subset $E$ such that   for any $x\in X\setminus E$, the orbit closure 
 $\overline{\Gamma(x)}$ is
   either  equal to $X$, 
  or to a  totally real, real analytic surface, with finitely many singularities.

 If furthermore  $X$ is a K3 or Enriques surface, and   $\Gamma$ 
 does not preserve any non-empty proper Zariski closed subset, then there exists 
 a smooth, totally real, real analytic surface $\Sigma\subset X$ such that:
  \begin{enumerate}[{\em (a)}]
 \item  if $x\notin \Sigma $ then 
 $\overline{\Gamma(x)} = X$, and 
 \item  if $x\in \Sigma $ then 
 $\overline{\Gamma(x)}$ is a union of connected components of $\Sigma $. 
 \end{enumerate}
\end{thm}

\begin{proof}
Since $\mathrm D_\Gamma\subset \STang_\Gamma$, 
 there is no $\Gamma$-invariant curve,  so Theorem~C of 
\cite{finite_orbits} guarantees that there is a maximal finite invariant set $E$.
Let $x\in X\setminus E$ be such that $\Gamma(x)$ is not dense in $X$. 
Since $\STang_{\Gamma}$ is empty, Proposition~\ref{pro:accumulation} entails that 
 $\overline{\Gamma(x)}=D\cup\Sigma$ where  $D$ is a finite  
 set and $\Sigma$ is a totally real surface with finitely many singularities, and both 
 $D$ and $\Sigma$ are $\Gamma$-invariant. Since $\Gamma(x)$ is infinite, $x$ does not belong to $D$,  hence $D$ is empty. This proves the first assertion. 
 
Under the assumptions of 
 the second part, we note that any invariant  real analytic surface $\Sigma$ is smooth, and 
admits a smooth $\Gamma$-invariant volume form $\vol_\Sigma$ (see Remark~\ref{rmk:real_area_form}).
By Theorem~\ref{thm:finiteness},   there are only finitely many such invariant measures,  which implies the 
existence of the maximal invariant real surface $\Sigma$. 
\end{proof}

%For the second part, let $\nu$ be a probability measure on $\Gamma$ with a finite exponential moment, and 
% whose support generates $\Gamma$. Then 
%if $\overline{\Gamma(x)} \neq X$, $\overline{\Gamma(x)}$ is a smooth $\Gamma$-invariant 
%surface $\Sigma$. Our assumption on $X$ implies that there is a $\Gamma$-invariant volume form 
%$\vol_\Sigma$ (see Remark~\ref{rmk:real_area_form}). Let now $\mu$ be any $\nu$-stationary measure on $\Sigma$. 
%Since there is no $\Gamma$-invariant curve, the stiffness theorem~\cite[Thm. A]{stiffness} asserts that 
%$\mu$ is $\Gamma$-invariant (and coincides with $\vol_\Sigma$). Furthermore since $X$ is not a torus and $\Gamma$ has no invariant curve,
% $\Gamma$ is not a Kummer group. So   there are only finitely many such invariant measures by  Theorem~\ref{thm:finiteness}, and we are done. 

\begin{rem} 
It is likely that the existence of the maximal invariant surface $\Sigma$ in the second part of the theorem 
could be obtained by a direct analysis of orbit closures instead of Theorem~\ref{thm:finiteness}, hence making it possible to extend 
the result to rational surfaces. 
This would however require to reproduce
most of Sections~\ref{sec:semi-analytic} and~\ref{sec:finiteness} in the topological category. 
%Note also that we do not know  any example of a non-elementary group of automorphisms on a rational surface which does not admit an  invariant curve. 
\end{rem}

\begin{rem}
Assume that $\Gamma$ admits no invariant curve but  
$\STang_\Gamma$ is a non-empty finite invariant subset. If $x\in X$ is such that $\Gamma(x)$ is infinite but not 
dense, then by Proposition~\ref{pro:accumulation}, outside $\STang_\Gamma$, $\overline{\Gamma(x)}$ is of the form $D\cup \Sigma$, where $D$ is discrete and $\Sigma$ is a totally real analytic surface. It is clear that 
only one of  $D$ and  $\Sigma$ can be   non-empty. If $\Sigma = \emptyset$, 
 $\overline{\Gamma(x)}$ accumulates 
only at $\STang_\Gamma$ and the problem is to analyse how it clusters near this finite set. 
 If $D = \emptyset$, the problem is to  describe the geometry of   
$\Sigma$ near $\STang_\Gamma$, and in particular the existence of finite invariant or stationary probability measures. We will develop tools to deal with 
these issues in~\cite{hyperbolic}.
\end{rem}

\subsection{Minimal invariant subsets}\label{subs:minimal} 
Recall that by definition a {\bf{minimal invariant subset}} $F\subset X$ is a non-empty 
closed $\Gamma$-invariant subset which is minimal for the inclusion; equivalently, 
a closed invariant subset $F$ is minimal if $\overline{\Gamma(x)}=F$ for every $x\in F$. 

When $\Gamma$ admits invariant Zariski closes subsets, we can still describe minimal invariant subsets. This is however not the right topological analogue of Theorem~\ref{thm:main} since closed invariant subsets are not unions of minimal ones.

\begin{lem}\label{lem:minimal}
Let $\Gamma$ be a non-elementary group of automorphisms of a complex projective surface.  
Let $F\subset X$ be a minimal invariant subset. If its Zariski closure $\mathrm{Zar}(F)$ has dimension 
$\leq 1$, then one of the following four possibilities occurs:
\begin{enumerate}[\em (a)]
\item $F$ is a finite orbit of $\Gamma$;
\item $\mathrm{Zar}(F)$ is a smooth $\Gamma$-invariant genus $1$ curve $C$ and either $F=C$ or
$F$ is a translate of a $1$-dimensional real subgroup of $C$ (in other words, it is a finite union 
of translates $x_i+L$ for some $1$-dimensional real torus $L\subset C$);
\item the components of $\mathrm{Zar}(F)$ are rational curves $C_i$ which are transitively permuted by 
$\Gamma$, the restriction of the stabilizer $\Gamma_i$ of $C_i$ to  $C_i$ is a relatively compact subgroup of $\Aut(C_i)$, its closure  $G_i\subset \Aut(C_i)$  is locally isomorphic  to $\PO_2(\R)$, 
and $F\cap C_i$ is a $G_i$-invariant circle or a $G_i$-invariant pair of circles;
\item the components of $\mathrm{Zar}(F)$ are smooth disjoint rational curves $C_i$ which are transitively permuted by 
$\Gamma$, the stabilizer $\Gamma_i$ of $C_i\simeq \P^1(\C)$ is non-elementary, and $F\cap C_i$ is the limit set of $\Gamma_i$.
\end{enumerate}
\end{lem} 

\begin{proof} The proof is elementary, modulo the following assertions; we leave the reader fill in the details. First, a non-elementary group can not preserve more than one irreducible curve of genus $1$, as  follows from 
Theorem 4.2 in~\cite{Diller-Jackson-Sommese}.
So if $\mathrm{Zar}(F)$ is reducible, its components are rational curves $C_i$
which  are  permuted transitively by $\Gamma$. Then,  if $C_i$ and $C_j$ are distinct components with non-empty intersection, the intersection $C_i\cap C_j$ 
is a finite set which  is invariant by $\Gamma_i = \mathrm{Stab}_{\Gamma}(C_i)$ and 
 cannot intersect $F$ by minimality. It follows that   $\Gamma_i$ is a relatively compact subgroup of $\Aut(C_i)$.  %\romain{a-t-on écarté le cas de 2 courbes rationnelles lisses  s'intersectant en  1 point?}\serge{Si tu as deux courbes rationnelles lisses qui se coupent en un point, les $\Gamma_i$ sont contenus dans le groupe affine mais ont une orbite qui n'accumule pas le point d'intersection (à l'infini). Donc un groupe relativement compact.}
\end{proof}

\begin{thm}\label{thm:minimal}
Let $X$ be a complex projective surface and $\Gamma\leq \Aut(X)$ be a non-elementary  subgroup  
 containing     parabolic elements. 
Let $F\subset X$ be a minimal invariant subset of $\Gamma$. 
If $F$ is not Zariski dense, then $F$ is as in Lemma~\ref{lem:minimal}. Otherwise, 
 $F$ is either equal to $X$ or equal to a smooth real analytic  totally real  surface..
\end{thm}

\begin{proof}
Assume that $F$ is Zariski dense and not equal to $X$. Then $F$ is disjoint from $\STang_\Gamma$, 
and by minimality $F = \mathrm{Acc}(F)$, so it follows from Proposition~\ref{pro:accumulation} that 
$F$ is a totally real  surface, which must be smooth because otherwise its singular locus would be a 
non-trivial proper closed invariant subset of $F$. 
\end{proof}

\section{Invariant analytic surfaces which are not real parts}\label{sec:not_real_parts}
%%%%%%%%%%%%%%%%%%%%%%%%%%%%%%%%%%%%%%%%%%%%%%
%%%%%%%%%%%%%%%%%%%%%%%%%%%%%%%%%%%%%%%%%%%%%%

In this section we construct   examples of pairs $(X,\Gamma)$ such that Property (c) in Theorem~\ref{thm:main} holds and for which:
\begin{enumerate}[(1)]
\item  the support $\Sigma$ of $\mu$ is an analytic real surface;
\item there is no real structure  on $X$ for which 
$\Sigma$ is   contained in the real part~$X(\R)$. 
\end{enumerate} 

\subsection{A family of lattices}\label{par:subsection-lattice} Let $t$ be a positive real number, and set $\tau=\frac{1}{2}+\ii t$, where $\ii=\sqrt{-1}$. 
Consider the lattice $\Lambda\subset \C$ defined by 
\begin{equation}
\Lambda=\Z\oplus \Z (\frac{1}{2}+\ii t) = \Z\oplus \Z \tau.
\end{equation}
Since $\frac{1}{2} - \ii t$ belongs to $\Lambda$, the complex conjugation $z\mapsto \overline{z}$ induces an anti-holomorphic involution
 $\sigma_E(z)=\overline{z}$ on
the elliptic curve $E=\C/\Lambda$:  this gives a real structure on $E$. The fixed point 
set of $\sigma_E$ gives the real part of $E$; one checks easily that 
it coincides with the projection of the real axis:
\begin{equation}
E(\R)= \R/\Z= \R/(\Lambda\cap \R).
\end{equation}

\subsection{Abelian and Kummer surfaces}\label{par:subsection-kummer} Now, consider the abelian surface $A:=E\times E=\C^2/\Lambda^2$, and the real structure
$\sigma_A(x,y)=({\overline{x}}, {\overline{y}}) \mod (\Lambda^2)$. Its fixed point set is $A(\R)=E(\R)^2=\R^2/\Z^2$. 
The group $\GL_2(\Z)$ acts linearly on $\C^2$ by preserving $\Lambda^2$, so it 
 also acts  on $A$ by ``linear automorphisms''. 
Set 
\begin{equation}
\bfe_1=(1,0), \; \bfe_2=(\tau, 0), \; \bfe_3=(0,1),\; \bfe_4=(0,\tau).
\end{equation}
The vectors $\bfe_1$ and $\bfe_3$ form a basis of the complex plane $\C^2$, and the four vectors $\bfe_1$, $\ldots$, $\bfe_4$ 
form a basis of the real vector space $\C^2\simeq \R^4$. The real planes $\Vect_\R(\bfe_1,\bfe_3)$ and $\Vect_\R(\bfe_2,\bfe_4)$ are 
invariant under the action of $\GL_2(\Z)$, and they determine two invariant real tori in $A$,   the first one being 
equal to $A(\R)$.  Define 
\begin{equation}
\Sigma_A= \Vect_\R(\bfe_2, \bfe_4)/\Lambda^2= \Vect_\R(\bfe_2, \bfe_4)/ (\Z\bfe_2\oplus \Z\bfe_4)\simeq\R^2/\Z^2.
\end{equation}

Now, consider the Kummer surface $X_0=A/\eta$, where $\eta$ is the holomorphic involution of $A$ defined by 
$\eta(x,y)=(-x,-y) \mod(\Lambda^2)$. This surface has sixteen singularities, each of which is resolved by a simple
blow-up. Let $\Sigma_0$ be the projection of $\Sigma_A$ in $X_0$. We also 
 let $\Sigma_X \subset X$ be the proper transform  of $\Sigma_0$  
  in the smooth K3 surface $X$ obtained by resolving the singularities of $X_0$. %\serge{être plus précis}\romain{plus précis sur la notion de transformée stricte? tel quel ça me va}

\subsection{Statements} 

\begin{thm} With notation as above, the following properties are equivalent:
\begin{enumerate}[\em (a)]
\item there is a real structure $s_A$  on $A$   whose real part contains 
$\Sigma_A$; 
\item there is a real structure $s_X$  on     $X$ whose real part contains 
 $\Sigma_X$;   
\item the positive real number $t$ is equal to $\sqrt{3}/2$, $1/2$, or $\sqrt{3}/6$.
\end{enumerate}
In addition when $t=1/2$ (resp. $\sqrt{3}/2$ or $\sqrt{3}/6$) the curve $E$ is isomorphic to the quotient of $\C$ by the lattice of Gaussian integers (resp. Eisenstein integers).
\end{thm}

The equivalences (a)$\Leftrightarrow$(c) and (b)$\Leftrightarrow$(c) are proven respectively 
 in \S~\ref{subs:proof_abelian} and~\ref{subs:proof_kummer} below.

\begin{cor}\label{cor:not_real_parts} 
There exist examples of abelian and Kummer surfaces $X$, and non-elementary subgroups $\Gamma\subset \Aut(X)$ 
such that {\emph{(1)}} $\Gamma$ preserves an analytic, totally real surface $\Sigma\subset X$, {\emph{(2)}} $\Gamma$ has 
a dense orbit and an invariant, ergodic, smooth, probability measure with support equal to $\Sigma$, and {\emph{(3)}} there is no real structure on $X$ whose real part  contains $\Sigma$.
\end{cor}

 Indeed, one can just take $\Gamma = \GL_2(\Z)$ in the previous examples, and for the invariant probability measure 
one takes the measure coming from the Lebesgue measure on $\Sigma_A$.
%If one replaces $\GL_2(\Z)$ by a non-elementary subgroup $\Gamma \subset \GL_2(\Z)$, the existence of dense
%orbits and the ergodicity of the Lebesgue measure are still satisfied (this follows from results \cite{Muchnik, Guivarch-Starkov, BFLM} for instance).\romain{intérêt de cette remarque? pas vraiment un nouvel exemple, c'est juste un sous-groupe du précédent}
%\serge{J'ai enlevé la phrase sur le passage à des sous-groupes.}

\subsection{Proof for the abelian surface}\label{subs:proof_abelian}
 Let $s$ be a real structure on $A$, i.e.\ an anti-holomor\-phic 
involution. Then $s\circ \sigma_A$ is holomorphic, so $s=B\circ \sigma_A$ for some automorphism $B$ 
of $A$. Now, assume that the fixed point set of $s$ is equal to $\Sigma_A$. Since the origin $(0,0)$ 
of $A$ is fixed by $\sigma_A$ and $s$, it is fixed by $B$ too. This means that $B$ is induced by a linear
automorphism of the complex vector space $\C^2$, i.e.\ by an element of $\GL_2(\C)$.

Near the origin of $A$, we can write $s(x,y)=B\circ\sigma_A(x,y)=B({\overline{x}}, {\overline{y}})$, and our 
assumption implies that the vector $\bfe_2=(\tau, 0)$ satisfies $B({\overline{\tau}}, 0)=(\tau,0)$, i.e. 
$B\bfe_2=(\tau/{\overline{\tau}})\bfe_2$ because $B$ is $\C$-linear. The same property is satisfied 
by $\bfe_4$. Since $(\bfe_2,\bfe_4)$ is a basis of the complex plane $\C^2$, $B$ is a homothety: $B=(\tau/{\overline{\tau}})\mathsf{Id}$. 
As a consequence, the linear map $M\colon \C\to \C$ defined by $M(z)=(\tau/{\overline{\tau}})z$ preserves the lattice 
$\Lambda=\Z\oplus\Z\tau$,
with $\tau=1/2+\ii t$. Thus, we can find a quadruple of integers $(a,b,c,d)$ such that 
\begin{equation}
\frac{\tau}{{\overline{\tau}}} = a + b \tau \; \,  {\text{ and }} \, \; \frac{\tau}{{\overline{\tau}}}\tau= c + d \tau.
\end{equation}
This implies $a=c=-1$, by looking at the imaginary parts after both equations have been multiplied by ${\overline{\tau}}$. 
Then 
\begin{equation}
1 = b \vert \tau \vert^2   \quad  {\text{and}} \quad  \frac{3}{4} - t^2= d\vert \tau\vert^2 .
\end{equation}
The first equation plus the relation $\vert \tau \vert^2= 1/4 + t^2 > 1/4$ imply that 
$1\leq b \leq 3$, and more precisely 
$(b,t)\in \{ (1,\sqrt{3}/2), (2, 1/2), (3, \sqrt{3}/6)\}$. Together with the second equation we end up with exactly 
three possibilities:% $(b,d,t)=(1,0, \sqrt{3}/2)$, $(2, 1, 1/2)$ or $(3,2,\sqrt{3}/6)$.
\begin{enumerate}[(1)]
\item  $(b,d,t)=(1,0, \sqrt{3}/2)$, the lattice $\Z\oplus \Z\tau$ is the lattice of Eisenstein integers $\Z[(1+\ii \sqrt{3})/2]$, and $s(x,y)=(e^{2\ii \pi/3}{\overline{x}}, e^{2\ii \pi/3}{\overline{x}})=e^{2\ii \pi/3}\sigma_A(x,y)$; 
\item $(b,d,t)=(2, 1, 1/2)$, $\tau = (1+ \ii)/2$, the lattice is the image of the lattice of Gaussian integers $\Z[\ii]$
  by the  similitude $z\mapsto \frac{1-\ii}{2}z$, and $s(x,y)=(\ii {\overline{x}} , \ii {\overline{y}} ) = \ii \sigma_A(x,y)$;
  \item $(b,d,t)=(3,2,\sqrt{3}/6)$ and $\tau=1/2+\ii \sqrt{3}/6$; modulo the action of $\PSL_2(\Z)$ on the upper half plane, $\tau$ is 
equivalent to $2-1/\tau=(1+\ii \sqrt{3})/2$ so we end up again with the Eisenstein integers, and $s(x,y)=(e^{2\ii \pi/6}{\overline{x}}, e^{2\ii \pi/6}{\overline{y}})=e^{2\ii \pi/6}\sigma_A(x,y)$
\end{enumerate} 
This completes the proof of the implication (a)$\Rightarrow$(c), and for the converse implication 
the explicit formulas for $s(x,y)$ provides the desired real structure. 
 
In each case, $s$ is conjugate to $\sigma_A$ by $\beta\id\in \GL_2(\C)$, with  respectively $\beta= e^{2\ii\pi/6}$, $e^{2\ii\pi/8}$, $e^{2\ii\pi/12}$. However
 in the second and third cases, this conjugacy is only satisfied near the origin, because 
 $\beta\id $ does not preserve the lattice (i.e.does not induce an automorphism of $A$).   \qed

 %\romain{on devrait pouvoir dire  que les actions sur $\Sigma$ et $A(\R)$ ne sont pas conjuguées dans $\Aut(A)$. J'ai l'impression que dans ces cas pour des raisons de  volume il n'y a pas d'automorphisme de $A$ qui envoie $A(\R)$ sur $\Sigma$. On pourrait faire une remarque. }\serge{?? S'il existait un automorphisme $f$ qui envoie $A(\R)$ sur $\Sigma_A$ alors en conjuguant $\sigma_A$ par $f$ on aurait une structure réelle qui fixerait $\Sigma_A$. Donc c'est un corollaire évident de notre theoreme. Par contre, les actions sont les mêmes à difféo analytique réel près ($\GL_2(\Z)$ sur $\R^2/\Z^2$)}\romain{je parlais des 3 cas (1) (2) (3) ci dessus, en rapport avec la conjugaison par $\beta I$, qui donne un tel automorphisme dans le cas (1) et je voulais arguer qu'il n'y en avait pas dans les 2 autres cas. Bref, c'est sans importance.}

\subsection{Proof for the (smooth) Kummer surface}\label{subs:proof_kummer}
 Suppose  there is a   real structure $s_X$ on $X$
whose   fixed point set contains $\Sigma_X$. Let $\sigma_X$ be the real structure induced by $\sigma_A$
on the Kummer surface $X$. Then, there is an automorphism $B_X$ of $X$ such that $s_X=B_X\circ \sigma_X$. 

Consider the origin $(0,0)$ of $A$, and its blow-up $\e\colon \hat{A}\to A$. In local coordinates it expresses as
$\e (u,v)=(u,uv)=(x,y)\in A $, with exceptional divisor $D=\{u=0\}$. 
The involution $\eta$ lifts to $\hat{\eta}(u,v)=(-u,v)$; it is the
identity on $D$, it acts  transversally as $u\mapsto -u$, and the quotient map $\hat{A}\to X=\hat{A}/\hat{\eta}$ is locally 
given by $q\colon (u,v)\mapsto (u^2,v)$.
Lifting $\sigma_A$, we obtain ${\hat{\sigma}}_A(u,v)=(\overline{u}, \overline{v})$. 
In $A$, $\Sigma_A$ is locally parametrized by $(s\tau, s'\tau)$ with $s$ and $s'$ small real numbers; 
its strict transform is the
 real analytic surface ${\hat{\Sigma}}_A$ given by $(u,v)=(s\tau, s'/s)$. So, the intersection 
of ${\hat{\Sigma}}_A$ with $D$ is determined by the equation $v\in \R$. In particular, ${\hat{\Sigma}}_A\cap D$
is fixed by ${\hat{\sigma}}_A$. The image of $D$ in $X$ is a curve $C\simeq \P^1(\C)$ and the 
image of the subset $\{u=0, \; v\in \R\}$ is a great circle $S\subset C$. This circle is fixed by $\sigma_X$, as we
just saw, and by $s_X$, by definition of $\Sigma_X$.
Thus, $B_X$ fixes $S$, hence $C$ itself since $S$ is Zariski dense in $C$ (for the complex algebraic structure on $X$).

Since $B_X$ fixes $C$, we can contract $C$ onto a singularity of $X_0$:  $B_X$ and $\sigma_X$ descend 
to regular (holomorphic and anti-holomorphic) maps on a neighborhood of the singularity. Since this singularity is
the quotient singularity $(\C^2, 0)/ \eta$, we can lift
%\romain{c'est quoi déja la propriété de la singularité qui permet de relever? (je devrais savoir...)}\serge{Tu relèves au revêtement universel local dans le complément de la singularité (revêtement double par $\C^2$ privé de l'origine, localement). Puis tu étends par Hartogs.}
$B_X$, $\sigma_X$, and $s_X$ to germs of diffeomorphisms $B_A$, $\sigma_A$, and $s_A$ near 
the origin in $\C^2$. Using the natural, local coordinates given by the projection $\C^2\to A\to X_0$, the lifts
can be written 
\begin{equation}
\sigma_A(x,y)=({\overline{x}}, {\overline{y}}) \; {\text{ and }} \; B_A(x,y)=(\sum_{k,\ell\geq 0} a_{k,\ell} x^k y^\ell, \sum_{k,\ell\geq 0} b_{k,\ell} x^k y^\ell)
\end{equation}
for some  locally convergent   power series $\sum_{k,\ell} a_{k,\ell} x^k y^\ell$ and $\sum_{k,\ell} b_{k,\ell} x^k y^\ell$.
In these coordinates, $\Sigma_X$ corresponds to the real plane 
$(u \tau, v\tau)$ for $(u,v)\in \R^2$, and the equation for the fixed points of $s$ gives $B_A(u{\overline{\tau}}, v{\overline{\tau}})
=(u \tau, v\tau)$ for $(u,v)\in \R^2$. This implies that $B_A$ is linear in these coordinates, equal to the homothety of 
factor $\tau/{\overline{\tau}}$.
% \serge{Reprendre à partir d'ici}\romain{j'ajoute juste quelques détails je ne sais pas  ce que  tu as en tête}

Up to this point, we have worked locally near the singularity of $X_0$ corresponding to the origin of $A$, we now globalize the argument. 
As a consequence of its local form, $B_A$ preserves the horizontal line $\{y=0\}$, 
so that $B_X$ preserves the   quotient curve 
\begin{equation}
(\C/(\Lambda)\times \{ 0\} )/\eta \simeq \P^1(\C)\subset X.
\end{equation}
Furthermore,  in the coordinate $x$ given by the projection 
\begin{equation}
\C\times\{0\} \to E\times\{0\} =\C/\Lambda\times \{0\}\to \P^1(\C)
\end{equation}
 (where the last arrow is a   branched cover of degree 2),
$B_X$ is  covered by the linear map $x\mapsto (\tau/{\overline{\tau}})x$. 
Thus the analysis of the previous subsection applies, and shows that 
 $\tau=1/2+\ii \sqrt{3}/2$, $(1+ \ii)/2$, or
$1/2+\ii \sqrt{3}/6$, and the proof of (b)$\Rightarrow$(c) complete.  For the converse implication, 
 it is enough to observe that for each of these three cases, 
  the explicit anti-holomorphic involution on $A$ 
  given in \S~\ref{subs:proof_abelian} commutes with $\eta$, so it descends to the Kummer surface $X$. 
\qed 

%%%%%%%%%%%%%%%%%%%%%%%%%%%%%%%%%%%%%%%%
%%%%%%%%%%%%%%%%%%%%%%%%%%%%%%%%%%%%%%%%
\section{Invariant surfaces with boundary}\label{sec:boundary}
%%%%%%%%%%%%%%%%%%%%%%%%%%%%%%%%%%%%%%%%
%%%%%%%%%%%%%%%%%%%%%%%%%%%%%%%%%%%%%%%%

In this section, we show that in case (c) of Theorem~\ref{thm:main}, the surface $\Sigma$
 may have a non-trivial boundary in $X$.
We provide two examples, one for a Kummer surface, and then a deformation keeping the main features of the first example 
but on a surface which is not anymore a Kummer surface. We also give  
 examples of invariant curves that do not support 
any invariant measure. 

%%%%%%%%%%%%%%%%%%%%%%%%%%%%%%%%%%%%%%%%
\subsection{On a Kummer surface}
%%%%%%%%%%%%%%%%%%%%%%%%%%%%%%%%%%%%%%%%

Consider a Kummer example,  with the same construction as in Sections~\ref{par:subsection-lattice} and~\ref{par:subsection-kummer}.
Embed the curve $E=\C/\Lambda$ in $\P^2$, in a Weierstrass form. Its equation is 
\begin{equation}
y^2=4x^3-g_2x-g_3
\end{equation}
with coefficients $g_i\in \R$ depending on the parameter $t$; the real structure $\sigma_E(z)=\overline{z}$ 
is the restriction to $E$ of the real structure $[x:y:z]\mapsto [{\overline{x}}: {\overline{y}}: {\overline{z}}]$ 
on $\P^2$. Since  $E(\R)={\mathrm{Fix}}(\sigma_E)$ is connected, $g_2$ is negative. By convention, we fix 
the origin of the elliptic curve for its group law  at the (inflexion) point at infinity. 
If $u$ and $v$ are two points of $E$, the line containing $u$ and $v$ intersects $E$ in a third point $w$. The sum $u+v+w$ 
is zero for the group law. If $u=(x,y)$ is a point of $E$, then $-u=(x,-y)$ and  the fixed points of this involution $u\mapsto -u$ 
on $E(\R)$ are the two points $(x_0,0)$ and $(\infty,\infty)$ where $x_0$ is the unique real solution of the equation 
$4x^3=g_2 x + g_3$. 

Now, consider the map $\Phi\colon E\times E\to \P^1\times \P^1\times \P^1$ 
which is defined as follows: if $(u,v)$ belongs to 
$E\times E$, with $u=(x_1,y_1)$ and $v=(x_2,y_2)$, and if $w=-(u+v)=(x_3,y_3)$, then $\Phi(u,v)=(x_1,x_2,x_3)$. 
One easily checks that $\Phi\circ \eta =\Phi$, with $\eta(u,v)=(-u,-v)$, and $\Phi$ embeds the Kummer surface $X_0=(E\times E)/\eta$
into $\P^1\times \P^1\times \P^1$  as a singular $(2,2, 2)$-surface  (see~\cite[\S 8.2]{Cantat:Panorama-Synthese}). The singularities of $X_0$ correspond to the fixed points of $\eta$, i.e. to the group $A[2]$ of torsion points 
of order $2$ in $A=E\times E$. This gives $16$ points, of which only $4$ are real: 
\begin{equation}
(x_0,x_0,\infty),  \; (x_0,\infty,x_0), \; (\infty, x_0,x_0), \;  (\infty,\infty,\infty).
\end{equation}
The real part of $X_0$ corresponds to the fixed point set of $\sigma_0$, {i.e.} of $\sigma_A$ viewed on the quotient space
$X_0$. Recall that $\sigma_A(z_1,z_2)=(\overline{z_1},\overline{z_2})$ if we think of $A$ as $\C^2/\Lambda^2$. 
The fixed points of $\sigma_0$ are of two types: those coming from the fixed points of $\sigma_A$, hence from the 
real part $A(\R)$, and those coming from 
\begin{equation}
\Delta:=\{ (z_1,z_2)\in A\; ; \; \sigma_A(z_1,z_2) = \eta(z_1,z_2)\}.
\end{equation}
In the quotient $X_0$, $A(\R)$ projects onto a sphere with four singularities; the projection of $\Delta$ is another sphere 
with the same four singularities. These two spheres are glued along those four points; locally $X_0$ is the quotient of $\C^2$ by $\set{\id, -\id}$, so  up to an analytic 
change of coordinates,    $X_0(\R)$ 
is a quadratic cone isomorphic to $x_1x_2=x_3^2$.

\begin{figure}[h]
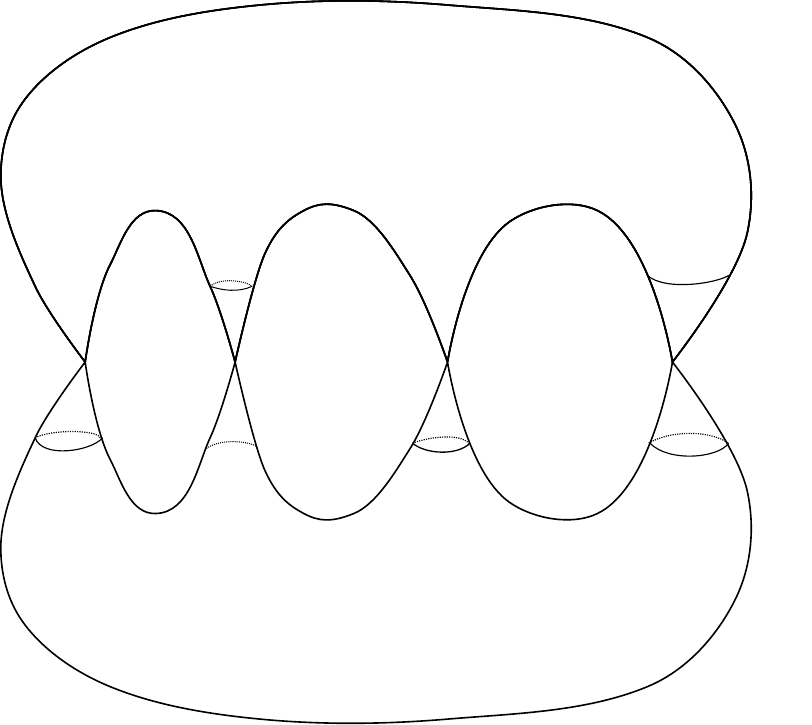
\end{figure}

The natural action of  $\GL_2(\Z)$ on $E\times E$ descends to an action of 
  $\PGL_2(\Z)$   on $X_0$, which preserves  $X_0(\R)$; each of the two connected 
components of $X_0(\R)\setminus \Sing(X_0(\R))$ is also preserved. The action of 
$\PGL_2(\Z)$ on these two punctured spheres has 
dense orbits (and finite orbits too, corresponding to torsion points of $A$). If we resolve the singularities of $X_0(\R)$, the two punctured spheres become 
two surfaces homeomorphic to a sphere minus four disks; they are glued together along their boundaries to form 
a closed, orientable surface of genus $3$, which is 
the real part of $X(\R)$ for the real structure $\sigma_X$. Thus, the generic 
orbit of $\PGL_2(\Z)$ in $X(\R)$ is dense in one of these two open subsets of $X(\R)$. This gives a first example of an invariant surface with boundary 
 $\Sigma$, given by the component $A(\R)/\eta$, with an invariant measure  $\mu$ given by
  the push forward of the Haar measure from $A(\R)$ to~$\Sigma$.

%%%%%%%%%%%%%%%%%%%%%%%%%%%%%%%%%%%%%%%%
\subsection{Deformation}\label{par:examples_deformation}
%%%%%%%%%%%%%%%%%%%%%%%%%%%%%%%%%%%%%%%%
Let us come back to $X_0(\R)$. Switching  the 
chart in    $\P^1\times \P^1\times \P^1$ so that the coordinates  $(x_1,x_2,x_3)$ are replaced by 
their inverses $(1/x_1,1/x_2,1/x_3)$, the four singularities become  
\begin{equation}
(\alpha,\alpha,0),  \; (\alpha,0,\alpha), \; (0, \alpha,\alpha), \;  (0,0,0)
\end{equation}
with $\alpha=1/x_0$.
Note that the three vectors $v_1=(\alpha,\alpha,0)$, $v_2=(\alpha,0,\alpha)$, and $v_3=(0, \alpha,\alpha)$ are linearly independent
(their determinant is $-2\alpha^3$). Thus, given any triple $(\e_1,\e_2,\e_3)\in \{\pm 1\}^3$, there
is a real quadratic form $Q(x_1,x_2,x_3)$ such that $\e_iQ(v_i)>0$ for each $1\leq i\leq 3$. 
If $P$ denotes the equation of $X_0$ (after changing the $x_i$ in $1/x_i$ as above) and $\e$ is a small 
real number, then $P+\e Q$ is an equation of a new surface $X_\e$ of degree $(2,2,2)$ in $\P^1\times \P^1\times\P^1$.
At the origin $(0,0,0)$, the linear term of the equation $P=0$ is not changed  by the addition of $\e Q$, and the quadratic term is only slightly perturbed  if $\e$ is small enough,  so $X_\e$ still admits 
a quadratic singularity, which    is non-degenerate of signature $(2,1)$. At the other three real 
singularities of $X_0$, we can choose the sign of $Q(v_i)$ in such a way that $X_\e(\R)$ is locally disconnected 
(as does  a hyperboloid with two sheets). 

We claim that, shifting $Q$ a little bit if necessary, all (real or complex) singularities of $X_0$ disappear
in the perturbation $X_\e$, except the origin. Indeed, by conjugating by  a
diagonal automorphism $(h,h,h)\in \Aut(\P^1)^3$, such that 
  $h\in \PGL_2(\C)$   fixes $0$ and $\alpha$, we can arrange that all singular points of $X$ belong to 
  $\C^3$. Let  $(s_j)_{j=0,\ldots, 15}$ be the singular points of $X_0$  
  (with $s_0 = 0$) and 
 choose $Q$ in such a way that $Q(s_j)\neq 0$ for $j\geq 1$. Let  
  $N = \bigcup_j B(s_j, \eta)$, where $\eta$ is so small that $Q\neq 0$ on $\bigcup_{j\geq 1} B(s_j, \eta)$. 
 Take $\e$ to be small and non-zero. Then,$X_\e$ is smooth in $\bigcup_{j\geq 1} B(s_j, \eta)$, because its
  equation $P+\e Q = 0$ reduces to $P/Q+\e = 0$ there, and such a hypersurface is smooth for $\e\neq 0$ small. 
Finally, smoothness being an open property,  $X_\e$ is also smooth in the complement of $N$.

\begin{figure}[h]
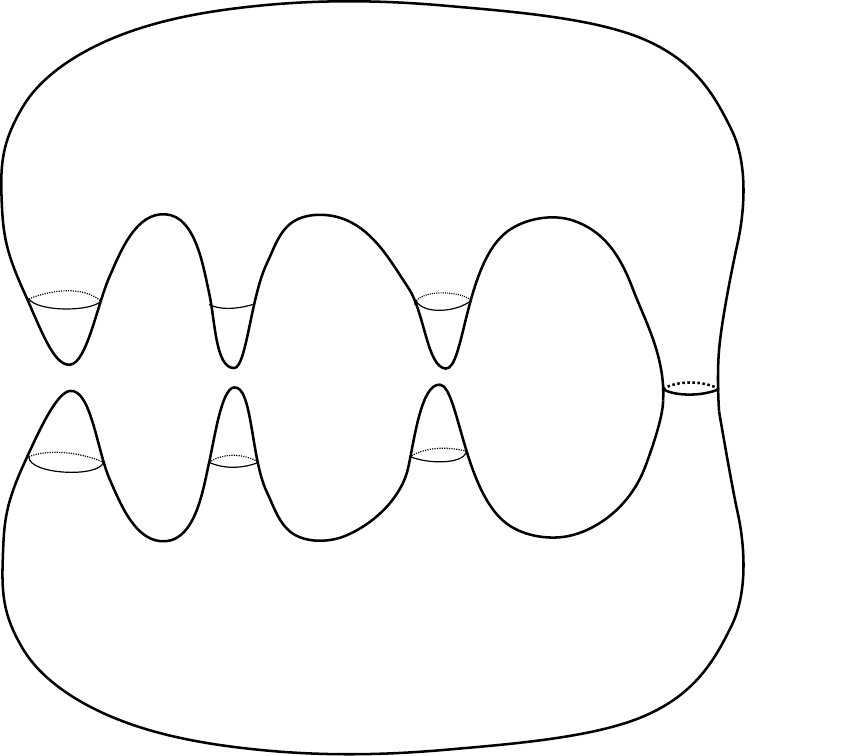
\end{figure}
\begin{rem} Such a deformation appears naturally in the closely related example of character
varieties for the once punctured torus (see~\cite{Cantat:BHPS, Goldman:1988}). \end{rem}

The result, for a sufficiently small $\e$ and a good choice of $Q$, is a real surface $X_\e$ of degree $(2,2,2)$
in $\P^1\times \P^1\times \P^1$ with a real part $X_\e(\R)$ satisfying the following properties.
\begin{enumerate}[--]
\item The surface $X_\e(\R)$ has a unique singularity,  at the origin.
\item After a minimal resolution of the singularity, we get a real K3 surface %\romain{pourquoi est-ce K3?}\serge{J'ai ajouté une vague idée de preuve. Sinon je peux donner une référence, ou détailler}\romain{c'est convaincant, mais si tu as une référence c'est encore mieux!}
${\hat{X}}_\e$. Indeed, the area form $\Omega_{X_\e}$ defined in Example~\ref{eg:wehler} lifts to 
a trivialization of the canonical bundle of ${\hat{X}}_\e$; and ${\hat{X}}_\e(\C)$ is simply connected (to 
see this one can use the fibrations ${\hat{\pi_i}}\colon {\hat{X}}_\e(\C)\to X_\e(\C)\to \P^1(\C)$).

\item The surface ${\hat{X}}_\e(\R)$ is homeomorphic to  a sphere ${\mathbb{S}}^2$; the exceptional divisor is a curve 
$S_\e\subset {\hat{X}}_\e(\R)^0$ that separates ${\hat{X}}_\e(\R)$ in two one-holed spheres with boundary $S_\e$;
\item The three involutions $\sigma_i\colon X_\e\to X_\e$ described in Example~\ref{eg:wehler} lift to three 
automorphisms of ${\hat{X}}_\e$. They generate a non-elementary subgroup $\Gamma_\e$
of $\Aut({\hat{X}}_\e)$ that preserves the real structure (the compositions $\sigma_i\circ \sigma_j$ are parabolic 
automorphisms with respect to distinct fibrations). A subgroup of index $4$ preserves simultaneously each component of ${\hat{X}}_\e(\R)\setminus S$ 
and the canonical area form of the K3 surface. 
\end{enumerate}
This provides {\sl{examples with non-trivial invariant open subsets of $X(\R)$ for a non-elementary subgroup of $\Aut(X_\R)$}}, in a case where
$X$ is not a Kummer surface (the dynamics of the group $\Gamma_\e$ is not covered by a linear dynamics on a torus).

%%%%%%%%%%%%%%%%%%%%%%%%%%%%%%%%%%%%%%%%
\subsection{Invariant curves} 
%%%%%%%%%%%%%%%%%%%%%%%%%%%%%%%%%%%%%%%%
Let us continue with the example given by $X_\e$. 

When $\e=0$, $X_0$ is a singular Kummer surface, and its
singularities are in $1$ to $1$ correspondance with the elements of $A[2]$. Let $\Gamma(2)\subset \GL_2(\Z)$ 
be the subgroup that fixes $A[2]$ pointwise. Its image in $\Aut(X_0)$ preserves the $16$
singularities of $X_0$. This group lifts to a group of automorphisms $\Gamma_0\subset \Aut(\hat{X}_0)$.
Let $S_0\subset {\hat{X}}_0$ be the $(-2)$-curve obtained by the minimal resolution of one of the singularities of $X_0$.
The dynamics of $\Gamma$ on $S_0$ coincides, up to conjugacy, with  the linear projective action of  $\Gamma(2)\subset \GL_2(\Z)$
on $\P(\C^2)$. This is a non-elementary subgroup of $\Aut(\P^1)\simeq \PGL_2(\C)$; in particular, this action does not preserve any probability measure. 

Now, consider the small perturbation ${\hat{X}}_\e$ and the group $\Gamma_\e$, as constructed in \S~\ref{par:examples_deformation}. 
Then, $\Gamma_\e$ preserves the $(-2)$-curve $S_\e\subset {\hat{X}}_\e$, and for $\e =0$ we recover $\Gamma_0$ up to finite 
index. Since the non-elementary property of $\Gamma_{0}\rest{S_0}\subset \PGL_2(\C)$ is invariant under small perturbations or after
taking finite index subgroups, we deduce that $\Gamma_\e$ induces a non-elementary subgroup of $S_\e$.
Thus, we obtain {\emph{examples of K3 surfaces ${\hat{X}}_\e$ with
of a non-elementary subgroup $\Gamma_\e\subset \Aut({\hat{X}}_\e)$ such that $\Gamma$ preserves a smooth rational curve $S_\e\subset {\hat{X}}_\e$ but
$S_\e$ does not support any $\Gamma_\e$-invariant probability measure}}; here, the examples are deformations of a Kummer example $({\hat{X}}_0, \Gamma_0)$.

\begin{rem}
Similar examples can be constructed on some Coble surfaces~$Y$:  $\Aut(Y)$ preserves a rational curve, coming from 
a plane sextic with ten nodes; and then, taking a  double cover of $Y$ ramified along the invariant sextic, one gets K3 surfaces.
\end{rem}

\begin{rem}
Consider the above example $({\hat{X}}_\e, S_\e, \Gamma_\e)$ and a probability measure $\nu$ on $\Gamma_\e$ 
whose support is finite and generates $\Gamma_\e$.   Then, the curve $S_\e\simeq \P^1(\C)$ supports a unique $\nu$-stationary 
measure $\mu_\nu$, because $\Gamma_\e\subset \PGL_2(\C)$ is non-elementary (see~\cite{Furstenber:1963}). 
If, as above, everything is defined over $\R$, the support of $\mu_\nu$ is contained in a circle; but if we apply the same construction 
with a well chosen, small complex deformation $X_\e$, the support of $\mu_\nu$ is supported on a fractal quasi-circle.
\end{rem}

%%%%%%%%%%%%%%%%%%%%%%%%%%%%%%%%%%%%%%%%%%%%%%%%%%%%%%%%%%%%%%%%%%
%%%%%%%%%%%%%%%%%%%%%%%%%%%%%%%%%%%%%%%%%%%%%%%%%%%%%%%%%%%%%%%%%%
\appendix
%%%%%%%%%%%%%%%%%%%%%%%%%%%%%%%%%%%%%%%%%%%%%%%%%%%%%%%%%%%%%%%%%%
%%%%%%%%%%%%%%%%%%%%%%%%%%%%%%%%%%%%%%%%%%%%%%%%%%%%%%%%%%%%%%%%%%

{\small {

%%%%%%%%%%%%%%%%%%%%%%%%%%%%%%%%%%%%%%%%%%%
%%%%%%%%%%%%%%%%%%%%%%%%%%%%%%%%%%%%%%%%%%%
\section{Abelian surfaces}\label{app:abelian}
%%%%%%%%%%%%%%%%%%%%%%%%%%%%%%%%%%%%%%%%%%%
%%%%%%%%%%%%%%%%%%%%%%%%%%%%%%%%%%%%%%%%%%%

In this appendix, we consider the case when all parabolic automorphisms $g$ of $\Gamma$ induce an automorphism
$g_B$ of infinite order on the base of their invariant fibration $\pi_g$. In that case, we know from~\cite[Proposition~3.6]{Cantat-Favre} that $X$ is a compact
torus, and  in fact an abelian surface since $\Gamma$ is non-elementary. Thus, we assume that 
\begin{enumerate}[(i)]
\item $X$ is an abelian 
surface, isomorphic to $\C^2/\Lambda$ for some lattice $\Lambda$;
\item  $\Gamma$ is a non-elementary group 
of automorphisms of $X$ that contains a parabolic element $g$; 
\item every parabolic element $g$ of $\Gamma$ acts on the base of its invariant fibration 
$\pi_g\colon X\to B_g$ by an automorphism $g_B\colon B_g\to B_g$ of infinite order. 
\end{enumerate}
We provide an argument to complete the proof of Theorem~\ref{thm:main} in that case; the strategy is 
the same as in Sections~\ref{par:Main Proof1} and~\ref{sec:Main_Proof}, but simpler since the dynamics is linear:

%\serge{J'ai enlevé ``torsion points''. Pour deux raisons. D'abord c'est faux, on peut toujours tout conjugué par une translation par un point qui n'est pas de torsion. Ensuite (et surtout), avec la preuve actuelle on ne montre plus cela. En fait, on montre que les projections $\ell_g(a_i)$ sont dans $R(g)$, mais maintenant $R(g)$ est plus gros que $\Q/\Z$ (à cause de $s_2$ dans le lemme A3.

%En fait je pense qu'on peut être plus précis. Ou bien tous les $S_i$ sont égaux, ou bien $\Gamma$ contient un sous-groupe d'indice fini qui, à conjugaison près par une translation, fixe l'origine; et après une telle conjugaison les $a_i$ sont de torsion. Je ne crois pas que ce soit utile de le rédiger, c'est dans le même esprit que la toute dernière sous-section.}%torsion
%\romain{OK}

\begin{pro}\label{pro:main_toris} Under 
the above hypotheses {\em{(i)}}, {\em{(ii)}}, {\em{(iii)}}, if $\mu$ is a $\Gamma$-invariant and ergodic measure, then 
either $\mu$ is the Haar measure on the abelian surface $X$, or  there are finitely 
many  subtori $S_j\subset X$ of real dimension $2$, and 
points $a_j\in X$, $j=1, \ldots, k$, such that 
\begin{enumerate}[\em (1)]
\item $\bigcup_j (a_j+S_j)$ is $\Gamma$-invariant;
\item $\Gamma$ permutes transitively the subsets $a_j+S_j$, $j=1, \ldots, k$;
\item $\mu$ is supported on $\bigcup_j (a_j+S_j)$ and on each $a_j+S_j$, $\mu$ is given by $\frac{1}{k} m_j$
where $m_j$ is the Haar measure on $a_j+S_j$.
\end{enumerate}
\end{pro}

Here,  what we call Haar measure on  $a_j+S_j$ is the image of the Haar measure on $S_j$ by the translation 
$s \in S_j \mapsto a_j+s$.
With the results of Sections~\ref{par:Main Proof1}, 
this proposition concludes the proof of Theorem~\ref{thm:main}.

\subsection{Parabolic, affine transformations}\label{par:linear_translation_tori} The group $\Gamma$ acts on $X$ by affine transformations
\begin{equation}
f(x,y)=A_f(x,y)+S_f \mod (\Lambda )
\end{equation}
where the linear part $A_f\in \GL_2(\C)$ preserves the lattice $\Lambda\subset \C^2$ and the translation part $S_f$ is an 
element of $\C^2/\Lambda$. Now, pick a parabolic element $g\in \Gamma$; its linear part is given by
\begin{equation}\label{eq:linear_part_of_g}
A_g=\left(\begin{array}{cc} 1 & 0 \\ 1 & 1\end{array}\right)
\end{equation}
after a linear change of coordinates in $\C^2$. In these coordinates,  the fibration $\pi_g$ is induced by the projection $\pi_1\colon (x,y)\mapsto x$, and 
a conjugation by a translation reduces $g$  to the form  
\begin{equation}
g(x,y)=(x+s,y+x) \mod    (\Lambda )
\end{equation}
where $s$ has infinite order in the elliptic curve $B_g=\C/\pi_1(\Lambda)$. 

\begin{lem}
If the orbits of $g_B\colon x\mapsto x+s$ are dense in $B_g$, then $g$ is uniquely ergodic: the unique $g$-invariant probability measure on $X$
is the Haar measure. 
\end{lem}

This result is due to  Furstenberg (see~\cite[\S 3.3]{Furstenberg:1981}). Thus, in this case $\mu$ is the Haar measure on $X$ and we are done. So in what follows, we assume that 
for every $g\in \Hal(\Gamma)$ the orbits of the  translation $g_B$ are not dense: 
they equidistribute along circles $x+T_g$, where $T_g$ is the closure of the group 
$\Z s\subset B_g$; changing $g$ into a positive iterate
we may assume that this closure is isomorphic to $\R/\Z$ (as a real Lie group). 
We let $\ell_g$ be the quotient map $\C^2/\Lambda\to B_g/T_g$. 

\begin{lem}\label{lem:real_projection}
Every fiber of the linear projection $\ell_g$ is a 3-dimensional $g$-invariant torus,  
 and $g$ is uniquely ergodic on almost every fiber. 
\end{lem}

To prove Lemma~\ref{lem:real_projection}, 
we think of $\C^2$ as a real vector space and fix a
basis of $\Lambda$. Then  $\C^2$ is identified with $\R^4$ and $\Lambda$ with $\Z^4\subset \R^4$. The eigenspace of $A_g$ for the eigenvalue $1$
is defined over $\Z$ with respect to $\Lambda=\Z^4$. Moreover, $A_g$ acts trivially on the quotient space $\R^4/{\mathrm{Fix}}(A_g)$. Thus adapting the basis of $\Z^4$ to $g$, we may assume that 
\begin{equation}
g(x_1,x_2,x_3,x_4)=(x_1, x_2+s_2, x_3+ax_1+bx_2, x_4+cx_1+dx_2),
\end{equation} 
for some irrational number $s_2$  and some integers $a$, $b$, $c$, and $d$. The linear projection $\ell_g$ is now given by $(x_1,x_2,x_3,x_4)\mapsto x_1$ and Lemma~\ref{lem:real_projection} boils down to the following statement.

\begin{lem}\label{lem:unique_ergodicity_appendix}
If $1$, $s_1$ and $s_2$ are linearly independent over $\Q$, then $g$ is uniquely ergodic on the level set 
$\set{x_1 = s_1}$. 
\end{lem}

\begin{proof} Let us first observe that $ad-bc\neq 0$. Indeed, otherwise 
 the linear
part of $g$ would have a fixed point set of dimension $3$, which  is impossible because $g$ is holomorphic (see Equation~\eqref{eq:linear_part_of_g}). To prove unique ergodicity, we use the following criterion due to Furstenberg (see~\cite[Prop. 3.10]{Furstenberg:1981}):  
let $h$ be a  homeomorphism  on $\R/\Z\times (\R^2/\Z^2)$ of the form $(x,y)\mapsto (x+u, y+\varphi(x))$, where $u$ is irrational, then $h$  is uniquely ergodic if and only if  it is ergodic  for the Haar measure. On the fiber 
$x_1 = s_1$, our map $g$ is of the form 
\begin{equation}
(x_2, x_3, x_4)\longmapsto (x_2+s_2, x_3+as_1+bx_2, x_4+cs_1+dx_2),
\end{equation}
so we need to check that it is ergodic for the Haar measure. For this, we pick a measurable invariant subset $A\subset \R^3/\Z^3$, 
we denote by $\mathbf 1_A\in L^2(\R^3/\Z^3)$ its indicator function, and we expand it into a Fourier series $\mathbf 1_A(x_2,x_3,x_4) = \sum_{(k,\ell, m)\in \Z^3} c_{k,\ell, m} e^{2\ii \pi (kx_2+\ell x_3+ mx_4)}$. Then 
\begin{equation}
\mathbf 1_A\circ g (x_2,x_3,x_4) = \sum_{(k,\ell, m)\in \Z^3} c_{k,\ell, m} e^{2\ii \pi ks_2} e^{2\ii \pi (\ell a+ mc)s_1} e^{2\ii \pi (k+\ell b+ m d)x_2 }e^{2\ii \pi \ell x_3} e^{2\ii \pi m x_4}
 \end{equation}
and, from the $g$-invariance of $\mathbf 1_A$ and the uniqueness of the expansion, we get
\begin{equation}\label{eq:cklm}
c_{k,\ell, m} = e^{2\ii \pi ks_2} e^{2\ii \pi (\ell a+ mc)s_1} c_{k- \ell b - md, \ell, m}
 \end{equation}
 for all $(k,\ell, m)\in\Z^3$. For $\ell = m  =0$, the irrationality of $s_2$ implies that 
 $c_{k, 0, 0} = 0$ unless $k=0$. If $\ell b + md \neq 0$, iterating  the relation 
 $\abs{c_{k,\ell, m}}  = \abs{c_{k- \ell b - md, \ell, m}}$ and using the fact that  Fourier coefficients decay to zero at infinity, we infer that  $c_{k,\ell, m} = 0$. Finally, if $\ell b + md= 0$ and one of $\ell$ or $m$ is nonzero, since 
 $ad-bc\neq 0$ we get that $\ell a+ mc\neq 0$ and 
 \eqref{eq:cklm} gives $c_{k,\ell, m} = e^{2\ii \pi (ks_2 + (\ell a+ mc)s_1)} c_{k, \ell, m}$. Since $1$, $s_1$, and $s_2$ are $\Q$-linearly independent, we derive  $c_{k,\ell, m} =0$. Thus, $\mathbf 1_A$ is a constant, which means that the Haar measure of $A$ is $0$ or $1$. 
 \end{proof}

\subsection{Using distinct parabolic automorphisms} To complete the proof of Theorem~\ref{thm:main} in the case of tori, one can now follow the same ideas as in Sections~\ref{par:Main Proof1} and~\ref{sec:Main_Proof}.
%when there are Halphen twists $h\in \Gamma$ acting trivially on the base of their invariant fibration $\pi_h$. 
We only need to replace the dimension $\dim_\R(\mu)$ by the minimal 
dimension of a real subtorus $Z\subset X$ such that $\mu(q+Z)>0$ for some $q\in X$,
the invariant fibration $\pi_g$ by the $\R$-linear projection $\ell_g$,  and the set $\mathrm R_g(B_g^\circ)\subset B_g^\circ$ 
by  %\serge{Nouvelle définition}
\begin{equation}
\mathrm R(g)=\{ y \in B_g/T_g\; ; \; g \text{ is not uniquely ergodic in } \ell_g^{-1}(y)\} \subset B_g/T_g.
\end{equation}  
Lemma~\ref{lem:unique_ergodicity_appendix} shows that $\mathrm R(g)$ is countable.

\begin{lem}
If there is a parabolic element $g\in \Gamma$ for which $((\ell_g)_*\mu)(\mathrm R(g))<1$, then $\mu$ 
is the Haar measure on $X$.   
\end{lem}

\begin{proof} The proof is the same as for 
 Proposition~\ref{pro:smooth_case}. Pick another parabolic transformation $h\in \Gamma$, 
such that $\ell_g$ and $\ell_h$ are linearly independent; such an $h$ exist because $\Gamma$ is non-elementary. The main tool is the disintegration of $\mu$ with respect to 
$\ell_g$; for $y$ in a subset ${\mathcal{Y}}_g\subset B_g/T_g$ of positive measure, the conditional measure $\lambda_{g,y}$ 
is the Haar measure on the $3$-dimensional torus $\ell_g^{-1}(y)$.  Hence $\dim_\R(\mu)\geq 3$ and 
$((\ell_g)_*\mu)(\mathrm R(g))=0$, as in Step 2 of the proof of Proposition~\ref{pro:smooth_case}. As in Steps 3 and 4, we infer that 
 $(\ell_h)_*\mu$ does not charge $\mathrm R(h)$
and  that $(\ell_h)_*\mu$  is absolutely continuous with respect to the Lebesgue measure on $\R/\Z$. This, implies that 
$\mu$ itself is invariant by  \emph{all} translations  along the fibers of $\ell_h$, because $\mu=\int_Y \lambda_{h,y} d((\ell_h)_*\mu)(y)$
and $\lambda_{h,y}$ is the Haar measure for almost every $y$. Permuting the roles of $g$ and $h$, $\mu$ is in fact invariant under 
all translations. Hence, $\mu$ is the Haar measure on $X$. 
\end{proof}

Now, we we are reduced to the case where 
$(\ell_g)_*\mu(R(g))=1$ for every parabolic automorphism $g$ in $\Gamma$. 
Since $R(g)$ is countable, $d_\R(\mu)\leq 3$ and $\mu$ charges some fiber 
 of $\ell_g$. Using another parabolic automorphism
$h$, we see that $\mu$ gives positive mass to a translate $a_0+S_0$ of a $2$-dimensional torus $S_0\subset X$ whose projections $\ell_g(a_0+S_0)=\ell_g(a_0)$ and 
$\ell_h(a_0)$ are in the countable sets $R(g) \subset B_g/T_g$ and $R(h)\subset B_h/T_h$ respectively. Thus, by ergodicity, we conclude that $\mu$ is supported 
on a finite union of translates of $2$-dimensional tori $a_j+S_j\subset X$, $0\leq i\leq k-1$ for some $k\geq 1$.

A subgroup $\Gamma_0$ of index $\leq k!$ in $\Gamma$ preserves $a_0+S_0$, and   $g^{k!}$ and $h^{k!}$ act on $a_0+ S_0\simeq \R^2/\Z^2$ 
as two linear parabolic transformations with respect to transverse linear fibrations. So, it follows that $\mu_{\vert a_0+S_0}$ 
is proportional to the Haar measure of $a_0+S_0$,  and 
the proof of Proposition~\ref{pro:main_toris} is complete.
%The projection  $\ell_f(a_j+S_j)$ is rational for every parabolic automorphism of $\Gamma$: this implies that the $a_j$ are torsion points of $X/S_j$;
%so, translating by an element of $S_j$,  we may assume that the $a_j$ are torsion points of $X$.

\subsection{No  or infinitely many invariant real tori} \label{subs:infinitely_many_measures}
Consider a compact complex torus $X=\C^2/\Lambda$ of dimension $2$. Let $\Gamma$ be 
a subgroup of $\Aut(X)$. As in \S~\ref{par:linear_translation_tori}, write the elements $f$ of $\Aut(X)$ in the form $f(x,y)=A_f(x,y)+S_f$, 
and denote by $A_\Gamma\subset \GL_2(\C)$ the image of $\Gamma$ by the homomorphism $f\mapsto A_f$. The group 
$\Gamma$ is non-elementary if and only if $A_\Gamma$ contains a free group, if and only if the Zariski closure of 
$A_\Gamma$ in the {\emph{real algebraic group}} $\GL_2(\C)$ is semi-simple. 

Now, assume that $\Gamma$ is non-elementary and preserves at least one ergodic probability measure $\mu$ with $\dim_\R(\mu)=2$.
Equivalently, after conjugation by a translation, there is a finite index subgroup $\Gamma_0\subset \Gamma$ that preserves a real, two-dimensional 
subtorus $\Sigma=\Pi/\Lambda_{\Pi}$, where $\Pi\subset \C^2$ is a real vector space of dimension $2$ and $\Lambda_\R:=\Pi\cap \Lambda$ is a lattice
in $\Pi$ (the restriction of $\mu$ to $\Pi/\Lambda_{\Pi}$ is proportional to the Haar measure). The goal of this last section is to explain
that, in fact, {\emph{$\Gamma$ preserves infinitely many ergodic measures $\mu_j$ with $\dim_\R(\mu_j)=2$.}} %There are  two reasons for this phenomenon. 
Two mechanisms can be used to establish this fact.  

The first one relies on the fact that $\Gamma_0$ acts on the quotient $Q=X/\Sigma\simeq \R^2/\Z^2$, fixing the origin. Moreover, the action of $\Gamma_0$ on 
$Q=\R^2/\Z^2$ is induced by an {\emph{injective homomorphism}} $\Gamma_0\to \GL_2(\Z)$ (to see this, note that $\C^2=\Pi\oplus_\R\ii \Pi$ and $\ii\Pi$ surjects onto $Q$). This implies that $\Gamma_0$ has arbitrarily large finite orbits in $Q$ (coming from torsion points of $Q$). The preimages of these orbits in $X$ provide
surfaces $\Sigma_j\subset X$;  they are ``parallel'' to $\Sigma$ and have an arbitrarily large number of connected components; they are $\Gamma_0$-invariant; and each of them supports a unique 
invariant, ergodic, probability measure $\mu_j$ with $\dim_\R(\mu_j)=2$.  

%\serge{J'avais oublié de préciser que les plans invariants se projetaient bien sur des tores compacts. En rédigeant la preuve, je me suis aperçu qu'il manquait un petit truc (i.e. que j'avais oublié une ``pente'' $\alpha/\beta$.}
For the second mechanism, we assume that $\Gamma_0$ fixes the origin and, changing $\Gamma_0$ in a finite index subgroup if necessary, we identify 
$\Gamma_0$ with a subgroup of $\SL_2(\C)$. Identify $(\Pi, \Lambda_\Pi)$ to $(\R^2 ,\Z^2)$ and the restriction  $\Gamma_{0}\rest{\Pi}$ to a subgroup of $\GL_2(\Z)$; since $\Gamma_0$ is non-elementary, $\Gamma_0$ is Zariski dense in $\SL_2(\C)$ and $\Gamma_{0}\rest{\Pi}$ is Zariski dense in $\SL_2(\R)$
(resp. in $\SL_2(\C)$). In particular, the $\Q$-algebra generated by $\Gamma_{0}\rest{\Pi}$ is the algebra of $2\times 2$ matrices with rational coefficients. 
The decomposition $\C^2=\Pi\oplus_\R \ii \Pi$ is $\Gamma_0$-invariant, and the multiplication by $\ii$ defines a $\Gamma_0$-equivariant map from $\Pi$ to $\ii\Pi$. Thus, $\Gamma_0$ preserves 
each of the real planes $\Pi_\eta=\{(x,y)+\eta \ii (x,y)\; ; \; {\text{for }} \; (x,y)\in \Pi\}$, with $\eta\in \R$. Now, consider the (real) projection $q$ of $\C^2$ 
onto $\Pi$ parallel to $\ii\Pi$, and set $\Lambda'=q(\Lambda)$. 
It is a $\Gamma_0$-invariant subgroup of $\Pi$ of rank at most $4$, and it contains $\Lambda_\Pi\simeq \Z^2$.  
Then, one checks easily that 
\begin{enumerate}[(1)]
\item $\Lambda'$ is commensurable  to $\Lambda_\Pi\oplus \alpha \Lambda_\Pi$, for some  $\alpha\in \R\setminus \Q$, or to $\Lambda_\Pi$, in which case we set $\alpha=0$;
\item $\Lambda$ is commensurable to $\Lambda_\Pi\oplus K_{\alpha,\beta}(\Lambda_\Pi)$ where $K_{\alpha,\beta}$ is the linear map from 
$\Pi$ to $\Pi\oplus \ii \Pi$ defined by $K_{\alpha, \beta}(u)=\alpha u+\beta\ii u$;
\item for $m$ in $\Z$, the real plane $\Pi_{m\alpha/\beta}$ is  $\Gamma$-invariant and intersects $\Lambda$ on  a cocompact 
lattice $\Lambda_{\Pi_{m\alpha/\beta}}$.
\end{enumerate}
Then, the surfaces $\Sigma_m=\Pi_{m\alpha/\beta}/\Lambda_{\Pi_{m\alpha/\beta}}$ form an infinite family of $\Gamma$-invariant tori in $X$.

\begin{rem}
 This second argument does not apply in the following case. 
Let $E=\C/\Z[\ii]$, $\Lambda= \Z[\ii]\times \Z[\ii]\subset \C^2$, and 
$X=\C^2/\Lambda= E\times E$. The group $\Gamma=\SL_2(\Z)\ltimes \R^2/\Z^2$ is a 
subgroup of $\Aut(X)$ that preserves the torus 
$\Pi/\Lambda_\Pi$ for $\Pi=\R^2\subset \C^2$, but has no fixed point (because
$\Gamma$ contains $\Pi/\Lambda_\Pi$), and every $\Gamma$-invariant surface 
is a finite union of translates of this torus .

On the other hand, this second argument applies when $X$ and $\Gamma$ come from a genuine Kummer example, that is,   
 a Kummer example 
defined on a surface that is not a compact torus. Indeed in that case $\Gamma$ contains a finite index subgroup with a fixed point. \end{rem}
}}
 %\newpage
%
%%%%%%%%%%%%%%%%%%%%%%%%%%%%%%%%%%%%%%%%%%%%%%%%%%%%%%%%%%%%%%%%%%
%

\bibliographystyle{acm}
\bibliography{biblio-serge}

\end{document}